\documentclass{article}
\usepackage{amsmath,amsthm,amsfonts,graphicx}
\usepackage{multirow}
\usepackage{slashbox}
\usepackage{multirow}
\def\smallddots{\mathinner{\raise7pt\hbox{.}\raise4pt\hbox{.}\raise1pt\hbox{.}}}
\def\smallsdots{\mathinner{\raise1pt\hbox{.}\raise4pt\hbox{.}\raise7pt\hbox{.}}}
\DeclareMathOperator{\adj}{adj}

\DeclareMathOperator{\diag}{diag}

\DeclareMathOperator{\rank}{rank}
\DeclareMathOperator{\nrank}{nrank}

\DeclareMathOperator{\nul}{nul}
\DeclareMathOperator{\nnul}{nnul}

\DeclareMathOperator{\nmb}{nmb}

\numberwithin{equation}{section}
\numberwithin{table}{section}
\newtheorem{theorem}{Theorem}[section]
\newtheorem{lemma}{Lemma}[section]

\newtheorem{corollary}{Corollary}[section]
\newtheorem{fact}{Fact}[section]

\newtheorem{protoalgorithm}{Proto-Algorithm}[section]

\newtheorem{definition}{Definition}[section]

\newtheorem{remark}{Remark}[section]

\begin{document}

\title{\bf Randomized Matrix Computations
\thanks {Some results of this paper have been presented at the 
 ACM-SIGSAM International 
Symposium on Symbolic and Algebraic Computation (ISSAC '2011), San Jose, CA, 2011,
the
3nd International Conference on Matrix Methods in Mathematics and 
Applications (MMMA 2011) in
Moscow, Russia, June 22-25, 2011, 
the 7th International Congress on Industrial and Applied Mathematics 
(ICIAM 2011), in Vancouver, British Columbia, Canada, July 18-22, 2011,
the SIAM International Conference on Linear Algebra,
in Valencia, Spain, June 18-22, 2012, and 
the Conference on Structured Linear and Multilinear Algebra Problems (SLA2012),
in  Leuven, Belgium, September 10-14, 2012}}

\author{Victor Y. Pan$^{[1, 2],[a]}$, Guoliang Qian$^{[2],[b]}$, and Ai-Long Zheng$^{[2],[c]}$
\and\\
$^{[1]}$ Department of Mathematics and Computer Science \\
Lehman College of the City University of New York \\
Bronx, NY 10468 USA \\
$^{[2]}$ Ph.D. Programs in Mathematics  and Computer Science \\
The Graduate Center of the City University of New York \\
New York, NY 10036 USA \\
$^{[a]}$ victor.pan@lehman.cuny.edu \\
http://comet.lehman.cuny.edu/vpan/  \\
$^{[b]}$ gqian@gc.cuny.edu \\
$^{[c]}$ azheng-1999@yahoo.com \\
} 
 \date{}

\maketitle


\begin{abstract}
Random matrices tend to be well 
conditioned, and we employ this well known property to 
advance 
 matrix 
computations. We
prove 
that our algorithms 
employing Gaussian random 
matrices
are efficient, but in our tests 
the algorithms 
have consistently
remained 
as powerful 
where we 
used sparse
and 
structured random matrices, 
defined by much fewer random parameters.
We numerically stabilize Gaussian 
elimination with no pivoting as
well as block
Gaussian 
elimination, 
precondition an
 ill 
conditioned linear system of equations,
compute numerical rank of a matrix
without pivoting and orthogonalization,
approximate the singular spaces of an ill conditioned 
matrix
 associated with 
its largest and smallest singular values,
and approximate this matrix 
with 
low-rank matrices,
with applications to its $2\times 2$
block triangulation and to
 tensor decomposition.
Some of our results and techniques can be of independent interest,
e.g.,  our estimates for the condition numbers of random Toeplitz 
and circulant matrices and our
variations of the Sher\-man--Mor\-rison--Wood\-bury formula.
\end{abstract}


 


\section{Introduction}\label{sintro}


It is well known that random matrices tend to be well conditioned
 \cite{D88},  \cite{E88},  \cite{ES05},
\cite{CD05}, \cite{SST06}, \cite{B11}, and 
we employ 
this property to advance  
matrix computations. 
We 
prove that with probability $1$ or near $1$
our
 techniques of randomized preprocessing
 precondition a large and
important class of ill conditioned matrices.
By employing randomization we stabilize
 numerically Gaussian 
elimination with no pivoting as
well as  block
Gaussian 
elimination, precondition an
 ill 
conditioned linear system of equations,
compute numerical rank of a matrix
using no pivoting and orthogonalization,  
approximate the singular spaces
of an ill conditioned matrix $A$
 associated with 
its largest and smallest singular values, 
approximate this matrix by low-rank matrices,
compute its $2\times 2$ block triangulation, 
and compute a Tensor Train 
 approximation of a tensor. 
Our analysis and experiments 
show substantial progress versus 
 the known algorithms.
In our tests most of our 
techniques have fully preserved their 
 power 
when we dramatically decreased the
number of random parameters involved,
which
should motivate
further
research.
Some of our  
results
and 
techniques 
can be of independent interest,
e.g., 
 our estimates for the condition numbers of random Toeplitz 
and circulant matrices
and 
our extensions of the Sher\-man--Mor\-rison--Wood\-bury formula.


\subsection{Numerically safe Gaussian elimination with no pivoting}\label{sgenp}


Hereafter ``flop" stands for ``arithmetic operation",
``expected" and ``likely"
mean ``with probability $1$ or close to $1$",
$\sigma_j(A)$ denotes the $j$th largest
singular value
of an $n\times n$ matrix $A$, 
and the ratio
 $\kappa (A)=\sigma_1(A)/\sigma_{\rho}(A)$
for $\rho=\rank (A)$
 denotes its condition number. 
$\kappa (A)=||A||~||A^{-1}||$ if 
$\rho=n$, that  is if $A$ is a nonsingular matrix.
If this number is large in context,
then the matrix $A$
is {\em ill conditioned}, otherwise {\em well conditioned}.
For matrix inversion and solving linear systems of equations
the condition number represents the output magnification 
of 
input errors,
\begin{equation}\label{eqkappa}
\kappa(A)\approx \frac{||{\rm OUTPUT~ERROR}||}{||{\rm INPUT~ERROR}||},
\end{equation}
and backward error analysis implies similar 
magnification of rounding errors
\cite{GL96}, \cite{H02}, \cite{S98}.

To avoid dealing with singular or ill conditioned matrices in
Gaussian elimination, one incorporates pivoting, that is row or column interchange.
 {\em Gaussian elimination with no pivoting} (hereafter 
we refer to it as {\em GENP}) can easily fail in numerical 
computations with rounding errors, except for some 
special input classes such as the classes of diagonally 
dominant and positive definite matrices. For such matrices,
GENP outperforms
 Gaussian elimination with pivoting
\cite[page 119]{GL96}.
By extending our previous study in \cite[Section 12.2]{PGMQ} 
and \cite{PQZa},
we expand these classes dramatically by proving 
 that pre- and 
post-mul\-ti\-pli\-ca\-tion of a well conditioned coefficient
matrix of full rank
by a square Gaussian random matrix is expected to 
yield
 safe numerical performance of 
GENP as well as block Gaussian elimination
(see Remark \ref{regprec}). 
The results of our tests  
support this theoretical finding consistently.
Furthermore the tests show that
the power of our preprocessing is fully preserved 
where we use just 
  circulant or Householder multipliers  
 generated by a vector or a pair of vectors,
respectively, and where we
fill these  vectors with integers $\pm 1$ 
and
limit randomization to the choice
of the signs $\pm$
(see our Table \ref{tab44}
and \cite[Table 2]{PQZa}).


\subsection{Randomized  preconditioning: the basic theorem}\label{srdpr0}


Given an ill conditioned matrix $A$,
can we extend our progress by applying randomized  
multipliers $X$ and $Y$
to yield a much better conditioned matrix product $XAY$? 
No, because random square matrices $X$ and $Y$
are expected 
to be nonsingular and 
well conditioned \cite{D88},  \cite{E88},  \cite{ES05},
\cite{CD05}, \cite{SST06}, \cite{B11} and  because
$\kappa  (XAY)\ge \frac{\kappa (A)}{\kappa (X)\kappa (Y)}$.
Approximate inverses are popular multipliers
but only where it is not costly to compute them,
that is   
only for some  
special, although important classes of matrices $A$,
except for the  surprising empirical technique
in  \cite{R90} (see our Rematk \ref{rezmr}).


We can readily produce a well conditioned matrix $C$
by applying
additive preprocessing $A\Longrightarrow C=A+P$, e.g.,  
by choosing   
 $P=I-A$, but it is not clear
how this could help us to 
 solve a linear system $A{\bf y}={\bf b}$.
(Here and hereafter $I$ 
denotes the 
identity matrix.)
Assume, however, that we are given  
a nonsingular ill conditioned $n\times n$
matrix $A$ 
together with a small upper bound $r$ on
its numerical nullity $\nnul(A)$,
that is 
on the number of its singular values 
that are much smaller than the 2-norm $||A||_2$.
Such matrices 
make up a large and important subclass of 
nonsingular ill conditioned matrices
(cf. \cite{CDG03} and Remark \ref{renul}),
for which 
 randomized additive
preconditioning is supported by the following theorem.
In Section \ref{sapaug} 
we prove it,
 extend it to 
rectangular matrices $A$,
and further detail its 
estimates. 
\begin{theorem}\label{thkappa}
Suppose $A$ is a real $n\times n$ matrix 
having a numerical rank $\rho$ 
(that is the ratio $\sigma_{\rho+1}( A)/|| A||$ is
small, but the ratio $\sigma_{\rho}( A)/|| A||$ is not small),  
$\sigma U$ and $\sigma V$ are
standard Gaussian random $n\times r$ matrices,
whose  
all $2nr$ entries are independent of each other, 
$0<\rho<n$,
$0<r<n$, 
 and  $ C= A+UV^T$.
Then 
(i)  we can expect that
\begin{equation}\label{eqnca}
0.5 || A||_2\le|| C||_2\le 1.5 || A||_2~
{\rm if,~say}~
\sigma/|| A||_2\le \frac{1}{10}.
\end{equation}
(ii) Furthermore the matrix $ C$ is 
singular or ill conditioned if $r<n-\rho$, 
(iii) but otherwise it is
nonsingular with probability $1$
 and 
(iv) is expected to be 
 well conditioned
 if
the ratio $\sigma/|| A||_2$
is neither large nor small,
e.g., if
$\frac{1}{100}\le \sigma/|| A||_2\le 100$. 
\end{theorem} 
 



\subsection{Randomized algorithms}\label{srdpr}


We recall some known randomized matrix algorithms \cite{D83},
  \cite{HMT11} and study some new ones.
Suppose we are given
a normalized nonsingular $n\times n$ matrix $ A$, such that $|| A||=1$,
and its small positive numerical nullity $r=\nnul(A)$
(cf.  Section \ref{stlss1} on computing numerical nullity).
Suppose also that we have
generated  two standard Gaussian random $n\times r$ matrices 
$U$ and $V$  and applied 
{\em randomized additive preconditioning}
$ A\Longrightarrow  C= A+UV^T$
producing 
nonsingular and well conditioned
matrix $ C$,
as we can expect by virtue of Theorem \ref{thkappa}. Then we can 
apply 
the  Sher\-man--\-Mor\-rison--\-Wood\-bury 
formula,
hereafter referred to 
as the {\em SMW formula},
to
reduce
an ill conditioned linear system $ A{\bf y}={\bf b}$ of $n$ equations
to well conditioned 
linear systems $ C{\bf x}={\bf f}$.
 
By virtue of (\ref{eqkappa})
we must perform our
computations with high accuracy
to obtain meaningful output 
where the matrix $ A$ is ill conditioned,
but 
we can
limit 
 highly accurate computations 
with
multiple or infinite precision 
to $O(n^2)$
flops
and performs the order of $n^3$ remaining flops
with the standard double or single IEEE precision
versus order of $n^3$ high precision flops in Gaussian elimination.
It may be even more promising
to  combine 
randomized
additive 
 preprocessing and the SMW formula
to
compute  
multiplicative preconditioners 
of the matrix $ A$
(see Remark \ref{rezmr}).

We also apply our approach
numerically, with  
double precision, 
to the approximation of the singular spaces
associated with
the $\rho=n-r$ largest and 
the $r$ smallest singular values
of an ill conditioned matrix $ A$
that has a positive numerical nullity $r=\nnul(A)$.
We rely on the following 
 observations. Suppose the matrix $ C$ of 
 Theorem \ref{thkappa}
 is 
 nonsingular and well conditioned,
as expected. 
Then
(a) we can readily compute 
the $n\times r$ matrices 
$ C^{-1}U$ and
$ C^{-T}V$  by solving $2r$ linear systems of equations 
with the matrices $ C$ and $ C^T$ and
(b) the ranges of  the matrices
$ C^{-1}U$ and
$ C^{-T}V$ 
 approximate closely
the left and  right trailing singular spaces,
respectively, 
associated with the $r$ smallest singular values  
of an ill conditioned matrix $ A$. 
Furthermore we approximate  
the left and  right leading singular spaces 
associated with the $\rho$ largest singular values of 
the matrix  $ A$, both
by extending these techniques and by means 
of random 
sampling $A\Longrightarrow AH$ and
 $A\Longrightarrow A^TG$
for Gaussian random $n\times \rho_+$ matrices $G$ and $H$ 
and
for small positive integers $\rho_+-\rho$
\cite{HMT11}. 
We extend these randomized algorithms to 
compute numerical rank of a matrix
using no pivoting or orthogonalization,
its
$2\times 2$
block triangulation
 and its  low-rank approximation
with a further extension to tensor decomposition.


\subsection{Sparse and structured matrix computations,
randomized augmentation,
and numerical experiments}\label{scstr}



We cannot extend our proof of Theorem \ref{thkappa}
to the case of Gaussian random Toeplitz matrices $U$ and $V$,
but such extension has been supported consistently by  
the results of our numerical
tests. In these tests 
our algorithms remained as efficient where we replaced  
Gaussian random matrices by matrices 
defined by much fewer random parameters
such as the circulant and Householder multipliers of  Section \ref{sgenp}
and the block vectors $U$ with blocks $\pm I$ and O in
the additive preprocessing $A\Longrightarrow C=A+UU^T$
in Section \ref{sprec}. In these cases we
limited randomization to the choice
of the block sizes in the  block vectors
and of the signs $\pm$.
We have amended our algorithms 
to preserve matrix sparseness and structure,
in particular for computing 
numerical rank and nullity.

Additive preprocessing $A\Longrightarrow C=A+UV^T$
preserves 
matrix structure and sparseness
quite well
where the matrices
$U$ and $V$ have consistent 
structure and sparseness,  
but both random sampling
and randomized augmentation,
such as the maps $A\Longrightarrow A^TG$ and
$A\Longrightarrow K=
\begin{pmatrix}
W  &  V^T  \\
-U   &   A
\end{pmatrix}$, 
achieve this
even 
more perfectly
where 
we properly extend the patterns
of sparseness and structure
of the matrix $A$
to the
random matrices
$G$,
$U$, 
$V$, and $W$.
For $V=-U$ and symmetric 
positive definite matrices $A$ and $K$
the above augmentation 
increases the condition number
$\kappa (A)$, in
contrast to additive preprocessing
$A\Longrightarrow C=A+VV^T$
where $A$ is a nonnegative definite matrix. 
Other than that, however,
 we observe close similarity between
 augmentation 
and additive preprocessing,
we prove similar preconditioning properties
of both maps under randomization
and extend to
augmentation 
the SMW and dual SMW formulae 
as well as  Theorem \ref{thkappa} and
our expressions for the bases of leading 
and trailing singular spaces
(see Section \ref{saug}, 
Corollaries \ref{cotrail0} and \ref{cotrail01}, and  equations
(\ref{eqaugsmw}) and (\ref{eqhead+})).
Then again  
our tests  
were in good accordance with  
all these
extensions 
even in the case of sparse and structured matrices 
$G$, $U$, 
$V$, and $W$, defined by a
small number of random parameters,
although our
 formal proofs 
only apply
to the case of Gaussian random matrices
$G$, 
$U$, 
$V$, and $W$.

Other than that our  tests 
for randomized
multiplicative and additive preprocessing 
and augmentation
were 
in good accordance with our theoretical estimates.
By applying randomized additive preprocessing  
to ill conditioned matrices 
we consistently observed 
dramatic
decrease
of
 the condition
numbers 
(see Table \ref{tabprec}),
and this enabled accurate solution of the respective 
ill conditioned linear systems of equations
(see Tables \ref{tabSVD_TAIL}--\ref{tabSVD_HEAD1T}
and  \ref{tablsmwsmnl}), whereas 
 random circulant multipliers
filled with $\pm 1$ have
stabilized GENP
numerically 
(see Table \ref{tab44}).
Furthermore according to our test results,
our algorithms approximated accurately
the leading and
trailing  
singular spaces of ill conditioned matrices and 
 approximated a matrix by a low-rank matrix
(see Tables \ref{tablsnmb}--\ref{tablsmrm}). We have
also matched the output accuracy of the customary algorithms
 for solving
ill conditioned  Toeplitz linear systems of equations
but outperformed them dramatically in terms of the CPU time
(see Table \ref{tabhank}). Finally our experimental
data on the condition numbers of Gaussian random
Toeplitz and circulant matrices are
in good accordance with our formal estimates  
(see Sections \ref{scgrtm} and
\ref{scgrcm} and 
Tables \ref{nonsymtoeplitz}--\ref{tabcondcirc}).
These estimates respond to a
challenge in \cite{SST06}
and can be surprising
because the paper  \cite{BG05} has 
proved  that the condition numbers
of $n\times n$ 
Toeplitz matrices
 grow exponentially in $n$ as $n\rightarrow \infty$
for some large and
important classes of Toeplitz matrices,
whereas we prove the opposit for
Gaussian random Toeplitz matrices.


\subsection{Organization of the paper and selective reading}\label{sorg}


We recall some definitions and basic results on matrix computations
in the next 
section.
We  estimate  
the condition numbers of Gaussian random 
general, Toeplitz and  circulant  matrices
in Section \ref{srrm} and 
of
randomized  matrix products
in Section \ref{smrc}. The latter estimate 
imply that randomized multipliers support GENP
and block Gaussian elimination.
In Section \ref{sapaug} we
prove that our randomized additive
preprocessing of an ill conditioned matrix is expected to produce
a well conditioned matrix.
In Section \ref{saug} we prove a similar property of
 randomized augmentation, which we link to 
randomized additive
preprocessing  and apply to the solution of
 ill conditioned Toeplitz linear systems of equations.
 
In Sections \ref{sapsr}--\ref{svianv1} we 
apply randomization
to computations with ill conditioned matrices
having small
numerical  nullities or ranks.
We compute 
numerical rank 
of such matrix using no pivoting or orthogonalization,
approximate its trailing and leading singular spaces,
approximate it by a low-rank matrix,
and point out applications to tensor decomposition
and to approximation by structured matrices.
We also apply our
randomized additive preconditioning 
to compute $2\times 2$
block triangulation
of an ill conditioned matrix $ A$,
to compute its inverse, and to precondition a
linear system $ A{\bf y}={\bf b}$. We comment on
randomized computations with structured and sparse inputs
in Section \ref{sstr}. 

 In Section \ref{sexp} we cover numerical tests,
which constitute the contribution of the second and the third authors.
In Section \ref{srel} we comment on the related works, 
our progress, and some directions for further study.
In Appendix A we recall some results on
 structured matrices.
In Appendix B we estimate the probability that 
 a 
 random matrix 
has full rank
under the uniform probability distribution.
In Appendix C we comment on the extension of our probabilistic estimates
 to the case of complex matrices.

The paper can be partitioned into two parts:
  the next five sections, Section \ref{srel}
and the Appendix
cover basic theorems and make up Part 1, 
whereas 
Sections \ref{sapsr}--\ref{sexp},
on the algorithms and tests, make up Part 2.
The correctness proofs of the algorithms of
 Part 2 employ the
results of Part 1, but
otherwise Part 2 can be read independently of Part 1.
In {\em selective reading} one can skip 
the subjects of structured matrices and tensors
(in particular Sections \ref{stplc},  \ref{stplhl}, \ref{scgrtm},
\ref{scgrcm}, \ref{shank}, 
 \ref{sstr}, \ref{sexgeneral}, \ref{sexphank}, and Appendix A)
or augmentation
(Section  \ref{saug} and all related materials),
or  can read only selected algorithmic material, e.g., 
on low-rank approximation by means of random
sampling (Proto-Algorithms \ref{algbasmp} and \ref{algnrank},
Section \ref{stails} and 
the supporting results)
 or 
on preprocessing that supports GENP and block Gaussian elimination
(Theorems \ref{thnorms},
\ref{thsignorm} and \ref{1},
Remark \ref{regprec} 
and Section \ref{sexgeneral}).
In the paper, unlike its introduction, we study the general case of
rectangular input matrices $ A$,
but again  one may 
restrict oneself 
to the simpler special case
of square matrices. 


\section{Some definitions and basic results}\label{sdef}


We assume computations in the field $\mathbb R$ of real numbers,
and
 comment on the extension to the 
field 
 $\mathbb C$  of complex numbers in Appendix \ref{scomplin}.

Hereafter ``flop" stands for ``arithmetic operation"; 
``expected" and ``likely"
mean ``with probability $1$ or close to $1$", and
 the concepts ``large", ``small", ``near", ``closely approximate", 
``ill  conditioned" and ``well conditioned" are 
quantified in the context. 
For two scalars $a$ and $b$ we write $a\ll\ b$ and $b\gg a$ 
if the ratio $|b/a|$ is large. 
We write $a\approx b$ 
if $|a-b|\ll |a|+|b|$. 
Next we recall and extend some customary definitions of matrix computations
\cite{GL96}, \cite{S98}.


\subsection{Some basic definitions on matrix computations}\label{smat}

 

$\mathbb  R^{m\times n}$ is the class of real $m\times n$ matrices $A=(a_{i,j})_{i,j}^{m,n}$.

$(B_1~|~\dots~|~B_k)=(B_j)_{j=1}^k$ is a $1\times k$ block matrix with blocks $B_1,\dots,B_k$. 
$\diag (B_1,\dots,B_k)=\diag(B_j)_{j=1}^k$ is a $k\times k$ block diagonal matrix with diagonal blocks $B_1,\dots,B_k$.

${\bf e}_i$ is the $i$th coordinate vector of dimension $n$ for
$i=1,\dots,n$. These vectors define
the identity
matrix $I_n=({\bf e}_1~|~\dots~|~{\bf e}_n)$ and 
 the  reflection matrix $J_n=({\bf e}_n~|~\dots~|~{\bf e}_1)$,
both
of size $n\times n$. 
$O_{k,l}$ is the $k\times l$ matrix filled with zeros. ${\bf 0}_k$ 
is the vector $O_{k,1}$.
We write $I$, $J$, $O$, and ${\bf 0}$ where the size of a matrix or a vector
is not important or is defined by context. Furthermore we write
\begin{equation}\label{eqi}
I_{g,h}=I_g~{\rm where}~g\le h,~{\rm whereas}~I_{g,h}=(I_h~|~O_{h,g-h})
~{\rm where}~g>h.
\end{equation}

$A^T$ is the transpose  of a 
matrix $A$. $A^H$ is its Hermitian transpose. 
A matrix $A$ is symmetric if $A=A^T$ and is 
symmetric positive definite if
 $A=B^TB$  for a real nonsingular matrix $B$.

A real matrix $Q$ is called  
{\em orthogonal} if $Q^TQ=I$ 
 or $QQ^T=I$. More generally,  over the
 complex field  $\mathbb C$ a matrix $U$ is called
{\em unitary} if  $U^HU=I$ or  $UU^H=I$.
Hereafter $Q(A)$ denote a unique orthogonal matrix 
specified by the following result.

\begin{fact}\label{faqrf} \cite[Theorem 5.2.2]{GL96}.
QR factorization $A=QR$ of a matrix $A$ having full column rank
into the product of an orthogonal matrix $Q=Q(A)$ 
and an upper triangular matrix $R=R(A)$ is unique 
provided that the factor $R$ is a square matrix 
with positive diagonal entries. 
\end{fact}


\subsection{Range, null space, rank, nullity, nmbs, and
generic rank profile}\label{srnsn}


$\mathcal R(A)$  denotes the range of an $m\times n$ matrix $A$, that is the linear space $\{{\bf z}:~{\bf z}=A{\bf x}\}$
generated by its columns. $\mathcal N(A)$ denotes its null space $\{{\bf v}:~A{\bf v}={\bf 0}\}$,
$\rank (A)=\dim \mathcal R(A)$ its rank, and $\nul (A)=\dim \mathcal N(A)=n-\rank (A)$ its   
right nullity or just {\em nullity}, 
whereas $\nul (A^T)=m-\rank (A)$ is the left nullity of $A$, equal to $\nul (A)$
if and only if $m=n$.
${\bf v}$ is the null vector of $A$ if $A{\bf v}={\bf 0}$. 


\begin{fact}\label{far1}
The set $\mathbb M$ of $m\times n$ matrices of rank  $\rho$
is an algebraic variety of dimension  $(m+n-\rho)\rho$.
\end{fact}
\begin{proof}
Let $M$ be an $m\times n$ matrix of a rank $\rho$
with a nonsingular leading $\rho\times \rho$ block $M_{00}$
and write $M=\begin{pmatrix}
M_{00}   &   M_{01}   \\
M_{10}   &   M_{11}
\end{pmatrix}$.
Then the $(m- \rho)\times (n- \rho)$ {\em Schur complement} $M_{11}-M_{10}M_{00}^{-1}M_{01}$
must vanish, which imposes $(m-\rho)(n-\rho)$ algebraic equations on the entries of $M$. 
Similar argument can be applied  where any $\rho\times \rho$
submatrix of the matrix $M$ 
(among $\begin{pmatrix}
m      \\
\rho
\end{pmatrix}\begin{pmatrix}
n     \\
\rho
\end{pmatrix}$ such submatrices)
is nonsingular. Therefore 
$\dim \mathbb M=mn-(m-\rho)(n-\rho)=(m+n-\rho)\rho$.
\end{proof}  


A matrix $B$ 
is 
a {\em matrix cover}
for its range $\mathcal R(B)$.
A matrix cover is a
{\em matrix basis} (for its range)
if it
has full column rank.
A matrix basis $B$ for
the null 
space
 $\mathcal N(A)$
is a {\em null matrix basis} or a {\em nmb} 
for the matrix $A$, and we write $B=\nmb(A)$. 
$\mathcal N(A^T)$ is the left null space of a matrix $A$, and similarly
the map $A\Longrightarrow A^T$ defines   
 left null vectors,  
left nmbs, and the left nullity
 of a matrix $A$.
$A_k^{(k)}$ denotes the leading, 
that is northwestern  $k\times k$ block 
submatrix of a matrix $A$.
A matrix of a rank $\rho$ has {\em generic rank profile}
if all its leading $i\times i$ blocks are nonsingular for $i=1,\dots,\rho$.
If such  matrix is nonsingular itself, then  it is called {\em strongly nonsingular}. 




\subsection{Norms, SVD, and singular spaces}\label{sosvdi}


$||A||_h$ is the $h$-norm and
$||A||_F=\sqrt{\sum_{i,j=1}^{m,n}|a_{i,j}|^2}$ is the Frobenius norm
of a matrix $A=(a_{i,j})_{i,j=1}^{m,n}$.
We write $||A||=||A||_2$ and $||{\bf v}||=\sqrt {{\bf v}^T{\bf v}}=||{\bf v}||_2$  
and recall from \cite[Section 2.3.2 and Corollary 2.3.2]{GL96} that
$$
{\rm max}_{i,j=1}^{m,n}|a_{i,j}|\le ||A||=||A^T||\le \sqrt {mn}~{\rm max}_{i,j=1}^{m,n}|a_{i,j}|,
$$
\begin{equation}\label{eqnorm12}
\frac{1}{\sqrt m}||A||_1\le||A||\le \sqrt n ||A||_1,~~||A||_1=||A^T||_{\infty},~~
||A||^2\le||A||_1||A||_{\infty}, 
\end{equation}
\begin{equation}\label{eqfrob}
||A||\le||A||_F\le \sqrt n ~||A||, 
\end{equation}
\begin{equation}\label{eqnorm12inf}
||AB||_h\le ||A||_h||B||_h~{\rm for}~h=1,2,\infty~{\rm and~any~matrix~product}~AB.
\end{equation}
A matrix $A$  is 
 {\em normalized}  if $||A||=1$.
A normalized  vector is orthogonal (unitary), and we call it {\em unit}.
We write $A\approx B$ if $||A-B||\ll ||A||+||B||$.



Define an {\em SVD} or {\em full SVD} of an $m\times n$ matrix $A$ of a rank 
 $\rho$ as follows,
\begin{equation}\label{eqsvd}
A=S_A\Sigma_AT_A^T.
\end{equation}
Here
$S_AS_A^T=S_A^TS_A=I_m$, $T_AT_A^T=T_A^TT_A=I_n$,
$\Sigma_A=\diag(\widehat \Sigma_A,O_{m-\rho,n-\rho})$, 
$\widehat \Sigma_A=\diag(\sigma_j(A))_{j=1}^{\rho}$,
$\sigma_j=\sigma_j(A)=\sigma_j(A^T)$
is the $j$th largest singular value of a matrix $A$
 for $j=1,\dots,\rho$, and we write
$\sigma_j=0$ for $j>\rho$. These values have 
the minimax property  
\begin{equation}\label{eqminmax}
\sigma_j=\max_{{\rm dim} (\mathbb S)=j}~~\min_{{\bf x}\in \mathbb S,~||{\bf x}||=1}~~~||A{\bf x}||,~j=1,\dots,\rho,
\end{equation}
where $\mathbb S$ denotes linear spaces  \cite[Theorem 8.6.1]{GL96}.
Consequently 
 $\sigma_{\rho}>0$,  
 $\sigma_1=\max_{||{\bf x}||=1}||A{\bf x}||=||A||$,
and
\begin{equation}\label{eqnormdiag}
||\diag(M_j)_j||=\max_j ||M_j||~{\rm for~any~set~of~matrices}~M_j.
\end{equation}
 
\begin{fact}\label{faccondsub} 
If $A_0$ is a 
submatrix of a 
matrix $A$, 
then
$\sigma_{j} (A)\ge \sigma_{j} (A_0)$ for all $j$.
\end{fact} 

 
\begin{proof}
\cite[Corollary 8.6.3]{GL96} implies 
the claimed bound
where $A_0$ is any block of columns of 
the matrix $A$. Transposition of a matrix and permutations 
of its rows and columns do not change singular values,
and thus we can extend the bounds to
all submatrices $A_0$.
\end{proof}

\begin{theorem}\label{they} 
We have $|\sigma_{j} (C)-\sigma_{j} (C+E)|\le ||E||$
for all $m\times n$ matrices $C$ and $E$ 
and  all $j$.
\end{theorem}
\begin{proof}
See \cite[Corollary 8.6.2]{GL96} or \cite[Corollary 4.3.2]{S98}. 
\end{proof}

In Sections \ref{sapsr}--\ref{svianv1}
we use the following definitions.
For every integer $k$ in the range $1\le k<\rank(A)$ define the partition 
$S_A=(S_{k,A}~|~S_{A,m-k})$ and $T_A=(T_{k,A}~|~T_{A,n-k})$
where the submatrices $S_{k,A}$ and $T_{k,A}$ are formed by the
first $k$ columns of the matrices $S_A$ and $T_A$, respectively.
Write $\Sigma_{k,A}=\diag(\sigma_j(A))_{j=1}^k$,
$\mathbb S_{k,A}=\mathcal R(S_{k,A})$ and  $\mathbb T_{k,A}=\mathcal R(T_{k,A}$).
If $\sigma_k>\sigma_{k+1}$, 
then 
$\mathbb S_{k,A}$ and $\mathbb T_{k,A}$ are
the left and right {\em leading singular spaces}, respectively,
    associated with the $k$ largest singular values of the matrix $A$,
whereas their orthogonal complements $\mathbb S_{A,m-k}=\mathcal R(S_{A,m-k})$ 
and $\mathbb T_{A,n-k}=\mathcal R(T_{A,n-k})$ 
are the left and right {\em trailing singular spaces}, respectively,
associated with the other singular values of $A$. 
The pairs of subscripts $\{k,A\}$ versus $\{A,m-k\}$ and $\{A,n-k\}$ mark 
the leading versus trailing
singular spaces.     
The left singular spaces of $A$ are 
the right   singular spaces of $A^T$ and vice versa.
All matrix bases for the singular spaces $\mathbb S_{k,A}$ and $\mathbb T_{k,A}$
are given by matrices $S_{k,A}X$ and
$T_{k,A}Y$, respectively,
for  nonsingular $k\times k$ matrices $X$ and $Y$.
Orthogonal matrices $X$ and $Y$ define orthogonal matrix bases
for these spaces.  
$B$ is an {\em approximate matrix basis} for 
a space $\mathbb S$ within a relative error norm bound $\tau$
if there exists a matrix $E$ such that $B+E$ is a matrix basis for 
this space $\mathbb S$ and if $||E||\le \tau ||B||$.

\subsection{Inverses, generalized inverses, 
and perturbation bounds}\label{sigipb}


$A^+=T_A\diag(\widehat \Sigma_A^{-1},O_{n-\rho,m-\rho})S_A^T$ is the Moore--Penrose 
pseudo-inverse of the matrix $A$ of (\ref{eqsvd}), and
\begin{equation}\label{eqnrm+}
||A^+||=1/\sigma{_\rho}(A)
\end{equation}
 for 
a matrix $A$ of a rank $\rho$.  $A^{+T}$ stands for $(A^+)^T=(A^T)^+$,
and $A^{-T}$ stands for $(A^{-1})^T=(A^T)^{-1}$.

An $n\times m$  matrix $X=A^{(I)}$ is a left 
inverse of an $m\times n$ matrix $A$ if $XA=I$ and is its 
right inverse if $AX=I$. 
$A^+$ is a left or right inverse $A^{(I)}$ if and only if a matrix $A$ has full rank.
 $A^{(I)}$ is unique and is equal to $A^{-1}$ 
if $A$ is  a nonsingular matrix.
Theorem \ref{they} implies the following bound.


\begin{theorem}\label{thpert1}  
Suppose two matrices $C,C+E\in \mathbb C^{m\times n}$ have full rank.
Then $||(C+E)^+-C^+||\le ||E||~||(C+E)^+~C^+||.$
\end{theorem}


This bound can be improved where the matrices $C$ and $C+E$ are nonsingular.


\begin{theorem}\label{thpert} 
Suppose $C$ and $C+E$ are two nonsingular matrices of the same size
and $||C^{-1}E||  =\theta<1$. Then
$||I-(C+E)^{-1}C||  \le \frac{\theta}{1-\theta}$ and
$\||(C+E)^{-1}-C^{-1}||\le \frac{\theta}{1-\theta}||C^{-1}||$,
in particular $\||(C+E)^{-1}-C^{-1}||\le 0.5||C^{-1}||$
if $\theta\le 1/3$.
\end{theorem}
\begin{proof}
See \cite[Corollary 1.4.19]{S98} for $P= -C^{-1}E$.
\end{proof}

 
\subsection{SMW and dual SMW formulae}\label{ssmw}

 
\begin{theorem}\label{thdsmw}\cite[page 50]{GL96}, \cite[Corollary 4.3.2]{S98}.
Suppose 
that 
$U,V\in \mathbb R^{n\times r}$, the matrices
$A\in \mathbb R^{n\times n}$ and $C=A+UV^T$ are
nonsingular, and
$0<r<n$. Then the matrix $G=I_r-V^TC^{-1}U$
is nonsingular
and we have
the  
Sher\-man--\-Mor\-rison--\-Wood\-bury 
(hereafter {\em SMW}) formula
$$A^{-1}=C^{-1}+C^{-1}UG^{-1}V^TC^{-1}.$$
\end{theorem}
 
\begin{corollary}\label{codsmw}
Suppose 
that 
$U_-,V_-\in \mathbb R^{n\times q}$,
$A\in \mathbb R^{n\times n}$, $A$ and 
$A^{-1}+U_-V_-^T$
are nonsingular matrices, and
$0<q<n$. Write 
\begin{equation}\label{eqc--1}
C_-^{-1}=A^{-1}+U_-V_-^T,~H=I_q+V_-^TAU_-,
\end{equation}
Then the matrix $H$ is nonsingular
and 
the following {\em dual SMW formula} holds, 
\begin{equation}\label{eqsmwd}
C_-=A-AU_-H^{-1}V_-^TA.
\end{equation}
\end{corollary}
\begin{proof}
Apply Theorem \ref{thdsmw} to the matrices $A^{-1}$, $U_-$, $V_-$
and $C_-^{-1}$ replacing the matrices $A$, $U$, $V$
and $C$, respectively.
\end{proof}


\subsection{Condition number, numerical rank and  numerical nullity, generic conditioning profile
}\label{scnpn}


$\kappa (A)=\frac{\sigma_1(A)}{\sigma_{\rho}(A)}=||A||~||A^+||$ is the condition 
number of an $m\times n$ matrix $A$ of a rank $\rho$. Such matrix is {\em ill conditioned} 
if $\sigma_1(A)\gg\sigma_{\rho}(A)$ and is {\em well conditioned}
otherwise. See \cite{D83}, \cite[Sections 2.3.2, 2.3.3, 3.5.4, 12.5]{GL96}, 
\cite[Chapter 15]{H02}, \cite{KL94}, and \cite[Section 5.3]{S98} 
on the estimation of norms and condition numbers 
of nonsingular matrices. 

An $m\times n$ matrix $A$ has {\em numerical rank} $\nrank(A)$, not exceeding $\rank (A)$, 
and has 
the right numerical nullity $\nnul(A)=n-\nrank(A)$
or just {\em numerical nullity}
if the ratios $\sigma_{j}(A)/||A||$
are small for $j>\nrank(A)$ but not for $j\le \nrank(A)$. 
 The {\em left numerical nullity} of the matrix $A$ 
equals the numerical nullity $\nnul(A^T)=m-\nrank(A)$ of 
the $n\times m$ transpose $A^T$
and coincides with
the numerical nullity of $A$ 
if and only if $m=n$.


\begin{remark}\label{renul}
One can specify the adjective ``small" 
above as
``smaller than  a fixed positive tolerance".
The choice of the tolerance can be a challenge,
e.g., for the matrix $\diag(1.1^{-j})_{j=0}^{999}$.
\end{remark}

  
If a well conditioned  $m\times n$ matrix $A$ has a rank $\rho<l=\min\{m,n\}$, 
 then almost all its close neighbours have full rank $l$
(see Section \ref{sngrm}), and
all of them have numerical rank $\rho$.
Conversely, suppose a matrix $A$
has a positive
numerical rank $\rho=\nrank (A)$ and 
{\em truncate its SVD}  by
setting 
to $0$ all its 
singular values, except for the  $\rho$ largest  ones.
Then the resulting matrix  $A-E$ is well conditioned and
has rank $\rho$ and
 $||E||=\sigma_{\rho+1}(A)$,
and so $A-E$ is a rank-$\rho$ approximation to the matrix $A$
within the error norm bound $\sigma_{\rho+1}(A)$.
At a
lower computational cost we can obtain rank-$\rho$ approximations 
of  the matrix $A$ from its
  rank-revealing factorizations 
\cite{GE96}, \cite{HP92}, \cite{P00a}, 
and we further decrease the computational cost
by
applying randomized algorithms
in Sections \ref{sapsr} and \ref{sstr}. 




An $m\times n$ matrix 
has {\em generic conditioning profile}
(cf. the  end of Section \ref{srnsn})
if it
has a numerical rank $\rho$ and if
its leading $i\times i$ blocks are nonsingular and well conditioned for $i=1,\dots,\rho$.
If such matrix has full rank (that is if $\rho=\min\{m,n\}$) 
and if it is well conditioned  itself, 
 then we call it {\em strongly well conditioned}.
The following theorem 
shows that
GENP 
and block Gaussian elimination
applied to a strongly 
well conditioned matrix
are numerically safe.  


\begin{theorem}\label{thnorms} Cf.  \cite[Theorem 5.1]{PQZa}.
Assume GENP
or block Gaussian elimination
applied to 
 an
$n\times n$
matrix $A$ and
write $N=||A||$ and $N_-=\max_{j=1}^n ||(A_j^{(j)})^{-1}||$. 
Then the absolute values of all pivot elements of GENP 
and the norms of all pivot blocks of 
block Gaussian elimination
do not exceed $N+N_-N^2$,
whereas the absolute values of the reciprocals of these 
elements and the norms of the inverses of the blocks do not
exceed $N_-$.
\end{theorem}


\subsection{Toeplitz, Hankel and $f$-circulant 
matrices}\label{stplc}


A {\em Toep\-litz} $m\times n$ matrix $T_{m,n}=(t_{i-j})_{i,j=1}^{m,n}$ 
is defined by its first row and column, that is by
the vector $(t_h)_{h=1-n}^{m-1}$
of dimension $m+n-1$. We write $T_n=T_{n,n}=(t_{i-j})_{i,j=1}^{n,n}$ (see (\ref{eqtz})).
 
A lower {\em triangular Toep\-litz}  $n\times n$ matrix $Z({\bf t})=(t_{i-j})_{i,j=1}^n$
(where $t_k=0$ for $k<0$)
is defined by its first column ${\bf t}=(t_h)_{h=0}^{n-1}$. 
We write $Z({\bf t})^T=(Z({\bf t}))^T$.
$Z=Z_0=Z({\bf e}_2)$
is the downshift $n\times n$
 (see (\ref{eqtz})). We have 
$Z{\bf v}=(v_i)_{i=0}^{n-1}$ and
$Z({\bf v})=Z_0({\bf v})=\sum_{i=1}^{n}v_{i}Z^{i-1}$
for  ${\bf v}=(v_i)_{i=1}^n$ and $v_0=0$,
\begin{equation}\label{eqtz}
T_n=\begin{pmatrix}t_0&t_{-1}&\cdots&t_{1-n}\\ t_1&t_0&\smallddots&\vdots\\ \vdots&\smallddots&\smallddots&t_{-1}\\ t_{n-1}&\cdots&t_1&t_0\end{pmatrix},~~Z=\begin{pmatrix}
        0   &       &   \dots    &   & 0\\
        1   & \ddots    &       &   & \\
        \vdots     & \ddots    & \ddots    &   & \vdots    \\
            &       & \ddots    & 0 &  \\
        0    &       &  \dots      & 1 & 0 
    \end{pmatrix}.
\end{equation}
Combine the equations $||Z({\bf v})||_1=||Z({\bf v})||_{\infty}=||{\bf v}||_1$
 with  (\ref{eqnorm12}) to obtain 
\begin{equation}\label{eqttn}
||Z({\bf v})||\le ||{\bf v}||_1.
\end{equation}



\begin{theorem}\label{thgs}
Write $T_{k}=(t_{i-j})_{i,j=0}^{k-1}$ for  $k=n,n+1$. 

(a) Let the matrix $T_n$ be nonsingular and write 
${\bf p}=T_n^{-1}{\bf e}_1$ and ${\bf q}=T_n^{-1}{\bf e}_{n}$.
If
$p_{1}={\bf e}_1^T{\bf p}\neq 0$,
then
$p_{1}T_n^{-1}=Z({\bf p})Z(J{\bf q})^T-Z(Z{\bf q})Z(ZJ{\bf p})^T.$

In parts (b) and (c) below let the matrix $T_{n+1}$ be nonsingular and write 
$\widehat {\bf v}=(v_i)_{i=0}^n=T_{n+1}^{-1}{\bf e}_1$,
${\bf v}=(v_i)_{i=0}^{n-1}$, ${\bf v}'=(v_i)_{i=1}^{n}$,
$\widehat {\bf w}=(w_i)_{i=0}^n=T_{n+1}^{-1}{\bf e}_{n+1}$, 
${\bf w}=(w_i)_{i=0}^{n-1}$, and ${\bf w}'=(w_i)_{i=1}^{n}$.

(b) If $v_0\neq 0$, then the matrix $T_n$ is nonsingular and
$v_0T_n^{-1}=Z({\bf v})Z(J{\bf w'})^T-Z({\bf w})Z(J{\bf v}')^T$.

(c) If $v_n\neq 0$, then the matrix $T_{1,0}=(t_{i-j})_{i=1,j=0}^{n,n-1}$ is nonsingular and
$v_nT_{1,0}^{-1}=Z({\bf w})Z(J{\bf v'})^T-Z({\bf v})Z(J{\bf w}')^T$.
\end{theorem}
\begin{proof}
 See \cite{GS72} on parts (a) and (b);  see \cite{GK72} on part (c).
\end{proof}

$Z_f=Z+f{\bf e}_1^T{\bf e}_n$ for a scalar $f\neq 0$
denotes the
$n\times n$ matrix of
 $f$-{\em circular shift}.
An $f$-{\em circulant matrix} $Z_f({\bf v})=\sum_{i=1}^{n}v_iZ_f^{i-1}$ 
is a special Toep\-litz $n\times n$ matrix defined by its first column vector 
${\bf v}=(v_i)_{i=1}^{n}$ and a scalar $f$. 
$f$-circulant matrix is called {\em circulant} if $f=1$ and {\em skew circulant} if $f=-1$.
By replacing $f$ with $0$ we arrive at a lower triangular 
Toep\-litz matrix $Z({\bf v})$.
The following theorem implies that the inverses 
 (wherever they are defined) and pairwise  products of  
$f$-circulant  $n\times n$ matrices are $f$-circulant  and can be computed 
in $O(n\log n)$ flops. 

\begin{theorem}\label{thcpw} (See \cite{CPW74}.)
We have 
$Z_1({\bf v})=\Omega^{-1}D(\Omega{\bf v})\Omega.$
More generally, for any $f\ne 0$, we have
$Z_{f^n}({\bf v})=U_f^{-1}D(U_f{\bf v})U_f$
where
$U_f=\Omega D({\bf f}),~~{\bf f}=(f^i)_{i=0}^{n-1}$,
$D({\bf u})=\diag(u_i)_{i=0}^{n-1}$ for a vector ${\bf u}=(u_i)_{i=0}^{n-1}$,  
$\Omega=(\omega_n^{ij})_{i,j=0}^{n-1}$ is the $n\times n$ matrix of the 
discrete Fourier transform at $n$ points, 
$\omega_n={\rm exp}(\frac{2\pi}{n}\sqrt{-1})$ being 
a primitive $n$-th root of $1$, and $\Omega^{-1}=\frac{1}{n}(\omega_n^{-ij})_{i,j=0}^{n-1}=\frac{1}{n}\Omega^H$.
\end{theorem}

 {\em Hankel} $m\times n$ matrices $H=(h_{i+j})_{i,j=1}^{m,n}$ can be 
 defined equivalently
as the products $H=TJ_n$ or $H=J_mT$ of
$m\times n$ Toep\-litz 
matrices $T$ and the Hankel reflection matrices $J=J_m$ or $J_n$.
Note that $J=J^{-1}=J^T$ and obtain the following simple fact.
\begin{fact}\label{fath}
For $m=n$ we have 
 $T=HJ$, $H^{-1}=JT^{-1}$ and $T^{-1}=JH^{-1}$ if $H=TJ$,
whereas $T=JH$, $H^{-1}=JT^{-1}$ and $T^{-1}=H^{-1}J$ if $H=JT$.
Furthermore in both cases $\kappa (H)=\kappa (T)$.
\end{fact}
By using the equations above we can 
readily extend any Toep\-litz  matrix inversion algorithm 
 to Hankel 
 matrix inversion
and vice versa, 
preserving the flop count and condition numbers.
E.g. $(JT)^{-1}=T^{-1}J$, $(TJ)^{-1}=JT^{-1}$,
$(JH)^{-1}=H^{-1}J$ and $(HJ)^{-1}=JH^{-1}$.


\subsection{Toeplitz-like, Han\-kel-like and some other structured
matrices}\label{stplhl}


Let us extend the class of
Toeplitz and Hankel matrices to a more general class of
structured
matrices, which we only employ  in Section \ref{sstr}.
With every pair of $n\times n$ {\em operator  
matrices} $A$ and $B$
   associate the class of $n\times n$ 
matrices $M$ for which the rank of the 
Sylvester {\em displacement} $AM-MB$
(called the {\em displacement rank} of $M$)
is small in context.
The matrices  $T$ with the structure of Toeplitz type (we call them
{\em Toeplitz-like} matrices) have small {\em displacement ranks} $d=d(A,B)$
for $A=Z_e$ and $B=Z_f$ and
for any pair of scalars $e$ and $f$. Such matrices 
extend the class of Toep\-litz matrices, for which $d\le 2$.
Any variation of a pair of scalars $e$ and $f$
can change 
the displacement rank of a matrix by at most $2$,
and so the class of Toeplitz-like
 matrices is independent of the choice of such pair.

Every matrix of a rank $d$, and in particular a displacement 
of a rank $d$,
can be nonuniquely represented as the sum of $d$
outer products 
${\bf g}_j{\bf h}_j^T$ of $d$ pairs 
of vectors ${\bf g}_j$ and ${\bf h}_j$ for 
$j=1,\dots,d$. 
Motivated by the  following result
we call the pair of matrices $G=G_{Z_e,Z_f}(M)=({\bf g}_j)_{j=1}^d$
and $H=H_{Z_e,Z_f}(M)=({\bf h}_j)_{j=1}^d$, made up of the vectors
${\bf g}_j$ and ${\bf h}_j$,
a {\em displacement generator} of length $d$
for the matrix $M$ and for the operator matrices $Z_e$
and $Z_f$ where $e\neq f$
(cf.  \cite[Example 4.4.2]{p01}).
\begin{theorem}\label{thtdsp} 
If $Z_eM-MZ_f=\sum_{j=1}^d{\bf g}_j{\bf h}_j^T$
for a pair of distinct scalars $e$ and $f$, then
\begin{equation}\label{eqtlk}
(e-f)M=\sum_{j=1}^dZ_e({\bf g}_j)Z_f(J{\bf h}_j).
\end{equation}
\end{theorem}
The theorem expresses the matrix 
through the displacement generator $\{G,H\}$ 
by using $2dn$ parameters
instead of $n^2$ entries. 
For $d\ll n$, this is a dramatic compression, which 
furthermore 
reduces  multiplication of the
matrix $M$ by a vector essentially to $2d$ multiplications of 
circulant matrices by vectors, that is to $O(dn\log n)$ 
flops. 
Moreover we can operate with matrices by using their
displacement representation, which preserves
Toep\-litz-like
structure and can
accelerate the computations 
dramatically 
where $d\ll n$.
For Toep\-litz-like matrices $T$, $T_1$
and $T_2$, scalars $e$, $f$, $\alpha$, and $\beta$, 
and operator matrices $A=Z_e$, $B=Z_f$, and $C=Z_c$, 
we can readily obtain the
 Toep\-litz-like matrices 
 $T^{-1}$ (if the matrix $T$ is nonsingular),
$T^T$,
$\alpha T_1+\beta T_2$, and
 $T_1T_2$. The following theorem
bounds  the  growth of the length of the associated
displacement generators and the respective flop cost. 

{\begin{theorem}\label{thops}
Assume that  $n\times n$ matrices $T$, $T_1$, and  $T_2$
 have been represented with their 
displacement generators of lengths $d_1$,
 $d_2$, and  $d$, respectively, for 
appropriate  operator matrices 
$A=Z_e$ and $B=Z_f$,
defining 
Toeplitz-like structure. 
Then there
exist displacement generators of length
$d$ for $T^{-1}$ 
(provided that the matrix $T$ is nonsingular)
and $T^T$, of length 
$d_1+d_2$ for $\alpha T_1+\beta T_2$,
 and of length 
$d_1+d_2+O(1)$ for $T_1T_2$ 
(for appropriate operator matrices 
defining 
Toeplitz-like structure
and for any  
pair of scalar
$\alpha$ and $\beta$).
One can compute these 
 generators by using $O(d^2n\log^2 n)$ 
and $O(d_1d_2n\log n)$ flops,
respectively. 
\end{theorem}
{\begin{proof}
The theorem readily follows from 
Theorem \ref{thdgs} and 
Corollary \ref{codgdr} in Appendix \ref{apstr},
which also 
 define all the respective displacement generators.
\end{proof}

A matrix $H$ is {\em Hankel-like} if
$\rank(AH-HB)$
is small
 where $A=Z_e$ and $B=Z_f^T$
for two scalars $e$ and $f$
or alternatively where $A=Z_e^T$ and $B=Z_f$. It follows
that $MN$ is a Han\-kel-like matrix if one of the
factors  is a Toep\-litz-like matrix
and another is a Han\-kel-like matrix,
whereas $MN$ is a Toeplitz-like matrix if
both $M$ and $N$ are Han\-kel-like matrices or 
both 
are Toep\-litz-like matrices.
We can alternatively define 
Hankel-like matrices as the products $TJ$
or $JT$ where $T$ is a Toeplitz-like matrix,
or we can define Toeplitz-like matrices $T$
as the products $HJ$ and $JH$ where $H$ are
Hankel-like matrices (cf. Fact \ref{fath}).
By using these properties we can 
readily extend our algorithms 
as well as expressions (\ref{eqtlk})
(cf. \cite[Example 4.4.4]{p01})
 from the case of Toep\-litz and
Toep\-litz-like to Hankel and
Hankel-like matrices, preserving the flop count.

\begin{remark}\label{restr}
By choosing the 
operator matrices $A$ and $B$
among $f$-circulant and appropriate diagonal matrices
we define the important classes of matrices $M$ with
the structures of Van\-der\-monde and Cauchy types
whose displacement rank, that is 
 $\rank(AM-MB)$, is small. This 
extends the classes of Van\-der\-monde matrices
$V_{\bf x}=(x_i^{j-1})_{i,j=1}^n$, 
having displacement rank $1$ for the operator matrices $A=\diag(x_i)_{i=1}^n$ 
and $B=Z_f^T$ for a scalar $f$, and Cauchy matrices
$C_{\bf s,\bf t}=(\frac{1}{s_i-t_j})_{i,j=1}^n$,
having displacement rank $1$ for the operator matrices $A=\diag(s_i)_{i=1}^n$ 
and $B=\diag(t_j)_{j=1}^n$.
Alternatively \cite{P90}, \cite{p01}, the  matrices of these classes 
can be defined as the products $UMV$
where $M$ is a Toeplitz matrix, whereas  $U$ and $V$ 
are properly selected among 
Van\-der\-monde matrices, their transposes and the identity matrices.
Similarly to the Toep\-litz--Han\-kel link at the end of the
previous subsection, this enables us to extend any successful 
algorithm for Cauchy-like inversion
to Toep\-litz-like, Han\-kel-like and  Van\-der\-monde-like inversion and vice versa
because $(UMV)^{-1}=V^{-1}M^{-1}U^{-1}$ \cite{P90}, 
although unlike the orthogonal reversion matrix $J$, 
 Van\-der\-monde
multipliers 
and their transposes are
usually ill conditioned 
except
for a narrow but important class including the
matrices 
$\Omega$
and $\Omega^H$ of Theorem \ref{thcpw}.
Theorems \ref{thtdsp} and \ref{thops} 
and other basic properties of Toep\-litz-like and Han\-kel-like matrices 
can be extended to the matrices having structures of
Van\-der\-monde or Cauchy types 
(see \cite{P00}, \cite{p01} or
Appendix A).
\end{remark}


\section{Ranks and conditioning of Gaussian random
matrices}\label{srrm}


\subsection{Random variables and  Gaussian random matrices}\label{srvrm}


\begin{definition}\label{defcdf}
$F_{\gamma}(y)=$ Probability$\{\gamma\le y\}$ for a real random variable $\gamma$
is the {\em cumulative 
distribution function (cdf)} of $\gamma$ evaluated at $y$. 
$F_{g(\mu,\sigma)}(y)=\frac{1}{\sigma\sqrt {2\pi}}\int_{-\infty}^y \exp (-\frac{(x-\mu)^2}{2\sigma^2}) dx$ 
for a Gaussian random variable $g(\mu,\sigma)$ with a mean $\mu$ and a positive variance $\sigma^2$,
and so   
\begin{equation}\label{eqnormal}
\mu-4\sigma\le y \le \mu+4\sigma~{\rm with ~a ~probability ~near ~1}.
\end{equation}
\end{definition}


\begin{definition}\label{defrndm}
A matrix or a vector is a {\em Gaussian random matrix or vector} with a mean 
$\mu$ and a positive variance $\sigma^2$ if it is filled with 
independent identically distributed Gaussian random 
variables, all having the mean $\mu$ and variance $\sigma^2$. 
$\mathcal G_{\mu,\sigma}^{m\times n}$ is the set of such
Gaussian  random  $m\times n$ matrices
(which are {\em standard} for $\mu=0$
and $\sigma^2=1$). By restricting this set 
to Toeplitz or $f$-circulant matrices we obtain the sets
$\mathcal T_{\mu,\sigma}^{m\times n}$ and
$\mathcal Z_{f,\mu,\sigma}^{n\times n}$ of
{\em Gaussian random Toep\-litz} 
and {\em Gaussian random $f$-circulant matrices}, 
respectively.  
\end{definition}


\begin{definition}\label{defchi}
$\chi_{\mu,\sigma,n}(y)$ is the cdf of the norm 
$||{\bf v}||=(\sum_{i=1}^n v_i^2)^{1/2}$ of a Gaussian random vector
${\bf v}=(v_i)_{i=1}^n\in \mathcal G_{\mu,\sigma}^{n\times 1}$. For 
 $y\ge 0$ we have 
$\chi_{0,1,n}(y)= \frac {2}{2^{n/2}\Gamma(n/2)}\int_{0}^yx^{n-1}\exp(-x^2/2) dx$ 
where $\Gamma(h)=\int_0^{\infty}x^{h-1}\exp(-x) dx$, $\Gamma (n+1)=n!$ for nonnegative integers $n$.
\end{definition}




\subsection{Nondegeneration of Gaussian random matrices}\label{sngrm}


The total degree of a multivariate monomial is the sum of its degrees
in all its variables. The total degree of a polynomial is the maximal total degree of 
its monomials.


\begin{lemma}\label{ledl} \cite{DL78}, \cite{S80}, \cite{Z79}.
For a set $\Delta$ of a cardinality $|\Delta|$ in any fixed ring  
let a polynomial in $m$ variables have a total degree $d$ and let it not vanish 
identically on this set. Then the polynomial vanishes in at most 
$d|\Delta|^{m-1}$ points. 
\end{lemma}


We assume that Gaussian random variables range 
over infinite sets $\Delta$,
usually over the real line or its interval. Then
the lemma implies that a nonzero polynomial vanishes with probability 0.
Consequently a square Gaussian random general, Toeplitz or circulant
matrix is nonsingular 
with probability 1
because its determinant is a polynomials
in the entries. 
Likewise rectangular
 Gaussian random general, Toeplitz and circulant 
matrices have full rank with probability 1. Furthermore  
all entries of such matrix $A$ and of its adjoint $\adj A$ 
are subdeterminants and thus are nonzeros 
with probability 1. Clearly this property of the adjoint also holds
for the inverse $A^{-1}=\frac{\adj A}{\det A}$
if the matrix $A$ is nonsingular. Hereafter,
wherever this causes no confusion,  
we assume by default that
{\em Gaussian random 
 general, Toeplitz and circulant 
matrices  have full rank,
and their inverses (if defined) have nonzero entries}.
These properties can be readily extended to the
products of the latter matrices by nonsingular and orthogonal 
matrices, and further to various 
functions of general, sparse and structured matrices. 
Moreover
similar properties hold with probability near 1 
where the random variables are sampled
under the uniform probability distribution
from a finite set of a large cardinality 
(see Appendix A).
 

\subsection{Extremal singular values of Gaussian random matrices}\label{scgrm}


Besides having full rank with probability 1,
Gaussian random matrices in Definition \ref{defrndm} are expected to be well conditioned  
\cite{D88}, \cite{E88}, \cite{ES05}, \cite{CD05}, \cite{B11}, and 
even the sum $M+A$ for  $M\in \mathbb R^{m\times n}$ and 
 $A\in \mathcal G_{\mu,\sigma}^{m\times n}$ is expected to
be well conditioned unless the ratio  $\sigma/||M||$ is small
or large 
\cite{SST06}. 

The following theorem 
states an upper bound 
proportional to $y$ on
 the cdf $F_{1/||A^+||}(y)$, that is  
on the probability that  the 
smallest positive singular value $1/||A^+||=\sigma_l(A)$ of a  Gaussian random matrix $A$ 
is less than a nonnegative scalar $y$ (cf. (\ref{eqnrm+}))
and consequently on the probability that the norm $||A^+||$
exceeds a positive scalar $x$.
The stated bound still holds if we replace the matrix $A$ by 
$A-B$ for any fixed matrix $B$, although
for $B=O_{m,n}$
the  bounds
can actually  be strengthened  
by a factor $y^{|m-n|}$ \cite{ES05}, \cite{CD05}.


\begin{theorem}\label{thsiguna} 
Suppose 
$A\in \mathcal G_{\mu,\sigma}^{m\times n}$, 
 $B\in \mathbb R^{m\times n}$,
$l=\min\{m,n\}$,  $x>0$, and $y\ge 0$. 
Then 
$F_{\sigma_l(A-B)}(y)\le 2.35~\sqrt l y/\sigma$, 
that is
$Probability \{||(A-B)^+||\ge 2.35x\sqrt {l}/\sigma\}\le 1/x$.
\end{theorem}
\begin{proof}
For $m=n$ this is \cite[Theorem 3.3]{SST06}. Apply
 Fact \ref{faccondsub} to extend it to any pair $\{m,n\}$.
\end{proof}


The following two theorems supply lower bounds
$F_{||A||}(z)$ and
$F_{\kappa (A)}(y)$
 on the probabilities 
that $||A||\le z$ 
and $\kappa(A)\le y$ for two scalars $y$ and $z$, 
respectively,
and a Gaussian random matrix $A$. 
We do not use the second theorem, but state it for the sake of completeness
and only for square $n\times n$ matrices $A$.
The theorems  
imply that  
the functions 
$1-F_{||A||}(z)$
and
$1-F_{\kappa (A)}(y)$ 
decay as 
$z\rightarrow \infty$ and
$y\rightarrow \infty$, respectively,
and that the decays are exponential in $-z^2$ and  proportional 
to $\sqrt{\log y}/y$, respectively.
 For small values $y\sigma$ and a fixed $n$ 
the lower bound of Theorem \ref{thmsiguna}
becomes negative, in which case 
the theorem becomes trivial. 
Unlike Theorem \ref{thsiguna}, in both theorems we assume that $\mu=0$. 


\begin{theorem}\label{thsignorm} \cite[Theorem II.7]{DS01}.
Suppose $A\in \mathcal G_{0,\sigma}^{m\times n}$,
$h=\max\{m,n\}$  and
$z\ge 2\sigma\sqrt h$. 
Then $F_{||A||}(z)\ge 1- \exp(-(z-2\sigma\sqrt h)^2/(2\sigma^2))$, and so
the norm $||A||$ is expected to have order $\sigma\sqrt h$. 
\end{theorem}


\begin{theorem}\label{thmsiguna}  \cite[Theorem 3.1]{SST06}.
Suppose  
$0<\sigma\le 1$,  
$y\ge 1$,  
 $A\in \mathcal G_{0,\sigma}^{n\times n}$. Then the matrix $A$
 has full rank with 
probability $1$ and 
$F_{\kappa (A)}(y)\ge 1-(14.1+4.7\sqrt{(2\ln y)/n})n/ (y\sigma)$.
\end{theorem}

  


 
\begin{proof}
See \cite[the proof of Lemma 3.2]{SST06}.
\end{proof}


\subsection{Extremal singular values of Gaussian random Toeplitz matrices}\label{scgrtm}


A  matrix 
$T_n=(t_{i-j})_{i,j=1}^n$
is the sum of two triangular 
Toeplitz matrices


\begin{equation}\label{eqt2tt}
T_n= Z({\bf t})+Z({\bf t_-})^T,~{\bf t}=(t_{i})_{i=0}^{n-1},~{\bf t}_-=(t'_{-i})_{i=0}^{n-1},~
t'_0=0.
\end{equation}
If $T_n\in  \mathcal T_{\mu,\sigma}^{n\times n}$, then
$T_n$ has $2n-1$ pairwise independent entries in $\mathcal G_{\mu,\sigma}$. Thus 
(\ref{eqttn}) implies that


$$ ||T_n||\le ||Z({\bf t})||+||Z({\bf t_-})^T||\le 
||{\bf t}||_1+||{\bf t_-}||_1= ||(t_{i})_{i=1-n}^{n-1}||_1\le \sqrt {2n}~||(t_{i})_{i=1-n}^{n-1}||.$$


\noindent Recall Definition \ref{defrndm} and obtain


\begin{equation}\label{eqtn}
F_{||T_n||}(y)\ge \chi_{\mu,\sigma,2n-1}(y/\sqrt {2n}).
\end{equation}


Next we estimate 
 the norm $||T_n^{-1}||$ 
for 
$T_{n}\in \mathcal T_{\mu,\sigma}^{n\times n}$.


\begin{lemma}\label{leinp} \cite[Lemma A.2]{SST06}.
For a nonnegative scalar $y$, a unit vector ${\bf t}\in \mathbb R^{n\times 1}$, and a vector
 ${\bf b}\in \mathcal G_{\mu,\sigma}^{n\times 1}$, 
 we have  
$F_{|{\bf t}^T{\bf b}|}(y)\le \sqrt{\frac{2}{\pi}}\frac{y}{\sigma}$.
\end{lemma}  


\begin{remark}\label{reinp}
The latter bound is independent of $\mu$ and $n$;
it holds for any $\mu$ even if 
all coordinates of the vector ${\bf b}$ are fixed except for a
single coordinate in $\mathcal G_{\mu,\sigma}$.
\end{remark}


\begin{theorem}\label{thsigunat1}  
Given a matrix 
$T_{n}=(t_{i-j})_{i,j=1}^n\in \mathcal T_{\mu,\sigma}^{n\times n}$,
assumed to be nonsingular (cf. Section \ref{sngrm}),
write 
$p_{1}={\bf e}_1^TT_n^{-1}{\bf e}_1$.
Then $F_{1/||p_{1}T_n^{-1}||}(y)\le 2n\alpha \beta$ 
 for two random variables  $\alpha$ and $\beta$
such that 
\begin{equation}\label{eqprtinv}
F_{\alpha}(y)\le \sqrt{\frac{2n}{\pi}}\frac{y}{\sigma}~{\rm and}~
F_{\beta}(y)\le \sqrt{\frac{2n}{\pi}}\frac{y}{\sigma}~{\rm for}~y\ge 0.
\end{equation}    
\end{theorem}


\begin{proof}  
Recall from part (a) of Theorem  \ref{thgs} that
$p_{1}T_n^{-1}=Z({\bf p})Z(J{\bf q})^T-Z(Z{\bf q})Z(ZJ{\bf p})^T$.
Therefore
$||p_{1}T_n^{-1}||\le ||Z({\bf p})||~||Z(J{\bf q})^T||+||Z(Z{\bf q})||~||Z(ZJ{\bf p})^T||$
for ${\bf p}=T_n^{-1}{\bf e}_1$,  ${\bf q}=T_n^{-1}{\bf e}_n$, and $p_1={\bf p}^T{\bf e}_1$.
It follows that
$||p_{1}T_n^{-1}||\le ||Z({\bf p})||~||Z(J{\bf q})||+||Z(Z{\bf q})||~||Z(ZJ{\bf p})||$
since $||A||=||A^T||$ for all matrices $A$.
Furthermore
$||p_{1}T_n^{-1}||\le||{\bf p}||_1~||J{\bf q}||_1+||Z{\bf q}||_1~||ZJ{\bf p}||_1$
due to (\ref{eqttn}).
Clearly $||J{\bf v}||_1=||{\bf v}||_1$ and $||Z{\bf v}||_1\le ||{\bf v}||_1$
for every vector ${\bf v}$, and so (cf. (\ref{eqnorm12}))
\begin{equation}\label{eqtpq}
||p_{1}T_n^{-1}||\le 2 ||{\bf p}||_1~||{\bf q}||_1\le 2n ||{\bf p}||~||{\bf q}||.
\end{equation}
  
By definition the vector ${\bf p}$ is orthogonal
to the vectors $T_n{\bf e}_2,\dots,T_n{\bf e}_n$, 
whereas ${\bf p}^TT_n{\bf e}_1=1$ (cf. \cite{SST06}).
Consequenty the vectors $T_n{\bf e}_2,\dots,T_n{\bf e}_n$
uniquely define the vector 
 ${\bf u}={\bf p}/||{\bf p}||$,
whereas 
$|{\bf u}^TT_n{\bf e}_1|=1/||{\bf p}||$.
The last coordinate $t_{n-1}$ of the vector $T_n{\bf e}_1$
is independent of the vectors $T_n{\bf e}_2,\dots,T_n{\bf e}_n$
and consequently of the vector ${\bf u}$. 
Apply 
 Remark \ref{reinp} to estimate the cdf of the random 
variable $\alpha=1/||{\bf p}||=|{\bf u}^TT_n{\bf e}_1|$  
 and obtain that
$F_{\alpha}(y)\le  \sqrt{\frac{2n}{\pi}}\frac{y}{\sigma}$ for $y\ge 0$.

Likewise the $n-1$ column vectors $T{\bf e}_1,\dots,T_{n-1}$
 define the vector ${\bf v}=\beta{\bf q}$ for 
$\beta=1/||{\bf q}||=|{\bf v}^TT_n{\bf e}_n|$.
The first coordinate $t_{1-n}$ of the vector $T_n{\bf e}_n$
is independent of the vectors $T{\bf e}_1,\dots,T_{n-1}$
and consequently of the vector ${\bf v}$. 
Apply 
 Remark \ref{reinp} to
estimate the cdf of the random 
variable $\beta$ and obtain that
$F_{\beta}(y)\le  \sqrt{\frac{2n}{\pi}}\frac{y}{\sigma}$ for $y\ge 0$.
Finally combine these bounds on the cdfs $F_{\alpha}(y)$ and 
$F_{\beta}(y)$ with (\ref{eqtpq}).
\end{proof}
 

By applying parts (b) and (c)
of Theorem  \ref{thgs} instead of its part (a),
we similarly
deduce the
bounds $||v_0T_{n+1}^{-1}||\le 2\alpha\beta$ and
 $||v_nT_{n+1}^{-1}||\le 2\alpha\beta$
for two pairs of random variables $\alpha$ and $\beta$
that
satisfy (\ref{eqprtinv}) for $n+1$ replacing $n$.
We have  $p_{1}=\frac{\det T_{n-1}}{\det T_{n}}$,
$v_0=\frac{\det T_n}{\det T_{n+1}}$, and
$v_n=\frac{\det T_{0,1}}{\det T_{n+1}}$
for $T_{0,1}=(t_{i-j})_{i=0,j=1}^{n-1,n}$.
Next we bound the 
geometric means 
of the 
ratios 
$|\frac{\det T_{h+1}}{\det T_{h}}|$
for
$h=1,\dots,k-1$. 
 $1/|p_1|$ and  $1/|v_0|$
are such  ratios for $k=n-1$ and $k=n$,
respectively,
whereas the  ratio  $1/|v_n|$ is similar to 
$1/|v_0|$, under slightly distinct notation. 



\begin{theorem}\label{thhdmr} 
 Let $T_h\neq O$ denote $h\times h$ matrices
for $h=1,\dots,k$  
whose entries have absolute values at most $t$
for a fixed scalar or random variable $t$, e.g. for $t=||T||$.
Furthermore let $T_1=(t)$.
Then the geometric mean $(\prod_{h=1}^{k-1}|\frac{\det T_{h+1}}{\det T_{h}}|)^{1/(k-1)}=\frac{1}{t}|\det T_{k}|^{1/(k-1)}$
is at most $k^{\frac{1}{2}(1+\frac{1}{k-1})}t$.
\end{theorem}


\begin{proof}
The theorem follows from 
Hadamard's upper bound
$|\det M|\le k^{k/2}t^k$, which holds
for any $k\times k$ matrix $M=(m_{i,j})_{i,j=1}^k$
with $\max_{i,j=1}^k|m_{i,j}|\le t$.
\end{proof}
 
The theorem says that
the geometric mean of the ratios $|\det T_{h+1}/\det T_{h}|$
for 
$h=1,\dots,k-1$
 is not greater than $k^{0.5+\epsilon(k)}t$
where $\epsilon(k)\rightarrow 0$ as $k\rightarrow \infty$.
Furthermore if
$T_n\in \mathcal G_{\mu,\sigma}^{n\times n}$
we can write
$t=||T||$  
 and 
 apply (\ref{eqtn}) to bound the cdf of $t$. 


\subsection{Extremal singular values of Gaussian random circulant matrices}\label{scgrcm}


Next we estimate the norms of a random Gaussian $f$-circulant matrix 
and its inverse. 


\begin{theorem}\label{thcircsing}
Assume $y\ge 0$ and a circulant $n\times n$ matrix $T=Z_1({\bf v})$ 
for ${\bf v}\in \mathcal G_{\mu,\sigma}^{n\times 1}$. Then
 
(a) $F_{||T||}(y)\ge \chi_{\mu,\sigma,n} (\sqrt {\frac{2}{n}}y)$
for  $\chi_{\mu,\sigma,n}(y)$ in Definition \ref{defchi} and
(b) $F_{1/||T^{-1}||}(y)\le \sqrt{\frac{2}{\pi}} \frac{ny}{\sigma}$.
\end{theorem} 


\begin{proof}
For the matrix $T=Z_1({\bf v})$
we have both equation (\ref{eqt2tt}) and the bound 
$||{\bf t_-}||_1\le||{\bf t}||_1$,
and so $||T||_1\le 2||{\bf t}||_1$. Now
part (a) of the theorem follows similarly to (\ref{eqtn}).
To prove part (b)
recall
 Theorem \ref{thcpw} and write 
$B=\Omega T\Omega^{-1}=D({\bf u})$,
${\bf u}=(u_i)_{i=0}^{n-1}=\Omega {\bf v}$. We have
$\sigma_j(T)=\sigma_j(B)$ for all $j$ because
$\frac{1}{\sqrt n}\Omega$ 
and $\sqrt n\Omega^{-1}$ are unitary matrices.
By combining the equations $u_i={\bf e}_i^T\Omega{\bf v}$, the bounds 
$||\Re ({\bf e}_i^T\Omega)||\ge 1$ for all $i$,
and Lemma \ref{leinp}, deduce that 
$F_{|\Re (u_i)|}(y)\le  \sqrt{\frac{2}{\pi}} \frac{y}{\sigma}$
for $i=1,\dots,n$.   
We have  $F_{\sigma_n(B)}(y)=F_{\min_i|u_i|}(y)$ because 
$B=\diag(u_i)_{i=0}^{n-1}$, and
clearly $|u_i|\ge |\Re (u_i)|$.
\end{proof}




\begin{remark}\label{retcond}
Our extensive experiments suggest that 
the estimates of Theorem \ref{thcircsing} are  overly pessimistic
 (cf. Table \ref{tabcondcirc}).
\end{remark}


Combining Theorem  \ref{thcpw} with minimax property (\ref{eqminmax}) implies that 
$$\frac{1}{g(f)}\sigma_j(Z_1({\bf v}))\le \sigma_j(Z_f({\bf v}))\le g(f) \sigma_j(Z_1({\bf v}))$$
for all vectors ${\bf v}$, scalars $f\neq 0$,
$g(f)=\max\{|f|^2,{1/|f|^2}\}$, and $j=1,\dots,n$. Thus we can readily extend
the estimates of Theorem \ref{thcircsing} to $f$-circulant matrices for $f\neq 0$.
In particular Gaussian random $f$-circulant matrices 
 tend to be 
well conditioned unless  $f\approx 0$ or $1/f\approx 0$.


\section{Condition numbers of randomized matrix products and generic preconditioning}\label{smrc}


Next we deduce probabilistic lower bounds on the smallest 
singular values of the products of fixed and random matrices.
We begin with three lemmas. The first of them is obvious,
the second easily follows  
from minimax property (\ref{eqminmax}).


\begin{lemma}\label{lepr2}
$\sigma_{j}(SM)=\sigma_j(MT)=\sigma_j(M)$ for all $j$ if $S$ and $T$ are square orthogonal matrices.
\end{lemma}


\begin{lemma}\label{lepr1} 
Suppose $\Sigma=\diag(\sigma_i)_{i=1}^{n}$, $\sigma_1\ge \sigma_2\ge \cdots \ge \sigma_n$,
$G\in \mathbb R^{r\times n}$, $H\in \mathbb R^{n\times r}$.
Then 
$\sigma_{j}(G\Sigma)\ge\sigma_{j}(G)\sigma_n$,
$\sigma_{j}(\Sigma H)\ge\sigma_{j}(H)\sigma_n$ for all $j$.
If also $\sigma_n>0$, then 
 $\rank (G\Sigma)=\rank (G)$, $\rank (\Sigma H)=\rank (H)$.
\end{lemma}

We employ the following result in the proof of Corollary \ref{cosumm3}.


\begin{corollary}\label{copr}
We have $\kappa (AB)\le \kappa (A) \kappa (B)$
if $A$ or $B$ is a nonsingular matrix.
\end{corollary}


\begin{proof}
Assume SVDs $A=S_A\Sigma_AT^T_A$ of (\ref{eqsvd}).
Then $\sigma_j(AB)=
\sigma_j(S_A\Sigma_AT^T_AB)=
\sigma_j(\Sigma_A\widehat B)$ where $\widehat B=T^T_AB$.
Let $A$ and consequently $\Sigma_A$ be nonsingular $n\times n$ matrices.
Apply Lemma \ref{lepr1} and deduce that
$\sigma_j(\Sigma_A\widehat B)\ge \sigma_j(\widehat B)\sigma_n(A)$,
whereas $\sigma_j(\widehat B)=\sigma_j(B)$ for all $j$.
We have $\rho=\rank (AB)=\rank (B)\le n$.
Combine the relationships above for $j=\rho$
and obtain that 
$\sigma_{\rho}(AB)=\sigma_{\rho}(\Sigma_A\widehat B)\ge \sigma_{\rho}(\widehat B)\sigma_n(A)=
\sigma_{\rho}(B)\sigma_n(A)$, and so 
$\sigma_{\rho}(AB)\ge \sigma_{\rho}(B)\sigma_n(A)$.
Also note that $||AB||\le ||A||~||B||$.
Combine the latter bounds and obtain that
 $\kappa (AB)=||AB||/\sigma_{\rho}(AB)\le ||A||~||B||/(\sigma_{\rho}(B)\sigma_n(A))=
\kappa (A) \kappa (B)$. Similarly prove the claimed bound where
$B$ is a nonsingular  matrix.
\end{proof}


\begin{lemma}\label{lepr3} \cite[Proposition 2.2]{SST06}.
Suppose $H\in \mathcal G_{\mu,\sigma}^{m\times n}$, $SS^T=S^TS=I_m$, $TT^T=T^TT=I_n$.
Then $SH\in \mathcal G_{\mu,\sigma}^{m\times n}$ and $HT\in \mathcal G_{\mu,\sigma}^{m\times n}$.
\end{lemma}

The following theorem implies that  
multiplication by
standard Gaussian random matrix is unlikely to decrease
the smallest positive singular value of a matrix
dramatically,  even though
$UV=O$ for some pairs of rectangular orthogonal matrices $U$ and $V$.

\begin{theorem}\label{1}
Suppose $G'\in \mathcal G_{\mu,\sigma}^{r\times m}$, $H'\in \mathcal G_{\mu,\sigma}^{n\times r}$,
 $M\in \mathbb R^{m\times n}$, 
$G=G'+U$, $H=H'+V$ for 
some matrices $U$ and $V$,
$r(M)=\rank (M)$, $x>0$ and
$y\ge 0$. 
Then
$F_{1/||(GM)^+||}(y)\le F(y,M,\sigma)$ and
$F_{1/||(MH)^+||}(y)\le F(y,M,\sigma)$
for $F(y,M,\sigma)=2.35 y \sqrt {\widehat r}||M^+||/\sigma$
and
$\widehat r=\min\{r,r(M)\}$,
that is 
$Probability \{||P^+||\ge 2.35x\sqrt {\widehat r}||M^+||/\sigma\}\le 1/x$
for $P=GM$ and $P=MH$.
\end{theorem} 



\begin{proof}
With probability $1$,
the matrix $MH$
has rank $\widehat r$ 
because $H\in \mathcal G_{\mu,\sigma}^{n\times r}$.
So (cf. (\ref{eqnrm+}))
\begin{equation}\label{eqfmw}
F_{1/||(MH)^+||}(y)=F_{\sigma_{\widehat r}(MH)}(y).
\end{equation}
 Let
$M=S_M\Sigma_MT^T_M$ be  full SVD where
$\Sigma_M=\diag(\widehat \Sigma_M, O)= \Sigma_M\diag(I_{r(M)},O)$ and
$\widehat \Sigma_M=\diag(\sigma_j(M))_{j=1}^{r(M)}$
is a nonsingular diagonal matrix. 
We have $MH=S_M\Sigma_MT_M^TH$, and so  
$\sigma_j(MH)=\sigma_j(\Sigma_MT_M^TH)$ for all $j$ 
by virtue of Lemma \ref{lepr2}, because 
$S_M$ is a square orthogonal matrix. 
Write 
$H_{r(M)}=(I_{r(M)}~|~O)T_M^TH$ and observe that
$\sigma_j(\Sigma_MT_M^TH)= \sigma_j(\widehat \Sigma_MH_{r(M)})$ 
 and consequently 
\begin{equation}\label{eqjmw}
\sigma_j(MH)= \sigma_j(\widehat \Sigma_MH_{r(M)})~{\rm for~all}~j.
\end{equation}
Combine equation
(\ref{eqjmw}) for $j=\widehat r$ with Lemma \ref{lepr1} 
for the pair $(\Sigma,H)$ replaced by $(\widehat \Sigma_M,H_{r(M)})$
 and 
obtain that 
$\sigma_{\widehat r}(MH)\ge \sigma_{r(M)}(M)\sigma_{\widehat r}(H_{r(M)})=
\sigma_{\widehat r}(H_{r(M)})/||M^+||$.
We have  $T_M^TH'\in \mathcal G_{\mu,\sigma}^{n\times r}$ 
by virtue of Lemma \ref{lepr3}, because $T_M$ is a square orthogonal matrix;
consequently  $H_{r(M)}=H_{r(M)}'+B$ for $H_{r(M)}' \in \mathcal G_{\mu,\sigma}^{r(M)\times r}$
and some matrix $B$.
Therefore we can apply
Theorem \ref{thsiguna} for $A=H'_{r(M)}$ 
and obtain the bound  of Theorem \ref{1} on $F_{1/||(MH)^+||}(y)$.
 One can similarly deduce the bound on $F_{1/||(GM)^+||}(y)$  
or can just 
apply the above bound on 
$F_{1/||(MH)^+||}(y))$ for $H=G^T$ and $M$ replaced by $M^T$
and then recall that $(M^TG^T)^T=GM$.
\end{proof}

By combining (\ref{eqnorm12inf}) with Theorems \ref{thsignorm} (for $B=O$) and \ref{1}
we can probabilistically bound the condition numbers of 
randomized  
products $GM$ and $MH$.
The following corollary extends the bound of Theorem \ref{1} 
for a randomized matrix product to the bounds for
its leading blocks. 


\begin{corollary}\label{cogh}
Suppose $j$, $k$, $m$, $n$, $q$ and $s$ are integers, $1\le j\le q$, $1\le k\le s$, 
$M\in \mathbb R^{m\times n}$, $\sigma>0$,
$G\in  \mathcal G_{\mu,\sigma}^{q\times m}$, $H\in \mathcal G_{\mu,\sigma}^{n\times s}$,
$\rank (M_j)=j$ for $M_j=M\begin{pmatrix}I_j   \\   O_{n-j,j}\end{pmatrix}$,
$\rank (M^{(k)})=k$ for $M^{(k)}=(I_k~|~O_{k,m-k})M$,    
and  $y\ge 0$.
Then 
(i) with probability $1$ the matrix $GM$ (resp. $MH$) has generic rank profile
if $\rank (M)\ge q$ (resp. if $\rank (M)\ge s$).
Furthermore
(ii) $F_{1/||((GM)_j^{(j)})^+||}(y)\le 2.35 y  \sqrt {j}/(||M_j^+||\sigma)$ 
 if $\rank (M)\ge j$,
$F_{1/||((MH)_k^{(k)})^+||}(y)\le 2.35 y  \sqrt {k}(||(M^{(k)})^+||\sigma)$ if $\rank (M)\ge k$.
\end{corollary} 

\begin{proof}
We immediately verify part (i)
by applying the techniques of Section \ref{sngrm}. 
To prove part (ii) 
apply 
Theorem \ref{1}
replacing $G$ by $(I_j~|~O_{j,q-j})G$
and replacing $M$  by   
 $M\begin{pmatrix}I_j   \\   O_{n-j,j}\end{pmatrix}$. 
For every $k$ apply 
Theorem \ref{1}   
 replacing $M$ by $(I_k~|~O_{k,m-k})M$ and replacing $H$  by   
 $H\begin{pmatrix}I_k   \\   O_{s-k,k}\end{pmatrix}$. 
\end{proof}


\begin{remark}\label{regprec}
It is well known that 
GENP and   block Gaussian elimination
are  numerically unsafe where  the input
matrix $M$
has a singular or ill conditioned leading block,
but if this matrix itself is well  conditioned, then
the latter results combined 
with (\ref{eqnorm12inf}) and Theorems
\ref{thnorms} and    
\ref{thsignorm} for $B=O$ imply that
 multiplication by Gaussian random matrices
is expected to fix this problem.
Namely both  elimination algorithms 
 applied to 
the matrices $GM$ and $MH$ 
for $G\in \mathcal G_{0,1}^{m\times m}$ 
and $H\in \mathcal G_{0,1}^{n\times n}$
are
expected to
use no divisions by absolutely small values.
\end{remark}


\begin{remark}\label{rertm}
We cannot  extend the proofs of
Lemma \ref{lepr3} and consequently Theorem \ref{1}
and its corollaries to the case of Gaussian random Toeplitz 
matrices $G\in \mathcal T_{\mu,\sigma}^{r\times m}$ 
and $H\in \mathcal T_{\mu,\sigma}^{n\times r}$, but 
the results of our tests have consistently
supported
such extensions (cf. Tables \ref{tab44} and \ref{tabSVD_HEAD1T}).
This is also the case for our results in the next two sections.  
\end{remark}


\section{Randomized additive and dual additive preconditioning}\label{sapaug}




In this section we prove Theorem \ref{thkappa} and 
extend it to the cases of rectangular  $m\times n$ matrices  
$A$
and dual additive preprocessing.
At first we prove the following specification of
the theorem where instead of the matrix $A$ having numerical rank $\rho$
we deal with its SVD truncation 
having rank $\rho$.

\begin{theorem}\label{thkappa1}
Suppose $A$ is a 
real $n\times n$ matrix of a rank $\rho$, $0<\rho<n$, 
$\sigma U$ and $\sigma V$ are
standard Gaussian random $n\times r$ matrices,
whose  
all $2nr$ entries are independent of each other,
$0<r<n$, 
 and  $C= A+UV^T$.
Then 
(i) we have 
 $||A||_2-||U||_2~||V||_2\le||C||_2\le ||A||_2+||U||_2~||V||_2$,
which implies (\ref{eqnca}), 
(ii) the matrix $C$ is 
singular  if $r<n-\rank (A)$, 
(iii) otherwise it is
nonsingular with probability $1$,
 and 
(iv)
 the value $\sigma_n(C)$
is expected to have at most order 
$\sigma_{\rho}(A)$
if
the ratio $\sigma/|| A||_2$
is neither large nor small,
e.g., if
$\frac{1}{100}\le \sigma/|| A||_2\le 100$. 
\end{theorem} 
   

\subsection{Proof of Theorem \ref{thkappa1} (parts (i)--(iv), case $r=n-\rho$)}\label{spmt0}


Part (i) is an immediate observation,
which implies (\ref{eqnca}) by virtue of
(\ref{eqnormal})
and Theorem \ref{thsignorm}
 for $A=U$, $A=V$ 
and $h=n$. 
Furthermore
we readily prove parts (ii) and (iii) of the theorem
(on singularity and nonsingularity)
by applying the techniques of Section \ref{sngrm}. 
 To prove  part (iv), that is to bound the ratio $||C^{-1}||/||A^+||$,
we 
factorize 
 the matrix $C$, which involves a number of technicalities.
In this subsection we only handle the case where $r=n-\rho$.
 
\begin{theorem}\label{th5.2} 
Suppose $A,C,S,T\in \mathbb R^{n\times n}$ and 
$U,V\in \mathbb R^{n\times r}$
for two positive integers $r$ and $n$, $r\le n$,
$A=S\Sigma T^T$ is  full SVD of the matrix $A$ (cf. (\ref{eqsvd})),
  $S$ and  $T$
are  square
orthogonal matrices,
$\Sigma = \diag(\sigma_j)_{j=1}^n$,
the matrix $C=A+UV^T$ 
is nonsingular, and so 
$\rho=\rank (A)=n-r$ and
$\sigma_{\rho}>0$.
 Write 
\begin{equation}\label{eqstuv}  
S^TU =     \begin{pmatrix}
                        \bar U    \\
                        U_{r}
                \end{pmatrix},
 ~ T^TV =     \begin{pmatrix}
                        \bar V    \\
                        V_r
                \end{pmatrix},
 ~ R_U =        \begin{pmatrix}
                        I_{\rho}     &    \bar U     \\
                        O_{r,\rho}       &       U_{r}
                \end{pmatrix},
 ~ R_V =        \begin{pmatrix}
                        I_{\rho}     &      \bar V    \\
                        O_{r,\rho}       &       V_r
                \end{pmatrix},
\end{equation}
where $U_r$ and $V_r$ are $r\times r$ matrices. Then 

${\rm (a)}~R_U\Sigma R_V^T=\Sigma$, whereas $R_U\diag(O_{\rho,\rho},I_r)R_V^T=S^TUV^TT$, 
and so 
\begin{equation}\label{eqcauv} 
C = SR_UDR_V^TT^T,~D=\Sigma+\diag(O_{\rho,\rho},I_{r})=\diag (d_j)_{j=1}^n
\end{equation}
where
$d_j=\sigma_j$ for $j=1,\dots,\rho$, $d_j=\sigma_j+1$ for $j=\rho+1,\dots,n$.




Furthermore 
suppose that the matrix $A$ has been normalized so that $||A||=1$ 
 and that the  $r\times r$ matrices $U_{r}$ 
and $V_r$ are  nonsingular,
which holds with probability $1$ where $U$ and $V$ are Gaussain random matrices
(cf. Section \ref{sngrm}).
Write 
\begin{equation}\label{eqpuv} 
p=||R_U^{-1}||~||R_V^{-1}||~
{\rm and}~f_r=\max \{1,||U_r^{-1}||\}~\max\{1,||V_r^{-1}||\}.
\end{equation}
 Then 

${\rm (b)}$~the matrix $C$ is nonsingular, 


${\rm (c)}~1\le ||R_V||~||R_U||\le \sigma_{\rho}(A)/\sigma_{n}(C)\le p$,

${\rm (d)}~ p\le (1+||U||)(1+||V||)f_r$,

${\rm (e)}~1\le \sigma_{\rho} (A)/ \sigma_n (C)\le (1+||U||)(1+||V||)f_r$.

\end{theorem}


\begin{proof}
Parts (a) and (b) are readily verified. 
 
(c) Combine the equations  
$S^{-1}=S^T$, $T^{-1}=T$ and 
(\ref{eqcauv}) 
and obtain
$C^{-1}= TR_V^{-T}D^{-1}R_U^{-1}S^T$
or equivalently $D^{-1}=R_VT^TC^{-1}SR_U$.  It follows that
$||C^{-1}||=||R_V^{-T}D^{-1}R_U^{-1}||$
and $||D^{-1}||=||R_VT^TC^{-1}SR_U||$.
Apply bound (\ref{eqnorm12inf}), substitute $||S||=||S^T||=||T||=||T^T||=1$
and obtain
$||C^{-1}||\le ||R_V^{-T}||~||D^{-1}||~||R_U^{-1}||$
and $||D^{-1}||\le ||R_V||~||C^{-1}||~||R_U||$.
Substitute the equations (\ref{eqpuv}),
 $||D^{-1}||=1/\sigma_{\rho}(A)$
(implied by 
the equations $||A||=1$ and (\ref{eqcauv}))
and $||C^{-1}||=1/\sigma_{n}(C)$
and  the bounds $||R_V||\ge 1$ and $||R_U||\ge 1$
and obtain that 
$1\le \sigma_{\rho}(A)/\sigma_{n}(C)\le p$.

(d) Observe that  $R_U^{-1} =        \begin{pmatrix}
                        I_{\rho}     &     -\bar U    \\
                        O       &       I_r
                \end{pmatrix}\begin{pmatrix}
                        I_{\rho}     &     O   \\
                        O       &       U_r^{-1}
                \end{pmatrix}$,
 ~ $R_V^{-1}=        \begin{pmatrix}
                        I_{\rho}     &       -\bar V     \\
                        O       &      I_r
                \end{pmatrix}\begin{pmatrix}
                        I_{\rho}     &     O   \\
                        O       &       V_r^{-1}
                \end{pmatrix}$,
$||\bar U||\le ||U||$ and $||\bar V||\le ||V||$.
Then combine these relationships. 

(e) Combine the bounds
of parts (c) and (d).
\end{proof}


We have $\frac{\kappa (C)}{\kappa (A)}\le \frac{||C||}{||A||}\frac{\sigma_{\rho}(A)}{\sigma_{n}(C)}$,
and so parts (d) and (e) together
 bound the ratio $\frac{\kappa (C)}{\kappa (A)}$ in terms of the norms
$||U||$, $||V||$, $||U_r^{-1}||$ and $||V_r^{-1}||$
as follows,
\begin{equation}\label{eqcca} 
\frac{\kappa (C)}{\kappa (A)}\le (1+||U||~||V||)(1+||U||)(1+||V||)~
\max \{1,||U_r^{-1}||\}~\max\{1,||V_r^{-1}||\},
\end{equation}
and in particular
\begin{equation}\label{equv}
\frac{\kappa (C)}{\kappa (A)}\le(1+||U||^2)(1+||U||)^2\max\{1,||U_r^{-1}||^2\}~{\rm if}~ U=V.
\end{equation}

Let us estimate the norms 
$||U||$, $||V||$, $||U_r^{-1}||$ and $||V_r^{-1}||$
where 
 $U$ and $V$ are Gaussian random matrices.


 

\begin{theorem}\label{thur} 
Suppose that $A$, $U$, $V$, 
 $U_r$ and $V_r$ denote the five matrices of  
 Theorem \ref{th5.2} where 
$U,V\in \mathcal G_{\mu,\sigma}^{n\times r}$.
Then 
 $\max \{F_{1/||U_r^{-1}||}(y),F_{1/||V_r^{-1}||}(y)\}\le  2.35~y\sqrt r/\sigma$ for $y\ge 0$.
\end{theorem}
 

\begin{proof}
Lemma \ref{lepr3} implies that $S^TU,T^TV\in \mathcal G_{\mu,\sigma}^{n\times r}$
by virtue of Lemma \ref{lepr3},
because $S$ and $T$ are square
orthogonal matrices. Hence
$U_r,V_r\in \mathcal G_{\mu,\sigma}^{r\times r}$.
Apply Theorem \ref{thsiguna}  for $A=U_r$ and $A=V_r$ where
in both cases $m=n=r$.
\end{proof}


Combine
Theorem
\ref{thur} 
 with relationships (\ref{eqpuv})--(\ref{equv})
and obtain part (iv) of Theorem \ref{thkappa1}
in the case where $r=n-\rho$ and $||A||=1$.
Relax 
the normalization assumption
by  scaling the matrix $A$.


\begin{remark}\label{reu=v}
So far our proof of Theorem \ref{thkappa1}  remains
valid where $U=V\in \mathcal G_{0,1}^{n\times r}$,
but not so in the case where $r>n-\rho$,
covered in the next subsection.
We can make similar comments on the extension to the proofs of
Theorems \ref{thkappa} and \ref{thkappa2} in Section \ref{sextcor}.
\end{remark}


\subsection{Proof of Theorem \ref{thkappa1} (part (iv), case $r>n-\rho$)}\label{spmt}


Next we  
extend the proof of part (iv) to the case where
$s=r+\rho-n>0$.
The extension is immediate 
where the matrix $A$ is nonnegative definite 
and $U=V$ but is more involved 
and less transparent in the general case.
Write 
 $U=(U^{(s)}~|~U_s)$ and $V=(V^{(s)}~|~V_s)$
where
$U^{(s)},V^{(s)}\in \mathcal G_{0,\sigma}^{(n-s)\times r}$
and 
$U_s,V_{s}\in \mathcal G_{0,\sigma}^{s\times r}$.
As we have proved already, 
 the $n\times n$ matrices 
$C^{(s)}=A+U^{(s)}V^{(s)T}$
and  $C=A+UV^T=C^{(s)}+U_{s}V_{s}^T$
are nonsingular with probability $1$
and are expected to
have norms of order $||A||$ for reasonably bounded values $\sigma$
(cf. (\ref{eqnca})),
whereas the matrix $C^{(s)}$ 
is expected to be well conditioned,
that is 
the ratio $\kappa_s=||C^{(s)}||/\sigma_n(C^{(s)})\ge 1$ is not 
expected to be large.
To simplify the notation, scale the matrices $A$, 
$C$, $U$
and $V$ to have $\sigma=1$,
expecting that the scaling factor
and 
the new value of the norm 
$||C^{(s)}||$
is
 neither large nor small
(cf. (\ref{eqnormal})).

Let $C^{(s)}=S\Sigma T^T$ be SVD,
where $\Sigma=\diag(\sigma_j(C^{(s)}))_{j=1}^n$,
premultiply the equation 
$C=C^{(s)}+U_{s}V_{s}^T$
by  $S^T$,
postmultiply it by $T$,
write $\widehat C=S^TCT$,
$\widehat U=S^TU_{s}$, and
$\widehat V^T=V_{s}^TT$,
and obtain $\widehat C=\Sigma+\widehat U\widehat V^T$
where $\sigma_j(\widehat C)=\sigma_j(C)$ for all $j$
by virtue of Lemma \ref{lepr2} and 
$\widehat U,\widehat V\in \mathcal G_{0,1}^{n\times r}$
by virtue of Lemma \ref{lepr3},
because $S$ and $T$ are square
orthogonal matrices.
Furthermore write 
 $\delta_{ij}=0$ for $i\neq j$, $\delta_{ii}=1$,
${\bf c}_i=\widehat C^T {\bf e}_i$,
${\bf c}_j^-=\widehat C ^{-1}{\bf e}_j$,
for $i,j=1,\dots,n$,
and so 
\begin{equation}\label{eqdelt} 
{\bf c}_i^T{\bf c}_j^-={\bf e}_i^T{\bf e}_j=\delta_{ij}
~{\rm for}~i,j=1,\dots,n,
\end{equation}
 \begin{equation}\label{eqrcpr}
||{\bf c}_j^-||=1/{\bf c}_j^T\widehat {\bf c}_j~{\rm for}~j=1,\dots,n.
\end{equation}
 Indeed for every $j$
 the  unit vector
 $\widehat {\bf c}_j={\bf c}_j^-/||{\bf c}_j^-||=(\widehat c_{ij})_{i=1}^n$
is unique because the
vector ${\bf c}_j^-$
is orthogonal to all vectors ${\bf c}_i$ for $i\neq j$
(cf. (\ref{eqdelt})).
Combine the latter equation $\widehat {\bf c}_j={\bf c}_j^-/||{\bf c}_j^-||$
with ${\bf c}_j^T{\bf c}_j^-=1$
and deduce equation (\ref{eqrcpr})
(cf. \cite[Proof of Lemma 3.2]{SST06}).

Next
 we estimate the norm $||{\bf c}_j^-||$  
from above or equivalently the value
${\bf c}_j^T\widehat {\bf c}_j$ from below
 for any fixed integer $j$, $1\le j\le n$.  
We write
$\widehat {\bf u}_j^T={\bf e}_j^T\widehat U\in  \mathcal G_{0,\sigma}^{n\times 1}$ 
and $\widehat {\bf t}_j=\widehat V^T\widehat {\bf c}_j$
and are going to deduce a probabilistic upper bound on the norm 
$||\widehat {\bf t}_j||$ as long as
we have a probabilistic upper bound on 
the norm $||{\bf c}_j^-||$. 
Represent the value ${\bf c}_j^T\widehat {\bf c}_j$ as
${\bf e}_j^T\widehat C\widehat {\bf c}_j=
{\bf e}_j^T\Sigma\widehat {\bf c}_j+{\bf e}_j^T\widehat U\widehat V^T\widehat {\bf c}_j$
and infer that
\begin{equation}\label{eqlbnd}
{\bf c}_j^T\widehat {\bf c}_j=
\sigma_j(C^{(s)})\widehat c_{jj}+\widehat {\bf u}_j^T\widehat {\bf t}_j.
\end{equation}
Recall Lemma \ref{leinp} for $\sigma=1$ and
${\bf b}={\bf u}_j$, recall Remark \ref{reinp}, and
 obtain
that 
$F_{{\bf c}_j^T\widehat {\bf c}_j}(y)\le \sqrt\frac{2}{\pi}\frac{y}{||\widehat {\bf t}_j||}$.
Write $p=F_{{\bf c}_j^T\widehat {\bf c}_j}(y)$
and infer that
\begin{equation}\label{eqnt}
 ||\widehat {\bf t}_j||\le \sqrt\frac{2}{\pi}\frac{y}{p}.
\end{equation}
Next deduce upper estimates for the values $|\widehat c_{ij}|$
for all $i$, at first for $i=j$.
Obtain from equation (\ref{eqlbnd})
that 
$\widehat c_{jj}=
({\bf c}_j^T\widehat {\bf c}_j-\widehat {\bf u}_j^T\widehat {\bf t}_j)/\sigma_j(C^{(s)})$.
Substitute ${\bf c}_j^T\widehat {\bf c}_j\le y$ and (\ref{eqnt}) and obtain
$|\widehat c_{jj}|\le (y+||\widehat {\bf u}_j||~||\widehat {\bf t}_j||)/\sigma_j(C^{(s)})
\le(y+||\widehat {\bf u}_j||~\sqrt\frac{2}{\pi}\frac{y}{p})/\sigma_j(C^{(s)})$,
and so 
\begin{equation}\label{eqj}
|\widehat c_{jj}|\le(1+\sqrt\frac{2}{\pi}\frac{||\widehat {\bf u}_j||}{p})y \kappa_s,
\end{equation}
where the value $\kappa_s=1/\sigma_n(C^{(s)})$ is
not expected to be large and where the cdf
$F_{||\widehat {\bf u}_j||}(y)=\chi_{0,1,n}(y)$
is bounded in
Definition \ref{defchi}.

Next
let $i\neq j$,
recall that ${\bf c}_i^T{\bf c}_j^-=0$ 
(cf. (\ref{eqdelt})),
substitute
${\bf c}_i^T={\bf e}_i^T\widehat C={\bf e}_i^T\Sigma +\widehat {\bf u}_i^T\widehat V^T$
and obtain
${\bf c}_i^T{\bf c}_j^-={\bf e}_i^T\Sigma {\bf c}_j^-+\widehat {\bf u}_i^T\widehat V^T{\bf c}_j^-=
\sigma_i(C^{(s)})\widehat c_{ij} +\widehat {\bf u}_i^T\widehat {\bf t}_j=0$.
Hence
$|\widehat c_{ij}|\le |\widehat {\bf u}_i^T\widehat {\bf t}_j|/\sigma_i(C^{(s)})$. 
Substitute (\ref{eqnt}) and obtain
\begin{equation}\label{eqii}
|\widehat c_{ij}|\le\sqrt\frac{2}{\pi}\frac{||\widehat {\bf u}_i||}{p}y\kappa_s
~{\rm for~ all}~i\neq j.
\end{equation}

Combine equations (\ref{eqj}) and (\ref{eqii}) 
and obtain that $||\widehat {\bf c}_j||^2=\sum_{i=1}^n  \widehat c_{ij}^2\le \gamma y^2 \kappa_s^2$
where  $\gamma=1+2\sqrt\frac{2}{\pi}\frac{||\widehat {\bf u}_j||}{p}+\frac{2}{\pi} \sum_{i=1}^n  ||\widehat {\bf u}_i||^2/p^2$
and $F_{||\widehat {\bf u}_i||}(y)=\chi_{0,1,n}(y)$ for all $i$.
Recall that $\widehat {\bf c}_j$ is a unit vector,
and consequently $\gamma y^2 \kappa_s^2 \ge 1$,
which probabilistically bounds the ratio 
$y/p$ from below, thus implying the desired upper bounds on 
the norms $||{\bf c}_j^-||=||\widehat C^{-1}{\bf e}_j||$ for all $j$
and consequently $||C^{-1}||=||\widehat C^{-1}||\le \sum_{j=1}^n |\widehat C^{-1}{\bf e}_j||$.
This completes our proof of part (iv) of
 Theorem  \ref{thkappa1}. 


\subsection{Extension of Theorem \ref{thkappa1} to the case of rectangular matrices}\label{serm}


Clearly part (i)
 of Theorem \ref{thkappa1}
holds for any pair of $m$ and $n$. 
By extending the concept of singularity of a matrix
to its rank deficiency, we  readily
extend parts (ii) and (iii).
Next we employ Fact \ref{faccondsub} 
and Lemma \ref{lepr3}
to extend our 
upper
bound on $\sigma_{\rho} (A)/\sigma_n (C)$
to the case where $m\neq n$.


\begin{theorem}\label{thur1} 
Suppose $A\in \mathbb R^{m\times n}$,
$U\in \mathcal G_{0,\sigma}^{m\times r}$, 
and $V\in \mathcal G_{0,\sigma}^{n\times r}$
for three 
 positive integers $m$, $n$ and $r$,
the matrix $C=A+UV^T$ has full rank $l=\min\{m,n\}$,
and  $l-r\le \rho=\rank (A)$.
Keep equation (\ref{eqstuv}) but write 
\begin{equation}\label{eqstuv1}  
I_{m,n}S^TU =     \begin{pmatrix}
                        \bar U    \\
                        U_{r}
                \end{pmatrix},
 ~ I_{m,n}T^TV =     \begin{pmatrix}
                        \bar V    \\
                        V_r
                \end{pmatrix}
\end{equation}
 for $I_{g,h}$ of (\ref{eqi})
where $U_{r}$ and $V_{r}$ still denote $r\times r$ matrices.
Keep the other assumptions of parts (a)--(e) of Theorem \ref{th5.2}.
Then the upper bound of
 part (e) of Theorem \ref{th5.2} can be extended,  
that is,
$\sigma_{\rho} (A)/\sigma_l (C)\le (1+||U||)(1+||V||)f_r$
where $f_r=\max \{1,||U_r^{-1}||\}~\max\{1,||V_r^{-1}||\}$ as in (\ref{eqpuv})
and where $U_r, V_r\in \mathcal G_{0,\sigma}^{r\times r}$.
\end{theorem}

\begin{proof}
Let $A=S_A\Sigma_AT^T_A$ be SVD
of (\ref{eqsvd}). 
Write  $\widehat C=I_{m,n}S_A^TCT_AI_{n,m}^T$, 
$\widehat U=I_{m,n}S_A^TU$, $\widehat V=I_{n,m}T_A^TV$,
$\widehat A=I_{m,n}S_A^TAT_AI_{n,m}^T$,
and so $\widehat A=(\sigma_j(A))_{j=1}^{l}$ and
$\widehat C=\widehat A+\widehat U\widehat V^T$.
Apply Theorem \ref{th5.2} to the $l\times l$ matrices 
$\widehat A$ and $\widehat C$ and
obtain that
$\sigma_{\rho} (\widehat A)/ \sigma_n (\widehat C)\le (1+||\widehat U||)(1+||\widehat V||) f_r$.
Complete the proof of Theorem \ref{thur1}
by combining this bound with the relationships 
$\sigma_{\rho}(\widehat A)=\sigma_{\rho}(A)$,
$\sigma_l(\widehat C)\le  \sigma_l(S_A^TCT_A)=\sigma_l(C)$,
$||U||=\sigma_1 (U)\ge \sigma_1 (\widehat U)=||\widehat U||$,
and $||V||=\sigma_1 (V)\ge \sigma_1 (\widehat V)=||\widehat V||$.
Here the equations hold 
by virtue of Lemma \ref{lepr2},
because the matrices $S_A$ and $T_A$
are  square and orthogonal. The inequalities hold 
by virtue of  Fact \ref{faccondsub},
because 
$\widehat C$, $\widehat U$, and $\widehat V$
are submatrices of the matrices
$S_A^TCT_A$,
$S_A^TU$, and $V^TT_A$, respectively.
\end{proof}

Combine Theorems \ref{thur} and \ref{thur1}  
to yield the  following result.


\begin{theorem}\label{thsumm} 
Assume that $A\in \mathbb R^{m\times n}$,
$U\in \mathcal G_{0,\sigma}^{m\times r}$,
$V\in \mathcal G_{0,\sigma}^{n\times r}$,
 $C=A+UV^T$, 
and
 $l=\min\{m,n\}$. Then the matrix
$C$ is rank deficient if 
 $r< l-\rank (A)$. Otherwise 
with probability $1$ the matrices $U$ and $V$ 
have full rank
and
the bound 
$\sigma_{\rho} (A)/\sigma_l (C)\le (1+||U||)(1+||V||)f_r$
of Theorem \ref{thur1} holds, 
 the norms $||U||$ and $||V||$ 
 satisfy the randomized bounds of Theorem \ref{thsignorm} (for $A=U$
and
$A=V$), and the values
$||U_r^{-1}||$ and $||V_r^{-1}||$ 
 satisfy the randomized bounds of Theorem \ref{thur}.
\end{theorem}


\begin{corollary}\label{cosumm1}
Theorem \ref{thkappa1} can be extended to 
the case of matrices $A\in \mathbb R^{m\times n}$,
 $U\in \mathcal G_{0,\sigma}^{m\times r}$ and
 $V\in \mathcal G_{0,\sigma}^{n\times r}$ 
for 
any pair of positive integers $\{m,n\}$
and    $l=\min\{m,n\}$,
that is
the matrix
$C=A+UV^T$ is rank deficient if 
 $r< l-\rank (A)$,
whereas for $r\ge l-\rank (A)$ 
it
has full rank with probability $1$
and  is expected to have condition number
of at most order
$||A||/\sigma_{l-r}(A)$
if the ratio 
$\sigma/|| A||_2$
is neither large nor small,
e.g., if
$\frac{1}{100}\le \sigma/|| A||_2\le 100$. 
Consequently
 the matrix $C$ is expected 
to be nonsingular and well conditioned if the matrix $A$ has numerical 
rank at least $l-r$.
\end{corollary} 


\subsection{Extension to and from Theorem \ref{thkappa}}\label{sextcor}


To extend Theorem \ref{thkappa1} to Theorem \ref{thkappa}
truncate the SVD of the matrix $A$ having numerical rank $\rho<n$
by setting to $0$ all its 
singular values, except for the  $\rho$ largest  ones.
This produces a well conditioned matrix $A-E$ of rank $\rho$
where $||E||=\sigma_{\rho+1}(A)$ and the ratio
$||E||/||A||$ is small because the matrix $A$ has numerical rank $\rho$.
Now to obtain Theorem \ref{thkappa} combine
 Theorem  \ref{thkappa1} (applied to the matrix $A-E$
rather than $A$)
with  Theorem \ref{thpert} and the simple bounds
$||A||-||E||\le ||A||\le ||A||+||E||$
and
$||C||-||E||\le ||C||\le ||C||+||E||$
and observe that 
an upper bound of
 at most order $||A||/\sigma_{\rho}(A)$
on $\kappa (A)$
implies that the matrix $A$ 
 having numerical rank $\rho$
is well conditioned.

We can apply Theorem \ref{thpert1}
instead of Theorem \ref{thpert} 
and similarly  extend 
the results of the previous subsection,
to rectangular matrices $A$.
To yield stronger estimates, however,
one should 
avoid using Theorem \ref{thpert1} and
instead  extend Theorem \ref{thkappa}
to the case of rectangular input 
by applying our techniques of the proof 
of Theorem \ref{thur1}.
Here is the resulting extension of Theorem \ref{thkappa}.

\begin{theorem}\label{thkappa2}
Suppose $ A$ is a real $m\times n$ matrix 
having a numerical rank $\rho$
(that is the ratio $\sigma_{\rho+1}( A)/|| A||$ is
small, but the ratio $\sigma_{\rho}( A)/|| A||$ is not small),  
the  $(m+n)r$ Gaussian random entries of two 
matrices
$U\in \mathcal G_{0,\sigma}^{m\times r}$ and
$V\in \mathcal G_{0,\sigma}^{n\times r}$,
are
independent of each other,
$ C= A+UV^T$,
$0<r<l$ 
 and  $0<\rho<l=\min\{m,n\}$.
Then 
(i) bounds of (\ref{eqnca}) hold, 
(ii)  the matrix $ C$ is 
singular or ill conditioned if $r<l-\rho$; 
(iii) otherwise it is
nonsingular with probability $1$
 and (iv)
is expected to be well conditioned if
the ratio $\sigma/|| A||_2$
is neither large nor small,
e.g., if
$\frac{1}{100}\le \sigma/|| A||_2\le 100$. 
\end{theorem} 


The next corollary slightly generalizes 
Theorem \ref{thkappa2}
to 
match 
the augmentation map of Theorem \ref{th5.2exp} of the next section.


\begin{corollary}\label{cosumm2}
Suppose that  $ A$, $U$ and $V$ denote
the same matrices as in
Theorem \ref{thkappa2},
$l$ still denotes the integer $\min\{m,n\}$,
but 
$ C= A+UMV^T$ for a normalized nonsingular $r\times r$
matrix $M$, $||M||=1$. Then 
Theorem \ref{thkappa2}  can be extended
as follows: the matrix
$ C$ is rank deficient if 
 $r< l-\rho$; otherwise
it 
has full rank with probability $1$ 
and 
is expected to have condition number
of  order
$\kappa(M)$
if
the ratio $\sigma/|| A||_2$
is neither large nor small,
e.g., if,
say
$\frac{1}{100}\le \sigma/|| A||_2\le 100$.
\end{corollary} 


\begin{proof}
Let $M=S_M\Sigma_MT_M^T$ be SVD and rewrite $ C= A+UMV^T$ as 
$ C= A+\bar U\bar V^T$
where $\bar U=US_M$ and $\bar V=\Sigma_MT_M^TV$. Note that 
 $\bar U\in \mathcal G_{0,\sigma}^{m\times r}$ and
 $T_MV\in \mathcal G_{0,\sigma}^{n\times r}$ 
by virtue of Lemma \ref{lepr3}. Now reapply the proofs of this section
replacing $U$ by $\bar U$ and $V$ by $\bar V$.
All the proofs are readily extended except for the estimates for the norm $V_r^{-1}$,
which grow by at most a factor $\kappa (\Sigma_M)=\kappa (M)$ 
by virtue of Lemma \ref{lepr1}.
\end{proof}


\begin{remark}\label{reclas}
How large is the class of $m\times n$ matrices 
having a numerical rank $\rho$?
We characterize it 
indirectly, by noting that 
by virtue of Fact \ref{far1}
the nearby matrices
of rank $\rho$
form a variety 
of dimension $(m+n-\rho)\rho$, which increases as $\rho$
increases.
\end{remark}  


\subsection{Dual additive preconditioning}\label{sdual}


For an $m\times n$ matrix $ A$ of full rank we
extend (\ref{eqc--1}) and (\ref{eqsmwd}) to
define  the {\em dual additive preprocessing} 
\begin{equation}\label{eqc-+}
 A^{+}\Longrightarrow  C_-^{+}= A^{+}+U_-V_-^T. 
\end{equation}
Our analysis implies that the value 
$\kappa ( C_-^+)$ (equal to $\kappa ( C_-)$)
 is expected to have 
at most order 
$\sigma_{q+1}( A)/\sigma_{l}( A)$
provided $l=\min\{m,n\}$, $U_-\in \mathcal G_{0,1}^{n\times q}$, 
$V_-\in \mathcal G_{0,1}^{m\times q}$, and 
 the norm $|| A^{+}||$ is neither large nor small. 
The randomized algorithm of \cite{D83}
is expected to
 estimate the norm $|| A^{+}||$
 at a low computational cost. We can  
work with the  $(m+1)\times (n+1)$
matrix $\widehat A=\diag( A,\epsilon)$
instead of the matrix $ A$ and choose a 
sufficiently small positive scalar $\epsilon$ 
such that $||\widehat A^+||=1/\epsilon$.
Then we can scale the matrix $\widehat A$
to obtain that 
$||(\widehat A/\epsilon)^+||=1$. 


\subsection{Can we weaken randomness?}\label{sweak}


Would Theorems \ref{thkappa},
\ref{thkappa1},
\ref{thkappa2}
 and other results of this section
and of the next one still hold if
we weaken randomness of
the matrices $U$ and $V$
by allowing them to be sparse and structured,
to share
some or all their entries,
 or  generally 
to be defined by a smaller 
number of independent parameters,
possibly under other probability distributions
rather than Gaussian?
We have some progress with our analytical study in this direction 
(see Sections \ref{scgrtm} and \ref{scgrcm} and Remark \ref{reu=v}),
but empirically all the presented randomized techniques
remain as efficient under very weak randomization
in the above sense (cf. Tables \ref{tabprec}, 
\ref{tab44}, \ref{tabSVD_HEAD1T}, and \ref{tabhank}).




\section{Randomized augmentation}\label{saug}


\subsection{Augmentation and an extension of the SMW formula}\label{ssmwaug}


The solution of a nonsingular linear system of $n$ 
equations, $A{\bf y}={\bf b}$ can be readily recovered from a null vector
 $\begin{pmatrix}
   -1/\beta  \\
   {\bf y} 
  \end{pmatrix}
$ of the matrix $K=(\beta{\bf b}~|~A)$ for a nonzero scalar $\beta$. If the matrix $A$
has numerical nullity $1$ and if the ratio $||A||/||\beta{\bf b}||$ is neither large 
nor small, then the matrix $K$ is 
well conditioned for the average vector ${\bf b}$ \cite[Section 13.1]{PQa}.
The above map $A\Longrightarrow K$ is a special case of more general
augmentation
\begin{equation}\label{eqnwaug}
K=
\begin{pmatrix}
W   &   V^T   \\
-U   &   A
\end{pmatrix},
\end{equation}
 which we study next, beginning with the following extension of
the SMW formula. 

\begin{theorem}\label{thsmwau}
Suppose equation (\ref{eqnwaug}) holds,  $m=n$  and
the matrices $A$, $W$ and $K$ are nonsingular.  
Write $S=A+UW^{-1}V^T$ and $R=I-V^TS^{-1}UW^{-1}$.
Then  the matrix $S$ is nonsingular, $S^{-1}$ is the 
trailing (southwestern) $n\times n$ block of $K^{-1}$,
 and
\begin{equation}\label{eqaugsmw}
A^{-1}=S^{-1}+S^{-1}UW^{-1}R^{-1}V^TS^{-1}.
\end{equation}
\end{theorem}  
 \begin{proof}
Apply the
SMW formula of Theorem \ref{thdsmw} for $C$ replaced by $S$, $U$ by $UW^{-1}$,
and $G$ by $R$.
 \end{proof} 


\subsection{Links to additive preprocessing
and condition estimates}\label{sauggen}


In contrast to the scaled randomized symmetric 
additive preprocessing $A\Longrightarrow C=A+VV^T$
(cf. (\ref{equv}) and \cite{W07}),
the map $A\Longrightarrow K=\begin{pmatrix}
W   &   V^T   \\
V   &   A
\end{pmatrix}$ 
  cannot decrease the condition number 
$\kappa (A)$  if
$K$ is a symmetric and positive definite matrix; 
this follows from
 the Interlacing Property of the eigenvalues of $K$
\cite[Theorem 8.6.3]{GL96}.
Nonetheless 
the following simple theorem links additive  
preprocessing $A\Longrightarrow  C=A+UMV^T$ to
the augmentation  $A\Longrightarrow K$ for $K$ of (\ref{eqnwaug})
and later we extend
Theorem \ref{thkappa2}
to the augmentation as well.


\begin{theorem}\label{th5.2exp}
Suppose  $A\in \mathbb R^{m\times n}$,
$W\in \mathbb R^{r\times r}$,
the matrix $W$ is nonsingular, $l=\min\{m,n\}$, 
a matrix
$K$
in $\mathbb R^{(m+r)\times (n+r)}$
 is defined by (\ref{eqnwaug}),
 and $C=A+UW^{-1}V^T$. 
Then we have
\begin{equation}\label{eqk}
K=\widehat U\diag(C,I_r) \widehat V\diag(W,I_n)
\end{equation}
for $\widehat U=\begin{pmatrix}
O_{r,m}  &   I_r  \\
I_m   &  -UW^{-1} 
\end{pmatrix}$,
$\widehat V=\begin{pmatrix}
  O_{n,r} &  I_n    \\
 I_r     &  V^T 
\end{pmatrix}$, 
the matrix $C$ has full rank if and only if the matrix $K$ has full rank,
and so both matrices are rank deficient for $r<l$.
Furthermore 
$\widehat U^{-1}=\begin{pmatrix}
  UW^{-1}  &  I_n   \\
 I_r  &  O_{r,n} 
\end{pmatrix}$,  
$\widehat V^{-1}=\begin{pmatrix}
-V^T  &   I_r  \\
I_n   &   O_{n,r}
\end{pmatrix}$.
For $m=n$ and  nonsingular matrices $C$ and $K$,
we have $C^{-1}=(I_n~|~O_{n,r})\widehat V\diag (W,I_n) K^{-1}\widehat U(I_n~|~O_{n,r})^T$ and
$K^{-1}= \diag (W^{-1},I_n)\widehat V^{-1}\diag(C^{-1},I_r)\widehat U^{-1}$.
 \end{theorem}


\begin{corollary}\label{cosumm3} (Cf. \cite[Remark 10.1 and Corollary 11.1]{PQa}.)
Define three integers $m$, $n$, and $l$ and three matrices 
$A$, $K$, and $W$ as in Theorem \ref{th5.2exp},
write $h=\max\{m,n\}$, and
suppose that 
$U\in \mathcal G_{0,\sigma}^{m\times r}$
and $V\in \mathcal G_{0,\sigma}^{n\times r}$.
Then (i) $||A||\le ||K||\le ||A||+||U||+||V||+||W||$, and
(ii) the matrix $K$ is rank deficient if
$r< l-\rank(A)$ but
 has full rank
with probability $1$ otherwise.
(iii) Furthermore suppose $r\ge l-\rank(A)$
and the ratio $\sigma/|| A||_2$
is neither large nor small,
e.g., 
say
$\frac{1}{100}\le \sigma/|| A||_2\le 100$.
Then
the matrix $K$
is expected to have condition number
of order $(1+2h)^4\kappa(W)^2/\sigma_{l-r}(A)$, that is
(iv)  of order $(1+2h\sigma)^4/\sigma_{l-r}(A)$
provided  $W\in \mathcal G_{0,1}^{r\times r}$.
\end{corollary}
\begin{proof}
Part (i) is verified immediately.
Next estimate the rank of the matrix $C$ as in parts (ii) and (iii)
of  Theorem \ref{thkappa} and  
apply equation (\ref{eqk}) to extend the estimates to
the rank of the matrix $K$.
This proves part (ii).
Equation (\ref{eqk}) and Corollary \ref{copr}
together imply that 
$\kappa(K)\le \kappa(\widehat U)\kappa(\widehat V)\kappa(C)\kappa(W)$.
(We can apply Corollary \ref{copr} because the matrices $\widehat U$,
$\widehat V$ and $W$ are nonsingular.)
We have $\kappa(\widehat U)\le (1+||\widehat U||)^2$ and
 $\kappa(\widehat V)\le (1+||\widehat V||)^2$,
and we can expect that  
$\max\{||\widehat U||,||\widehat V||\}\le 1+2h\sigma$ 
for $h=\max\{m,n\}$
by virtue of Theorem \ref{thsignorm}.
Now apply 
Corollary \ref{cosumm2}
for $M=W^{-1}$ to bound $\kappa(C)$ and
recall that $\kappa(W^{-1})=\kappa(W)$.
Combining these estimates proves part (iii).
To extend part (iii) to part (iv) note that
 a matrix $W$ in $\mathcal G_{0,1}^{r\times r}$ is 
nonsingular with  probability $1$ and is expected to be  
well conditioned (see Sections \ref{sngrm} and \ref{scgrm}).
\end{proof}


\subsection{Direct condition estimation: Gaussian random leading blocks}\label{saugdir}

To obtain sharper bounds and better insight into the  subject, 
let us estimate the condition number $\kappa(K)$
directly, without  
reducing this task to additive preprocessing.
Some initial study of randomized augmentation 
in this direction
can be found in
 \cite{PQa}. 
In particular the results 
of \cite[Corollary 11.1]{PQa} are
similar to Theorem \ref{thaugkp},
but  \cite{PQa}
only provides a pointer to the idea of a proof.
Part (i) of Corollary \ref{cosumm3} is
extended immediately, and next we extend the other parts.


\begin{theorem}\label{thaugkp}
Suppose  
a real normalized $m\times n$ matrix $A$ has a  rank $\rho<n$,
$U\in \mathcal G_{0,1}^{m\times q}$,
$V\in \mathcal G_{0,1}^{n\times s}$,
$W\in \mathcal G_{0,1}^{s\times q}$, 
$K$
in $\mathbb R^{(m+s)\times (n+q)}$
  defined by (\ref{eqnwaug}),
 $l=\min\{m,n\}$, $r=\min\{m-q,n-s\}>0$. 
Then 

(i) the matrix $K$ is rank deficient
if $\rho<r$, 

(ii) otherwise
the matrix has full rank $l'=\min\{m+s,n+q\}$ with probability $1$ 
and 

(iii) is expected to
have the condition number $\kappa(K)$
of order at most $1/\sigma_{r}(A)$.
 \end{theorem}

Thus we can expect that
 the matrix $K$ has full rank and
is well conditioned if $\rho\ge r$.

\begin{proof}
Assume that the entries of the matrices $U$,
$V$, and $W$ are indeterminates.
Then clearly the matrix $(-U~|~A)$ has full rank, that is has $m$
linearly independent rows,
if and only if $\rho+q\ge m$.
Likewise the matrix $\begin{pmatrix}
 V^T      \\
 A      
\end{pmatrix}$ has full rank, that is has $n$
linearly independent columns,
if and only if $\rho+s\ge n$.
The transition from these
matrices to the matrix $K$
increases the numbers of linearly independent rows by $s$ 
and columns by $q$.   
Summarizing we obtain parts (i) and (ii)
provided that the entries of the matrices 
$U$, $V$, and $W$ are indeterminates.
Relax this 
assumption by applying Lemma \ref{ledl}.

Next assume that $\rho+s+q\ge l'=\rank (K)$ 
and estimate the condition number $\kappa(K)=||K'||/\sigma_{l'}(K)$.
By virtue of Theorem \ref{thsignorm}
we can expect that
the  norms of the matrices $U$, $V$, and $W$
 are in $O(1)$,
that is do not exceed a fixed constant,
and so $||K||=O(1)$ as well
because $||K||\le ||U||+||V||+||W||+||A||$ and $||A||=1$.
It remains to estimate the value $\sigma_{l'}(K)$ from below.
We can assume that
$l'=m+s\le n+q$, and so $r=m-q$,
for otherwise we can estimate $\kappa(K^T)=\kappa(K)$.

At first let $s=0$. 
Then  $l'=\rank (K)=m\le n+q$,
$K=(-U~|~A)$,
 and
$V$ and $W$
are empty matrices.
Reuse and extend the idea of Section \ref{serm},
that is reduce the original task to
the case of an $m\times m$ 
submatrix $\bar K$ of the matrix $K$,
which is nonsingular with probability $1$
and for which we have $\sigma_m(K)=\sigma_m(\bar K)$;
then estimate the value $\sigma_m(\bar K)$ as 
the reciprocal $1/||\bar K^{-1}||$.
Namely assume the SVD
 $A=S_A\Sigma_AT_A^T$  of (\ref{eqsvd})
and write
$K'=S_A^TK\diag(I_q,T_A)=(U'~|~\Sigma_A)$.
Note that $S_A$, $T_A$ and 
$\diag(T_A,I_q)$ are square orthogonal matrices
and infer that 
 $\sigma_{l'}(K)=\sigma_{m}(K')$
by virtue of Lemma \ref{lepr2}, whereas
$U'=S_A^TU\in \mathcal G_{0,1}^{m\times q}$
by virtue of Lemma \ref{lepr3}.
The $m\times (n+q)$ 
matrix  $K'$ has the $m\times m$ 
leading submatrix 
$\bar K=\begin{pmatrix}
 U_0 & \Sigma_{m-q}      \\
 U_1   & O_{q,m-q}   
\end{pmatrix}$ where 
$U_0\in \mathcal G_{0,1}^{(m-q)\times q}$,
$U_1\in \mathcal G_{0,1}^{q\times q}$,
 $\rank (\bar K)=\rank (K)=m$,
$\Sigma_{m-q}=\diag(\sigma_{j}(A))_{j=1}^{m-q}$,
and so $\rank (\Sigma_{m-q})=m-q$
and $\sigma_{l'}(K)=\sigma_{m}(K)\ge\sigma_{m}(\bar K)=1/||\bar K^{-1}||$.
We have 
$$\bar K^{-1}=
\begin{pmatrix}
 O_{q,m-q} & U_1^{-1}      \\
 \Sigma_{m-q}^{-1}   &    -\Sigma_{m-q}^{-1}U_0 U_1^{-1}
\end{pmatrix}=
\diag(I_q,\Sigma_{m-q}^{-1})
\begin{pmatrix}
 O_{q,n} & I_q      \\
 I_{m-q}   &      -U_0 
\end{pmatrix}
\diag(I_{m-q}, U_1^{-1}).$$
\noindent Therefore
$||\bar K^{-1}||\le ||\Sigma_{m-q}^{-1}||(1+||U_0||)||U_1^{-1}||$
where $||\Sigma_{m-q}^{-1}||=1/\sigma_{m-q}(A)$.
Theorems \ref{thsiguna} and \ref{thsignorm} together
bound the norms $||U_1^{-1}||$ and $||U_0||$,
implying that the value $1/\sigma_{l'}(K)=||\bar K^{-1}||$
is expected to have at most order $1/\sigma_{m-q}(A)$.

Now let $s>0$.
Then again we reduce our task to the case 
of a square matrix $\widehat K\in \mathbb R^{l'\times l'}$
(where $l'=m+s$)
such that $\sigma_{j}(K)\ge \sigma_{j}(\widehat K)$
for all $j$ and
then estimate the value $\sigma_{l'}(\widehat K)$
as 
the reciprocal $1/||\widehat K^{-1}||$.
Namely represent the matrix $K$ as $\begin{pmatrix}
 B     \\
 F 
\end{pmatrix}
$ 
where $B=(W~|~V^T)$, $F=(-U~|~A)$,
and
the value $||F^+||$ has at most order $1/\sigma_{m-q}(A)$,
as we proved above. 
Let $F=S_F\Sigma_FT_F^T$ be SVD and write
$K''=\diag(I_s,S_F^T)KT_F=
\begin{pmatrix}
 B_{0} & B_1      \\
 \widehat \Sigma_F   &      O_{m,n+q-m} 
\end{pmatrix}
$ where 
$B_0\in \mathcal G_{0,1}^{s\times m}$,
$B_1\in \mathcal G_{0,1}^{s\times (n+q-m)}$,
$\widehat \Sigma_F=\diag(\sigma_j(F))_{j=1}^{m}$,
$\rank (K'')=\rank (K)=l'$, and so 
the matrix $\widehat \Sigma_F$ is nonsingular
and $||\widehat \Sigma_F^{-1}||=||F^+||$. 
We have $\sigma_j(K'')=\sigma_j(K)$ for all $j$
because the matrices $\diag(I_s,S_F^T)$ and $T_K$
are  square and orthogonal.

Delete the last $n+q-m-s$ columns of the matrix $K''$
and obtain the $l'\times l'$ submatrix
$\widehat K=\begin{pmatrix}
  B_{0} &  \bar B_1      \\
 \widehat  \Sigma_F   &      O_{m,n+q-m} 
\end{pmatrix}
$. 
We have $\sigma_{l'}(K)=\sigma_{l'}(K'')\ge \sigma_{l'}(\widehat K)$.
The Gaussian random $s\times s$ matrix $\bar B_1$
is nonsingular with probability $1$. We assume that it is
nonsingular, and then so is the matrix $\widehat K$ as well,
and consequently $\sigma_{l'}(K)= \sigma_{l'}(\widehat K)=
1/||\widehat K^{-1}||$.
Observe that
$$\widehat K^{-1}=
\begin{pmatrix}
 O_{m,n+q-m} & \widehat  \Sigma_F^{-1}      \\
 \bar B_{1}^{-1}   &   - \bar B_{1}^{-1} B_0 \widehat  \Sigma_F^{-1}
\end{pmatrix}=
\diag(I_q,\bar B_{1}^{-1})
\begin{pmatrix}
 O_{q,n} & I_q      \\
 I_{m-q}   &      -B_0 
\end{pmatrix}
\diag(I_{m-q},\Sigma_F^{-1}).$$
\noindent Therefore $||\widehat K^{-1}||\le||\bar B_{1}^{-1}||(1+||B_0||)||\widehat \Sigma_F^{-1}||$.
Theorems \ref{thsiguna} and \ref{thsignorm}  together
bound the norms $||\bar B_1^{-1}||$ and $||B_0||$.
To complete the proof of the theorem
recall that $||\widehat \Sigma_F^{-1}||=||F^+||$
and that the norm $||F^+||$
is expected to have at most order $1/\sigma_{m-q}(A)$.
\end{proof}

Compared to Corollary \ref{cosumm3},
Theorem \ref{thaugkp} allows rectangular matrices $W$ in $\mathcal G^{s\times q}$.
Combined with Theorem \ref{th5.2exp} it implies 
Corollary \ref{cosumm2} restricted to the case where
$M^{-1}\in \mathcal G_{0,1}^{s\times s}$. 
In the next subsection we 
extend Theorem \ref{thaugkp} by relaxing
this restriction 
provided that  $s=q$ and allowing any 
well conditioned
block $W$
with the norm not exceeding $1$.


\subsection{Direct condition estimates: well conditioned leading blocks}\label{saugalt}


Next we outline
a direct proof 
of Corollary \ref{cosumm3}
allowing any scaled well conditioned square leading blocks $W$.
The supporting estimates 
are 
stronger than the ones deduced via
  combining
Corollary \ref{cosumm2} and
Theorem \ref{th5.2exp}. We begin with providing 
all details 
in the case where $m=n$ and $r=l-\rank(A)$.
In this case we allow 
ill conditioned  blocks $W$.
 
\begin{theorem}\label{thnwaug} 
Suppose $n$ and $r$ are 
two
positive integers,
 a real normalized $n\times n$ matrix $A$ has a  rank $\rho<n$,
$U,V\in \mathcal G_{0,1}^{n\times r}$, 
 $W\in \mathbb R^{r\times r}$, 
   $||W||\le 1$, and $K$ denotes the matrix of (\ref{eqnwaug}).
Then (i) the matrix $K$ is
singular or ill conditioned if $r<n-\rho$. Otherwise it is
 nonsingular with probability $1$.
Furthermore (ii) if  
 $r=n-\rho$,
then the condition number $\kappa(K)$ is expected to have
at most  order $||A||/\sigma_{n-r}(A)$.
\end{theorem}
\begin{proof}
We proceed similarly to the proof of Theorem \ref{thaugkp}.
Suppose $A=S_A\Sigma_AT_A^T$ is the SVD of (\ref{eqsvd})
and write
$\bar K=\diag(I_r, S_A^T)K\diag(I_r, T_A)$.
Then 
$$\bar K=
\begin{pmatrix}
W   &   \bar V^T   \\
-\bar U   &  \Sigma_A
\end{pmatrix}$$
where $\sigma_j(K)=\sigma_j(\bar K)$ for all $j$ and
$\bar U,\bar V\in \mathcal G_{0,1}^{n\times r}$.
Furthermore write 
$\Sigma_{\rho}=\diag(\sigma_j(A))_{j=1}^{\rho}$,
$\Sigma_A=\diag(\Sigma_{\rho},O_{n-\rho,n-\rho})$,
$\bar U=\begin{pmatrix}
U_0     \\
U_1   
\end{pmatrix}$
and 
$\bar V=\begin{pmatrix}
V_0     \\
V_1   
\end{pmatrix}$
where
$U_0,V_0\in \mathcal G_{0,1}^{\rho\times r}$
and $U_1,V_1\in \mathcal G_{0,1}^{(n-\rho)\times r}$
and obtain 
\begin{equation}\label{eqaug3bl}
\bar K=
\begin{pmatrix}
W   &   V_0^T  & V_1^T \\
U_0   &  \Sigma_{\rho}  & O_{\rho,n-\rho}  \\
U_1   &   O_{\rho,n-\rho}  &  O_{n-\rho,n-\rho}
\end{pmatrix}.
\end{equation}
Now we can readily verify the claims about $\rank (K)$.
It remains to estimate the condition number 
$\kappa(K)=||K||~||K^{-1}||=||\bar K||~||\bar K^{-1}||$ provided
that $r=n-\rho$ and the matrix $K$ is nonsingular.
To bound the norm $||K||$, 
note that 
$K=
\begin{pmatrix}
W   &   O_{r,n}  \\
O_{n,r}  &   A
\end{pmatrix}
+\begin{pmatrix}
O_{r,n}  &   I_{r,r}  \\
I_{n,n} &   O_{n,r}  
\end{pmatrix}
\begin{pmatrix}
-U   &   O_{n,n}  \\
O_{r,r} &   V^T   
\end{pmatrix}$,
recall that $||W||\le||A||=1$, 
apply bound (\ref{eqnormdiag})
and 
obtain 
$$||K||\le 1+\max\{||U||,||V||\}.$$
By virtue of randomized bounds of
Theorem \ref{thsignorm} we expect to have 
the norms $||U||$ and $||V||$ in $O(1)$,
that is bounded by a constant,
and then $||K||$ is in $O(1)$ as well.

We conclude the proof by estimating the norm
$||\bar K||=||K||$.
We readily verify that 
$$\bar K^{-1}=
\begin{pmatrix}
O_{r,r}   &   O_{r,n-r}  & -U_1^{-1} \\
O_{n-r,r}   &  \Sigma_{n-r}^{-1}  & -\Sigma_{n-r}^{-1}U_0U_1^{-1}  \\
V_1^{-T}   &  -V_1^{-T}V_0^T\Sigma_{n-r}^{-1}    &  V_1^{-T}(W+V_0^T\Sigma_{n-r}^{-1}U_0)U_1^{-1}
\end{pmatrix}.$$

Apply  bound (\ref{eqnormdiag}) and deduce that
$$ || K^{-1}||=||\bar K^{-1}||\le N_1+N_2+N_3$$
where $N_1=\max \{||V_1^{-T}||,||\Sigma_{n-r}^{-1}||,||U_1^{-1}||\}$,
$N_2=\max \{||V_1^{-T}V_0^T\Sigma_{n-r}^{-1}||,||\Sigma_{n-r}^{-1}U_0U_1^{-1}||\}$
and 
$N_3=|| V_1^{-T}(W+V_0^T\Sigma_{n-r}^{-1}U_0)U_1^{-1}||$.
Recall that $||W||\le 1$, $||\Sigma_{n-r}^{-1}||=1/\sigma_{n-r}(A)$, $||V_0^T||=||V_0||$, 
and $||V_1^{-T}||=||V_1^{-1}||$
and deduce that

$$N_1=\max \{||V_1^{-1}||,||U_1^{-1}||,1/\sigma_{n-r}(A)\},$$

$$N_2\le\max \{||V_1^{-1}||~||V_0||,||U_0||~||U_1^{-1}||\}/\sigma_{n-r}(A),$$

$$N_3\le|| V_1^{-1}||~||U_1^{-1}||(1+||V_0||||U_0||/\sigma_{n-r}(A)).$$

Apply Theorems  \ref{thsiguna} and \ref{thsignorm} 
to estimate the norms $||U||$, $||V||$, $||V_0||$, $||U_0||$, $||U_1^{-1}||$ 
and $||V_1^{-1}||$. 
Combine all the above bounds to estimate
the norm $||K^{-1}||$.
\end{proof}


Next we extend Theorem \ref{thnwaug} to the case where $r\ge n-\rho$
and where we require that the leading block 
$W$ be normalized, square and well conditioned.

\begin{theorem}\label{thnwaug0} 
 Theorem \ref{thnwaug} still holds where
$r> n-\rho$ and 
the leading block $W$ of the matrix $K$ is 
normalized, square and well conditioned.
\end{theorem}


\begin{proof} 
Clearly parts (i) and
(ii) of Theorem \ref{thnwaug} are extended,
and moreover
we immediately deduce that for $r> n-\rho$ the matrix
$K$ is
 nonsingular with probability $1$ and that 
its norm
$||K||$
is expected to be in $O(1)$.
It remains to estimate the norm $||K^{-1}||$.
Assume SVDs 
$W=S_W\Sigma_WT_W^T$ and 
$A=S_A\Sigma_AT_A^T$,
write 
$\bar K=\diag(S_W^T,S_A^T)K\diag(T_W,T_A)$
and observe that
$\bar K=
\begin{pmatrix}
\Sigma_W   &   \bar  V^T  \\
\bar U  &   \Sigma_A
\end{pmatrix}$
where 
$\bar U,\bar V\in \mathcal G_{0,1}^{\rho\times r}$
by virtue of Lemma \ref{lepr3}, because the matrices
$S_W$, $T_W$, $S_A$, and $T_A$ are square and orthogonal.
Now complete the proof by extending the techniques 
used  in Section \ref{spmt} in the proof of part (iv)
of Theorem \ref{thkappa},
in the case where $r>n-\rho$. 
 \end{proof}


\begin{corollary}\label{coaugkappamn1} 
Keep the assumptions of Theorem \ref{thnwaug0},
except that now let the matrix $A$ have numerical rank $\rho$,
rather than rank $\rho$.
Then
 Theorem \ref{thnwaug0}
is extended and furthermore 
 the matrix $K$ is expected to  be
well conditioned.
\end{corollary}
\begin{proof} 
We immediately extend parts (i) and (ii)
of Theorem \ref{thnwaug}.
To extend part (iii) 
truncate the SVD of the matrix $A$ by setting to $0$ 
all its singular values except for the $\rho$
largest ones to obtain matrices $A-E\approx A$
of rank $\rho$ 
and $\widehat K=\begin{pmatrix}
W   &   V^T  \\
-U  &   A-E
\end{pmatrix}$
such that
$||\widehat K-K||\le ||E||= \sigma_{\rho+1}(A)$.
The value $\sigma_{\rho+1}(A)$ is small
because the matrix $A$ has numerical rank $\rho$,
whereas the norm $||\widehat K^{-1}||$ is not 
expected to be 
large 
by virtue of
Theorem \ref{thnwaug0}.
Therefore we can expect that
 $||E \widehat K^{-1}||\le 1/3$,
and consequently that
$||K^{-1}||\le 1.5 ||\widehat K^{-1}||$
by  virtue of
Theorem \ref{thpert}.
Consequently
the condition number
$\kappa(K)$ is also expected to be of at most
 order $||A||/\sigma_{l-r}(A)=1/\sigma_{n-r}(A)$
as in
 part (iii)
of
Theorem \ref{thnwaug0},
but in the corollary this means that 
the matrix is
well conditioned 
because we  assume that
the matrix $A$ has numerical rank $\rho$, and so
the ratio $||A||/\sigma_{\rho}(A)$
is not 
large.
 \end{proof}


The corollary implies that the matrix 
$K$ is nonsingular with probability $1$ and is
expected to be well conditioned in the case where 
the matrix $A$ has
a numerical rank at least $n-r$.
Let us extend our analysis  to the case of rectangular 
matrices $A\in \mathbb R^{m\times n}$.


\begin{theorem}\label{thaugkappamn} 
Suppose $m$, $n$, and $r$ are 
three 
positive integers, $l=\min\{m,n\}$,
 $A\in \mathbb R^{m\times n}$, 
the matrix $A$ has a
numerical rank 
$\rho<l-r$, 
$U\in \mathcal G_{0,1}^{m\times r}$, 
$V\in \mathcal G_{0,1}^{n\times r}$,
 $W\in \mathbb R^{r\times r}$, 
   $||W||\le 1$,
$K$ is the
$(m+r)\times (n+r)$
matrix defined by equation (\ref{eqnwaug}). 
Then (i) this matrix is rank deficient or ill conditioned if 
 $\rho<l-r$, but otherwise
 has full rank with probability $1$
 and (ii) is expected to be well
 conditioned.
\end{theorem}
\begin{proof}
Part (i) is readily verified. Let us prove part (ii).
Suppose $A=S_A\Sigma_AT_A^T$ is the SVD of (\ref{eqsvd})
and write $\bar K'=\diag(I_r, S_A^T)K\diag(I_r, T_A)=\begin{pmatrix}
W   &   \bar V^T  \\
\bar U  &   \Sigma_A
\end{pmatrix}$ where 
$\bar U=-S_A^TU$ and $\bar V^T=V^TT_A$.
Observe that 
$\kappa (K)=\kappa (\bar K')$,
$\bar U\in \mathcal G_{0,1}^{m\times r}$,
$\bar V\in \mathcal G_{0,1}^{n\times r}$,
and
$\kappa (K)=\kappa (\bar K')$
because $S_A$ and $T_A$ are square orthogonal matrices. 
Define the leading $l \times l$ submatrix 
$\widehat K'=I_{m+r,n+r}\bar K' I_{n+r,m+r}^T$
for
$I_{g,h}$
of (\ref{eqi})
and observe that 
$\bar K'= \begin{pmatrix}
\widehat K'\\
\widehat U
\end{pmatrix}$, 
$\widehat U=(U_2~|~O_{m-n,n})$, $U_2\in \mathcal G_{0,1}^{(m-n)\times q}$ if $m\ge l=n$,
whereas 
$\bar K'= (\widehat K'~|~\widehat V^T)$, $\widehat V^T=(V_2^T~|~O_{n-m,m})$, 
$V_2\in \mathcal G_{0,1}^{(n-m)\times r}$ if $n\ge l=m$.
Clearly $\sigma_l (\bar K')\ge \sigma_l (\widehat K')$ (cf. Fact \ref{faccondsub})
and $||\bar K'||\le ||\widehat K'||+||F||$ for $F=U_2$ or $F=V_2$. In both cases
$F\in \mathcal G_{0,1}^{|n-m|\times r}$, and so we can expect that $||\bar K'||=O(||\widehat K'||)$
because $||\widehat K'||\ge ||\Sigma_A||=||A||=1$. 
Corollary \ref{coaugkappamn1}
implies that with probability $1$
the $(l+r)\times (l+r)$ matrix $\widehat K'$ is nonsingular, and then
$\rank (K)=\rank (\bar K')=\rank (\widehat K')=l+r$,
implying that the matrix $K$ has full rank.
Furthermore 
Corollary \ref{coaugkappamn1}
implies that the matrix 
$\widehat K'$ is expected to be
well conditioned.
It remains to extend this 
property
to the matrix $K$.
Recall that $\kappa (\widehat K')=||\widehat K'||/\sigma_l (\widehat K')$
and 
$\kappa (K)=\kappa (\bar K')=||\bar K'||/\sigma_l (\bar K')$
and combine the above equations with the bounds 
$\sigma_l (\bar K')\ge \sigma_l (\widehat K')$ and 
$||\bar K'||=O(||\widehat K'||)$, deduced earlier.
\end{proof}


\subsection{A randomized Toeplitz solver}\label{shank}


Let us apply Theorem \ref{thgs} to support randomized augmentation for  
solving a nonsingular Toeplitz linear system $T{\bf y}={\bf b}$
of $n$ equations provided the matrix $T$ has numerical nullity $1$. 

To compute the vector
${\bf y}=T^{-1}{\bf b}$, we first embed the matrix $T$ into 
a Toeplitz  $(n+1)\times (n+1)$ matrix $K=
\begin{pmatrix}
w  &   {\bf v}^T   \\
{\bf f}   &   T
\end{pmatrix}$. We write $w={\bf e}_1^T T{\bf e}_1$ and 
fill the vectors ${\bf f}=(f_i)_{i=1}^n$ 
and ${\bf v}=(v_i)_{i=1}^n$ with appropriate entries of the matrix $T$
except for the two coordinates $f_n$ and $v_n$, which we choose at 
random and then scale to have the ratios $\frac{|f_n|}{||K||}$
and  $\frac{|v_n|}{||K||}$ neither 
large nor small. 

 Part (b) of Theorem \ref{thgs} expresses the inverse $T^{-1}$ via the 
vectors ${\bf v}=K^{-1}{\bf e}_1$
and ${\bf w}=K^{-1}{\bf e}_{n+1}$.

In view of Section \ref{sngrm} and Appendix A,
 this policy is likely to produce a nonsingular matrix $K$
 whose inverse is likely to have a nonzero entry ${\bf e}_1^TK^{-1}{\bf e}_1$.   
In good accordance with 
these formal results 
our tests have always produced nonsingular and
well conditioned  
matrices $K$ such that ${\bf e}_1^TK^{-1}{\bf e}_1\neq 0$.

To summarize, we reduce the solution of a nonsingular 
ill conditioned Toeplitz linear system $T{\bf y}={\bf b}$ to computing highly accurate solutions of two
 linear systems $K{\bf x}={\bf e}_1$ and $K{\bf z}={\bf e}_{n+1}$,
both expected to be well conditioned.
High accuracy shall counter the magnification of the input and rounding errors,
expected in the case of ill conditioned input. 

In the important special case where a Toeplitz matrix $T$ is real symmetric,
we choose  real scalars $w$ and $f_n=v_n$ to yield a
real symmetric matrix $K=
\begin{pmatrix}
w  &   {\bf v}^T   \\
{\bf v}   &   T
\end{pmatrix}$. In this case 
$J_{n+1}K^{-1}J_{n+1}=K^{-1}$, 
 and so
 $K^{-1}{\bf e}_{n+1}=J_{n+1}K^{-1}{\bf e}_1$
because
$J_{n+1}{\bf e}_{n+1}={\bf e}_1$.
Thus we only need to solve a single linear system with the matrix $K$.
For the transition back to the solution of
the original problem, we can employ expression (\ref{eqaugsmw}) 
or Theorem \ref{thgs}.
Hereafter we refer to the resulting algorithm 
for the  linear system $T{\bf y}={\bf b}$
as {\bf Algorithm 6.1}. 
In Section \ref{sexphank} we test this algorithm
for solving an ill conditioned real symmetric Toeplitz linear system. 


One can readily extend the approach of this section to
the case of Toeplitz-like, Hankel and Hankel-like inputs
and to augmenting the input matrix with $r$ rows and $r$ columns 
for $r>1$.


\section{Low-rank 
approximation, approximation of singular spaces, and
computation of numerical rank}\label{sapsr}


\subsection{Randomized low-rank approximation: an outline and an extension
to approximation by structured matrices}\label{sprb0}


Our next theorem expresses
a rank-$\rho$ approximation to a matrix 
$ A$ through
an approximate matrix basis 
for the left or right leading
singular space $\mathbb T_{\rho, A}$
or $\mathbb S_{\rho, A}$.
We can  obtain
such basis 
by computing the
SVD of the matrix 
 $ A$  or
 its rank-revealing 
 factorization \cite{GE96}, \cite{HP92}, \cite{P00a},
but if 
the matrix $A$ has a small numerical rank $\rho$ and 
if we are given its
reasonably small upper bound
$\rho_+$, then
with  a probability near $1$
we can 
compute such basis
 at a
low cost
   from the
product $ A^TG$ for 
 $G\in \mathcal G_{0,1}^{m\times \rho_+}$.
 Theorem \ref{thover} of Section \ref{sprb2}
formally supports 
correctness of
the respective randomized
algorithm, but our tests support it
consistently even
where $G\in \mathcal T_{0,1}^{m\times \rho_+}$
(see Tables \ref{tabSVD_HEAD1} and \ref{tabSVD_HEAD1T}),
and we conjecture that the same is  true for 
various other classes of sparse and structured matrices
$G$ defined by fewer random parameters.
We specify a low-rank approximation algorithm in Section \ref{sapsr0},
its amendments in Section \ref{samend}, and
some related randomized algorithms 
of independent interest
for the approximation
of leading and trailing singular spaces of an
ill conditioned matrix in Sections \ref{saptr} and \ref{saplddl}.
By applying low-rank approximation algorithms  
to a displacement of a  matrix $W$ 
having a possibly unknown numerical 
displacement rank $d$,
 that is lying near some matrices with 
displacement rank $d$, we can
approximate the matrix $W$ by one of these matrices
and output $d$ as by-product.
In Section \ref{snewt}
we apply this observation
to Newton's structured matrix inversion.


\subsection{The first basic theorem:
low-rank approximation via the basis of 
a leading singular space}\label{sprb}


The following theorem expresses 
a rank-$q$ approximation
(within an error norm $\sigma_{q+1}(A)$)
to a matrix $A$ 
through a
matrix basis of 
its leading singular space $\mathbb T_{q,A}$
or $\mathbb S_{q,A}$.
 

\begin{theorem}\label{thsng} 
Suppose  $A$ is
an $m\times n$ matrix,
$S_A\Sigma_AT_A^T$ is its SVD of (\ref{eqsvd}), 
 $q$ is a positive integer,
$q\le\min\{m,n\}$, and
$T$ and $S$ are 
matrix 
bases for the spaces $\mathbb T_{q,A}$
and $\mathbb S_{q,A}$, respectively.
Then 
\begin{equation}\label{eqlrap}
||A-AT(T^TT)^{-1}T^T||=||A-S(S^TS)^{-1}S^TA||=\sigma_{q+1}(A).
\end{equation}
For orthogonal matrices $T$ and $S$ we have $T^TT=S^TS=I_q$
and  
\begin{equation}\label{eqlrap1}
||A-ATT^T||=||A-SS^TA||=\sigma_{q+1}(A).
\end{equation}
\end{theorem}
\begin{proof}
Let us first 
write $P=T_{q,A}T_{q,A}^T$ and $r=n-q$
and estimate the norm $||A-AP||$. 
We have $AP=S_A\Sigma_AT_A^TT_{q,A}T_{q,A}^T$.
Substitute $T_A^TT_{q,A}=\begin{pmatrix}
I_{q}  \\
O_{r,q}
\end{pmatrix}$
and obtain $AP=S_A\Sigma_A \begin{pmatrix}
T_{q,A}^T  \\
O_{r,q}
\end{pmatrix}$, whereas 
$A=S_A\Sigma_A \begin{pmatrix}
T_{q,A}^T  \\
T_{A,r}^T
\end{pmatrix}$. Therefore
$$A-AP=S_A\Sigma_A \begin{pmatrix}
O_{q,n}  \\
T_{A,r}^T
\end{pmatrix}=S_A\diag(O_{q},\diag(\sigma_j)_{j=q+1}^n) \begin{pmatrix}
O_{q,n}  \\
T_{A,r}^T
\end{pmatrix},$$ 
and so
 $||A-AP||=||\diag(\sigma_j)_{j=q+1}^n||=\sigma_{q+1}$
because $S_A$ and $T_{A,r}$ are orthogonal matrices.
Similarly deduce that $||A-S_{q,A}S_{q,A}^TA||=\sigma_{q+1}(A)$.
This proves  
(\ref{eqlrap}) 
and (\ref{eqlrap1})
for
$T=T_{q,A}$ and $S=S_{q,A}$.

Now let the matrices  $T$ and $S$ have full rank, 
$\mathcal R(T)=\mathbb T_{q,A}=\mathcal R(T_{q,A})$,
$\mathcal R(S)=\mathbb S_{q,A}=\mathcal R(S_{q,A})$,
and   so
$T=T_{q,A}U$ and $S=S_{q,A}V$
for two nonsingular matrices $U$ and $V$.
Consequently $T(T^TT)^{-1}T^T=T_{q,A}U(U^TT_{q,A}^TT_{q,A}U)^{-1}U^TT_{q,A}^T$.
Substitute $T_{q,A}^TT_{q,A}=I_{q}$ and deduce that 
$(U^TT_{q,A}^TT_{q,A}U)^{-1}=(U^TU)^{-1}=U^{-1}U^{-T}$.
Therefore $U(U^TT_{q,A}^TT_{q,A}U)^{-1}U^T=UU^{-1}U^{-T}U^T=I_{q}$,
and so $T(T^TT)^{-1}T^T=T_{q,A}U(U^TT_{q,A}^TT_{q,A}U)^{-1}U^TT_{q,A}^T=
T_{q,A}T_{q,A}^T$.
Similarly $S(S^TS)^{-1}S^T=S_{q,A}S_{q,A}^T$,
implying the desired extension.
\end{proof}


\subsection{The second basic theorem:
a basis of 
a leading singular space via randomized products}\label{sprb2}


The following theorem supports randomized approximation of 
matrix bases for the leading singular spaces 
$\mathbb T_{\rho, A}$
and $\mathbb S_{\rho, A}$
of a matrix $ A$ having numerical rank $\rho$.
 

\begin{theorem}\label{thover} 
Suppose 
a matrix $ A\in \mathbb R^{m\times n}$
has a numerical rank $\rho$,
$H\in \mathcal G_{0,1}^{n\times \rho_+}$
and $G\in \mathcal G_{0,1}^{m\times \rho_+}$
for $\rho_+\ge \rho$.
Then  the matrices $T=A^TG$ and $S=AH$
have full rank with probability $1$ and 
we can expect that
they have numerical rank $\rho$ and that
\begin{equation}\label{eqlsap}
S+\Delta=S_{\rho, A}U~{\rm and}~ 
 T+\Delta'=T_{\rho, A}V
\end{equation} 
for two matrices $\Delta$ and $\Delta'$
having norms of order $\sigma_{\rho+1}( A)$
and for two nonsingular matrices $U$ and $V$
having condition numbers of at most order $||A||/(\sigma_{\rho}( A)\sqrt \rho)$.  
\end{theorem}
\begin{proof}
The techniques of Section  \ref{sngrm} and Theorem \ref{1}
support the
claims about ranks and numerical ranks.
It remains to deduce  the former probabilistic relationship 
$\mathbb S_{\rho, AH+\Delta}=\mathbb S_{\rho, A}$
of (\ref{eqlsap})
because we can apply it to $A^T$ to obtain the latter  relationship
$\mathbb T_{\rho, A^TG+\Delta'}=\mathbb T_{\rho, A}$.

Assume the SVD $A=
S_{ A}\Sigma_{ A}T_{ A}^T$ and
note that $||\Sigma_{ A}-\diag(\Sigma_{\rho, A},O_{m-\rho,n-\rho})||\le \sigma_{\rho+1}(A)$. 
Consequently
$|| A-S_{ A}~\diag(\Sigma_{\rho, A},O_{m-\rho,n-\rho})~T_{ A}^T||\le \sigma_{\rho+1}(A)$
and 
$ AH=S-\Delta,~S=S_{\rho, A}U,~||\Delta||\le \sigma_{\rho+1}(A)~||AH||$
 where $U=\Sigma_{\rho, A}B$,
$B=T_{\rho, A}^TH$,
and we can expect that the norm $||H||$  
is bounded from above and below by two positive constants 
(see Theorem \ref{thsignorm}). This implies  (\ref{eqlsap}).
It remains to estimate $\kappa(U)$.

With probability $1$ 
the $\rho\times \rho$ matrices $B$ and $U$ are nonsingular 
(see Section \ref{sngrm}).
Furthermore 
we have $||U||\le ||\Sigma_{\rho, A}||~||B||$
where $||\Sigma_{\rho, A}||=|| A||$
and $||B||\le ||T_{\rho, A}||~||H||=||H||$. So $||U||\le|| A||~||H||=O(||A||)$.
We also have
$||U^+||\le ||\Sigma_{\rho, A}^{-1}||~||B^{-1}||$
 for nonsingular matrix $B$.
Observe that 
$||\Sigma_{\rho, A}^{-1}||=1/\sigma_\rho(A)$,
apply Theorem \ref{1} where $M=T_{\rho, A}^T$, 
$\widehat r=\rho$ and $\sigma_{r(M)}(M)=\sigma=1$
and obtain that the norm $||B^{-1}||$ is expected to have at most order $1/\sqrt \rho$.
Summarizing we can expect that the norm $||U^+||$  has 
at most order $1/(\sigma_\rho(A)\sqrt \rho)$. 
Consequently $\kappa (U)=||U||~||U^+||$ has
 at most order $||A||/(\sigma_\rho( A)\sqrt \rho)$.
\end{proof}


\subsection{A prototype algorithm for low-rank approximation}\label{sapsr0}


Theorems \ref{thsng} and \ref{thover} imply correctness of the following
prototype algorithm (cf. \cite[Section 10.3]{HMT11}),
where the input matrix has an unknown numerical rank
and we know its upper bound.


\begin{protoalgorithm}\label{algbasmp} {\bf Rank-$\rho$ approximation of a matrix.}


\begin{description}


\item[{\sc Input:}] 
A  matrix 
$ A\in \mathbb R^{m\times n}$
having an unknown 
numerical rank $\rho$,
an integer $\rho_+\ge \rho$,
 and two tolerances $\tau$ and $\tau'$
of order $\sigma_{\rho+1}( A)/||A||$.
(We can  choose  $\tau$  at
Stage 2
based on  
rank revealing factorization, 
can  choose  $\tau'$  at
Stage 3 based 
on the required 
output accuracy,
and can  adjust both tolerances 
if  the algorithm fails
to produce a satisfactory output.)


\item[{\sc Output:}]
FAILURE (with a low probability) or 
an integer $\rho$ and two matrices 
$T\in \mathbb R^{n\times \rho}$ and 
$A_{\rho}\in \mathbb R^{m\times n}$, both having ranks 
at most $\rho$ and
such that $||A_{\rho}- A||\le \tau' ||A||$ and
$T$ satisfies (\ref{eqlsap}) for $||\Delta'||\le \tau ||A||$.


\item[{\sc Computations:}] $~$


\begin{enumerate}


\item 
 Compute the $n\times \rho_+$ matrix $T'= A^TG$ for  
$G\in \mathcal G_{0,1}^{m\times \rho_+}$.

\item 
Compute a rank revealing factorization of the matrix $T'$
and choose the minimal integer $s$ and an
$n\times s$ matrix  $T$ 
such that $||T'-(T~|~O_{n,\rho_+-s})||\le \tau ||A||$.

\item 
Compute the matrix
$A_s=AT(T^TT)^{-1}T^T$.
Output $\rho=s$, $T$ and $A_{\rho}$  and stop
if 
$||A_{\rho}- A||\le \tau' ||( A)||$.
Otherwise output
FAILURE  and stop.


\end{enumerate}


\end{description}


\end{protoalgorithm}

 Assume 
that both tolerances
$\tau$ and $\tau'$ have been
chosen properly. 
Then by virtue of Theorem \ref{thover},
we can expect that at Stage 2 $s=\rho$ and
 $T$ is
an approximate matrix basis
for the singular space  $ \mathbb T_{\rho, A}$
(within an error norm of at most order $\sigma_{\rho+1}( A)$).
Consequently Stage 3  outputs FAILURE with a 
probability near $0$,
by virtue of Theorems \ref{thsng}.
(In the case of FAILURE we can reapply the algorithm 
 for new values of random parameters
or for the adjusted tolerance values 
$\tau$ and $\tau'$.)
At Stage 2 we have $s\le \rho$
because $\nrank(A^TG)\le \nrank(A)=\rho$,
whereas the bound 
$||A_{\rho}- A||\le \tau' ||( A)||$
at Stage  3
implies that $s\ge \nrank(A)$.
This certifies the outputs $\rho$, $T$,
and $A_{\rho}$ of the algorithm.
 
We can similarly approximate the matrix $ A$ by
a rank-$\rho$ matrix $S(S^TS)^{-1}S^T A$, 
by first computing the matrix $S'=AH$ for 
$H\in \mathcal G_{0,1}^{n\times \rho_+}$,
then computing its rank revealing 
factorization, which is expected to define
an approximate matrix  
basis $S$ for the space $ \mathbb S_{\rho, A}$,
and finally applying Theorem \ref{thsng}.
We  have $T^TT=I_n$ and $S^TS=I_m$ where
the matrices $T$ and $S$ are
 orthogonal, and then
the expressions for 
rank-$\rho$ approximation are simplified.

\begin{remark}\label{resmpl}
One can weaken
 reliability of the output
to
simplify Stage 3 by testing 
whether 
$||K^T( A-A_{\rho})L||\le \tau ||K||~|| A||~||L||$
for matrices $K\in \mathcal G_{0,1}^{m\times \rho'}$
and $L\in \mathcal G_{0,1}^{n\times \rho''}$ and
for two small positive integers $\rho'$ and $\rho''$,
possibly for $\rho'=\rho''=1$,
instead of testing 
whether 
$||A_{\rho}- A||\le \tau' ||( A)||$.  
One can 
similarly simplify Stage 2.
\end{remark}


\begin{remark}\label{reovers}
For $\rho_+=\rho$ Stage 2 can be omitted because 
the matrix $ A^TG$ is expected to be a
desired 
approximate matrix basis
by virtue of Theorem  \ref{thover}. 
The increase of the dimension $\rho_+$ beyond $\rho$
(called {\em oversampling} in  \cite{HMT11}) 
is relatively inexpensive
if the bound $\rho_+$
is small. \cite{HMT11} suggests 
using small
oversampling even if 
the numerical rank $\rho$ is known,
because
we have 
$${\rm Probability}~\{|| A- ATT^T||\le (1+9\sqrt{\rho_+\min\{m,n\}})\sigma_{\rho+1}(A)\}\ge
 1-3(\rho_+-\rho)^{\rho-\rho_+}~{\rm for}~\rho_+>\rho.$$
Theorem \ref{thover}, however,  bounds the norm
$||A- ATT^T||$ strongly also  
for $\rho=\rho_+$, 
 in good accordance with the data
of 
Tables \ref{tabSVD_HEAD1} and
\ref{tabSVD_HEAD1T}.
\end{remark}


\subsection{Computation of  
nmbs and approximation of trailing singular spaces}\label{saptr}


One can approximate trailing singular spaces 
$\mathbb S_{ A,\rho}$ and $\mathbb T_{ A,\rho}$
of a nonsingular
ill conditioned matrix $ A$ having numerical rank $\rho$
by applying Proto-Algorithm \ref{algbasmp} to the matrix $ A^{-1}$,
because $\mathbb T_{ \rho,A^{-1}}=\mathbb S_{ A,\rho}$
and $\mathbb S_{ \rho,A^{-1}}=\mathbb T_{ A,\rho}$.
Next we achieve the same goal without inverting the matrix $A$,
furthermore we cover the case of rectangular inputs. 
At first we  
 compute a
nmb 
 of a rank deficient matrix $A$ and
then  
approximate the trailing singular space $\mathbb T_{A,r}$ of an
ill conditioned matrix $A$ by truncating its SVD  and applying 
Theorem \ref{thpert1} or \ref{thpert}. 
We can compute left nmbs and approximate 
left trailing singular spaces
by applying the same algorithms to the matrix $A^T$.



\begin{theorem}\label{thnmb1} \cite[Theorem 3.1 and Corollary 3.1]{PQ10}.
Suppose
a matrix $A\in \mathbb R^{m\times n}$ has rank $\rho$, 
$U\in \mathbb R^{m\times r}$, $V\in \mathbb R^{n\times r}$,
and 
the matrix $C=A+UV^T$  
has full rank $n$. 
Write $B=C^{(I)}U$.
Then 
$r\ge n-\rho$,
$\mathcal R(B)\supseteq \mathcal N(A)$;
moreover if  $r=n-\rho$, then
 $C^{(I)}U=\nmb(A)$.
Furthermore
  $\mathcal R(BX)=\mathcal N(A)$
if  $\mathcal R(X)=\mathcal N(AB)$.
(Note that $AB=U(I_rV^TC^{-1}U)$
for
$m=n$.)
\end{theorem}


\begin{theorem}\label{thnwe}
Suppose that
 $A \in \mathbb{R}^{m \times n}$,  
$U\in \mathbb{R}^{m\times q}$, 
$V\in \mathbb{R}^{n\times s}$, 
$W\in \mathbb{R}^{s\times q}$, $K=
\begin{pmatrix}
W   &   V^T   \\
-U   &   A
\end{pmatrix}$,
$\rank(W)=q\ge \nul (A)$, $\rank (K)=n+q$,
$m\ge n$.
Write $Y=(O_{n,q}~|~I_n)K^{(I)} 
\begin{pmatrix}
O_{s,q}  \\
U
\end{pmatrix}$.  
 Then 

(a) $\mathcal N(A) \subseteq \mathcal R(Y)$ and if
$\rank (U)=\nul (A)$, 
then 
$\mathcal N(A)=\mathcal R(Y)$,

(b)  $\mathcal R(YZ)=\mathcal N(A)$
if $\mathcal R(Z)=\mathcal N(AY)$,
 whereas

(c) $\mathcal R(Z)=\mathcal N(AY)$
if $\mathcal R(YZ)=\mathcal N(A)$
and if $\rank (Y)=q$.
\end{theorem}

\begin{proof}
See \cite[Theorems 11.2 and  11.3]{PQa}.
\end{proof}


\begin{remark}\label{reaggr0} 
Both theorems define {\rm aggregation processes} (cf. \cite{MP80}).
For 
$r> n-\rho$,
Theorem \ref{thnmb1} 
reduces the computation of a $\nmb(A)$
to the same task for the input $BX$ of a
  smaller size $n\times (r-n+\rho)$.
Furthermore, 
suppose that the matrices $U$ and $Y$ have full rank $q$. Then
part (a) of Theorem \ref{thnwe} implies that $Y$ is a $\nmb(A)$ if
$q=\nul (A)$,
 but
otherwise
parts (b) and (c)  
 reduce the original task of computing a nmb$(A)$
to the case of the input $AY$
of a smaller size $m\times (q-\nul (A))$.
\end{remark}


\begin{theorem}\label{thtrail}
Assume that
$U\in \mathbb R^{m\times r_+}$, 
$V\in \mathbb R^{n\times  r_+}$,
$m\ge n$,
a real $m\times n$ matrix $A$ has numerical rank $\rho=n-r$, 
and  the matrix $C=A+UV^T$ has full rank and is well conditioned.
Then 
$\rho\ge n-r_+$ and
there is a scalar $c$ 
independent of $A$, $U$, $V$, $m$, $n$ and $\rho$ such that
$||C^{+}UX-T_{A,r}||\le c\sigma_{\rho+1}(A)||U||$ where 
$X\in \mathbb R^{r_+\times r}$,
$X=\nmb(AC^{+}U+\Delta)$,  $||\Delta||\le c\sigma_{\rho+1}(A)||U||$.
\end{theorem}
\begin{proof} 
The theorem turns into Theorem \ref{thnmb1} if $\rho=\nrank(A)=\rank(A)$.
If $\rho=\nrank(A)<\rank(A)$, 
set to zero all but the $\rho$ largest singular values in the SVD of the matrix $A$.
Then $\rho=\nrank(A-E)=\rank(A-E)$ and the theorem holds
for the resulting matrix $A-E$ and the matrix $C-E=A-E+UV^T$.
Therefore $T_{A-E,r}=(C-E)^+UX$ 
where for $X$ we choose  an orthogonal 
$\nmb((A-E))((C-E)^+U)$, of size $r_+\times r$.
Clearly $||T_{A-E,r}-T_{A,r}Q||=O(\sigma_{\rho+1}(A))$
for some $r\times r$ orthogonal matrix $Q$,
and it remains to estimate the norm $||(C-E)^+UX-C^+UX||$.
We have $||((C-E)^+-C^+)UX||\le 
||(C-E)^+-C^+||~||U||$. 
The norm $||E||=\sigma_{\rho+1}(A)$ is small
because the matrix $A$ has numerical rank $\rho$,
whereas the norm $||C-E)^+||$ is not large 
because  the full rank matrix $C$ is well conditioned. 
Therefore the value $\tau=||C-E)^+||-||C^+||$  
has at most order $\sigma_{\rho+1}(A)$
by virtue of Theorem \ref{thpert1}.
\end{proof}


\begin{corollary}\label{cotrail}
Suppose 
a normalized real $m\times n$ matrix $A$ has numerical rank $\rho=n-r$, 
$U\in \mathcal G_{0,1}^{m\times r_+}$,
 $V\in \mathcal G_{0,1}^{n\times r_+}$,
$m\ge n$,  and
$C= A+UV^T$.
Then (i) the matrix $C$ 
is singular or ill conditioned if $r_+< r$ but otherwise
(ii) has full rank with probability $1$, and 
(iii) we can expect that the matrix $C^{+}UX$ is an
approximate matrix basis for
the singular space $\mathbb T_{A,r}$
within an error norm of 
at most 
order $\sigma_{\rho+1}(A)$
where  
$X$ is an orthogonal 
$\nmb(AC^{+}U+\Delta)$ 
of the size $r_+\times r$
and  $||\Delta||\le c\sigma_{\rho+1}(A)$.
\end{corollary}
\begin{proof}
Part (i) is 
immediately verified.
Furthermore
by virtue of Theorem \ref{thkappa2} the matrix $C$ 
has full rank with probability $1$ and
is expected to  be well conditioned,
whereas the norm $||U||$ is expected to be not large 
by virtue of Theorem \ref{thsignorm}. 
Therefore Corollary \ref{cotrail}
follows from Theorem \ref{thtrail}.
\end{proof}


Likewise by employing
Theorems  \ref{thaugkp}  and 
\ref{thnwe}  
instead of  Theorems \ref{thkappa2}
and
\ref{thnmb1},
we obtain the following 
result.


\begin{corollary}\label{cotrail0} 
Suppose that
a normalized real $m\times n$ matrix $A$ has numerical nullity $r=\nnul(A)$,
$U\in \mathcal G_{0,1}^{m\times q}$,
$V\in \mathcal G_{0,1}^{n\times s}$,   
$W\in \mathcal G_{0,1}^{s\times q}$,
$ K=
\begin{pmatrix}
W   &   V^T   \\
-U   &    A
\end{pmatrix}$,
$\rank(W)=q$, $\rank ( K)=n+q$,
$Y=(O_{n,q}~|~I_n) K^{+} 
\begin{pmatrix}
O_{s,q}  \\
U
\end{pmatrix}$, and $m\ge n$.  
 Then (i) the matrix $K$ is rank deficient or ill conditioned where
$q<r$ but otherwise has full rank with probability $1$
and is expected to be well conditioned.
Furthermore    
 we can expect that
 within an error norm of 
at most 
order $\sigma_{n-q+1}( A)$
a matrix basis for
the singular space 
$\mathbb T_{ A,q}$ is
 approximated by
(ii)  the matrix $Y$ 
if $r=q$ or (iii)
 the matrix $YZ$  if $q> r$
where
$Z\in \mathbb R^{q\times r}$,
$Z=\nmb(AY+\Delta)$,  $||\Delta||\le c\sigma_{n-q+1}(A)$.
\end{corollary}

Corollaries \ref{cotrail} and \ref{cotrail0} 
(for $s=q$) imply correctness 
of the two following Prototype Algorithms.


\begin{protoalgorithm}\label{algbasap} {\bf An approximate basis for 
a trailing singular space by using randomized additive 
preprocessing.}


\begin{description}


\item[{\sc Input:}] 
A  matrix 
$ A\in \mathbb R^{m\times n}$ for $m\ge n$ 
with $||A||\approx 1$,
an upper bound $r_+$
on its unknown numerical nullity $r=\nnul(A)$,
 and two  tolerances $\tau$
and $\tau'$
of order $\sigma_{n-r+1}( A)$.
(The tolerances are defined by the requested output accuracy.
In a variation of the algorithm one can
reapply it with a decreased tolerance 
$\tau'$ instead of outputing FAILURE
 at Stage 4.)


\item[{\sc Output:}] FAILURE (with a low probability)
or the numerical nullity $r$
and an approximate matrix 
basis $B$, 
within an error norm in
$O(\sigma_{n-r+1}( A))$,
of the trailing singular space $\mathbb T_{A,r}$.


\item{\sc Initialization}: $~$
Generate two matrices
$U\in \mathcal G_{0,1}^{m\times r_+}$ and
$V\in \mathcal G_{0,1}^{n\times r_+}$
for $\sigma$ of order $||A||$.


\item{\sc Computations}: $~$ 


\begin{enumerate}


\item 
Compute the matrix $ C= A+UV^T$.

\item 
Stop and output FAILURE
if this matrix  is
rank deficient or ill conditioned.
Otherwise compute the matrices $Y= C^+U$
and $ AY$.


\item 
Output $r=r_+$ and $B=Y$ and stop
if $|| A Y||\le \tau || A||~||Y||$.


\item 
Otherwise apply an algorithm 
(e.g. employing SVD,
rank revealing factorization, 
a technique from \cite{PQ10} or
 \cite{PQa}, or one of 
Proto-Algorithms \ref{algbasap}
and \ref{algbasaug})
that
 for the matrix $AY$
and a fixed tolerance $\tau'$
computes an integer $r$ and
an orthogonal approximate matrix basis $X$ (of size $r_+\times r$)
for the space $\mathbb T_{AY,r}$.
If $|| A B||\le \tau || A||~||B||$ then
output $r$ and $B=YX$
and stop. 
Otherwise output FAILURE and stop.

\end{enumerate}


\end{description}


\end{protoalgorithm}



\begin{protoalgorithm}\label{algbasaug} {\bf An approximate basis for 
a trailing singular space by
using randomized augmentation.}


\begin{description}


\item[{\sc Input, Output}]
and Stages 3 and 4 of {\sc Computations}
are as in
Proto-Algorithm 
 \ref{algbasap}.


\item{\sc Initialization}: $~$
Generate three matrices
$U\in \mathcal G_{0,1}^{m\times r_+}$,
$V\in \mathcal G_{0,1}^{n\times r_+}$,
and
$W\in \mathcal G_{0,1}^{r_+\times r_+}$
for $\sigma$ of order $||A||$.


\item{\sc Computations}: $~$ 


\begin{enumerate}


\item 
Stop and output FAILURE
if the matrix 
$ K=
\begin{pmatrix}
W   &   V^T   \\
-U   &    A
\end{pmatrix}$
 is
rank deficient or ill conditioned.


\item 
Otherwise compute the matrices 
$Y=(O_{n,r_+}~|~I_n) K^{+} 
\begin{pmatrix}
O_{r_+,~r_+}  \\
U
\end{pmatrix}$
and $ AY$.



\end{enumerate}


\end{description}


\end{protoalgorithm}


\subsection{Alternative methods for the approximation of
leading singular spaces}\label{saplddl}


Next we extend Theorem \ref{thtrail} and Corollary \ref{cotrail}
assuming a nonsingular input matrix 
and an upper bound on the numerical rank
of its inverse. 


\begin{theorem}\label{thhead} (Cf. Remark \ref{reqq+}.)
Assume that five matrices $ A\in \mathbb R^{n\times n}$,
$U_-, V_-\in \mathbb R^{n\times q_+}$,
$H=I_{q_+}+V_- AU_-^T$, and
$ C_-= A- AU_-H^{-1}V_-^T A$
have full ranks, 
the matrix $C_-$ is well conditioned,
and  $q=\nrank (A)=\le q_+$.
Then 
there exists a scalar $c_-$ independent of $ A$, $U_-$, $V_-$, $n$ and $q_+$ 
and such that
$|| C_-^{T}V_-Y_--T_{q, A}||\le c_-\sigma_{q+1}( A)$
where $Y_-\in \mathbb R^{q_+\times q}$, $Y_-$
is a matrix basis for the space
$\mathbb T_{A^{-1}C_-^TV+\Delta,q}$ and  
$||\Delta||\le c_-\sigma_{q+1}( A)$.
\end{theorem}
\begin{proof}
Recall that
$ C_-^{-1}= A^{-1}+U_-V_-^T$ 
(cf. (\ref{eqsmwd})) and that
$\nrank (A)=\nnul (A^{-1})$.
Rewrite SVDs $ A=S_{ A}\Sigma_{ A}T_{ A}^T$  
and $ A^{-1}=T_{ A}\Sigma_{ A}^{-1}S_{ A}^T$
as follows,
$$ A=
(S_{q, A}~|~S_{ A,n-q})\diag(\Sigma_{q, A},
\Sigma_{ A,n-q})(T_{q, A}~|~T_{ A,n-q})^T,$$

$$ A^{-1}=
(T_{q, A}~|~T_{ A,n-q})\diag(\Sigma_{q, A}^{-1},\Sigma_{ A,n-q}^+)(S_{q, A}~|~S_{ A,n-q})^T.$$
\noindent 
Apply Theorem \ref{thtrail} to the matrices $ A^{-1}$, $C^{-1}$, and $X$ replacing $A$,
$C$, and $Y$, respectively.
\end{proof}


\begin{remark}\label{reqq+} 
One can first compute the
 numerical nullity  $q=\nnul(A)$ of
 an ill conditioned
 matrix $A$ 
(see Section \ref{stlss1}
on this computation)
and then
an approximate 
matrix basis $C_-^TV$ of the space
$\mathbb T_{q, A}$.
This  
can be more attractive than
computing the matrix $A^{-1}C_-^TV$.
In the next two corollaries 
we assume that the numerical rank of the input matrix $A$
is available.
\end{remark}


\begin{corollary}\label{cohead} (Cf. Remark \ref{reqq+}.)
Suppose $ A\in \mathbb R^{n\times n}$,
$U_-, V_-\in \mathcal G_{0,\sigma}^{n\times q}$,
$\sigma$ has order $|| A^{-1}||$,
$H=I_{q}+V_- AU_-^T$, 
$C_-=A-AU_-H^{-1}V_-^TA$ for $H=I_q+V_-^TAU_-$
(cf. (\ref{eqsmwd})),
and  $q=\nrank (A)$. 
(See Section \ref{sdual}
on estimating the norm $|| A^{-1}||$.)
Then 
the matrix $ C_-^{T}V_-$ 
is 
expected to
approximate within an error norm of at most order $\sigma_{q+1}( A^{-1})$
a matrix basis of the  leading
singular space
$\mathbb T_{q, A}$.
\end{corollary}
\begin{proof}
By virtue of Theorem \ref{thkappa2} the matrix $ C_-$ of Theorem \ref{thhead}
has full rank with probability $1$ and
is expected to  be well conditioned, and so Corollary \ref{cohead}
follows from Theorem \ref{thhead} for $q=q_+$.
\end{proof}


\medskip
The {\em dual augmentation}
of the following corollary
provides an alternative expression for an approximate 
matrix basis of a leading singular space.


\begin{corollary}\label{cotrail01} 
Suppose that
 $ A \in \mathbb{R}^{n \times n}$,
$U, V\in \mathcal G_{0,\sigma}^{n\times q}$,   
$W\in \mathcal G_{0,\sigma}^{q\times q}$, 
the matrix $ A$ is nonsingular,
 $\sigma$  has  order $||A^{-1}||$,
$q=\nrank(A)$,
$ K_+=
\begin{pmatrix}
W   &   V^T   \\
-U   &    A^{-T}
\end{pmatrix}$,
$\rank(W)=r$, and $\rank (K_+)=n+q$.
 Then we can expect that
 the matrix $T_+=(O_{n,q}~|~I_n) K_+^{-1} 
\begin{pmatrix}
O_{q,q}  \\
U
\end{pmatrix}$   
approximates within an error norm 
in $O(\sigma_{q+1}(A))$
a matrix basis for
the right  
leading singular space $\mathbb T_{q, A}$.
\end{corollary}
\begin{proof}
Write SVD $ A=S_{ A}\Sigma_{ A}T^T_{ A}$ of (\ref{eqsvd})
and deduce that
$ A^{-T}=S_{ A}(\Sigma_{ A})^{-T} T^T_{ A}$
where
$(\Sigma_{ A})^{-T}=\diag(1/\sigma_{j}( A))_{j=1}^n$.
Note that $\nrank(A)=\nnul (A^{-T})$
and
apply Corollary \ref{cotrail0} 
replacing the matrix $A$ with $ A^{-T}$.
\end{proof}

Closer examination of  the expression for the matrix $T_+$
enables us to  simplify it as follows,
\begin{equation}\label{eqhead+}
T_+=B-BS^{-1}V^TB~{\rm for}~B= A^TU
\end{equation}
where 
$S=W+U^T A^TV$.
 $S^{-1}$ is the only
 matrix inverse 
involved into computing $T_+$
(cf. (\ref{eqsmwd})).


\subsection{Some amendments}\label{samend}


\begin{remark}\label{regap}
Approximation of the leading and  trailing singular spaces 
as well as the computation of numerical rank and 
 numerical nullity 
(see Section \ref{stlss1})
are facilitated
as the gaps increase between the singular
values of the input matrix $A$.
This motivates using
the power transforms $A\Longrightarrow B_h=(AA^T)^hA$
 for positive integers $h$
because $\sigma_j(B_h)=(\sigma_j(A))^{2h+1}$ for all $j$.
\end{remark}


\begin{remark}\label{renmb}
In the case where $m=n$ 
the computations are
simplified and stabilized,
and
furthermore we can apply 
Theorem \ref{thnmb1} 
or \ref{thnwe}
to both $A$ and $A^T$
to define both left and right nmbs.
We can  reduce to this case the
computation for 
a rectangular matrix $A$
in various ways, e.g., by observing that
(a) $\mathcal N(A)=\mathcal N(A^TA)$,  
(b) $\mathcal N(A)=\mathcal N(B^TA)$ if $A,B\in \mathbb R^{m\times n}$
 and the matrix $B$ has full rank $m\le n$, and
(c) $(A~|~O_{n,m-n}){\bf u}={\bf 0}_{m}$
if and only if $A\widehat {\bf u}={\bf 0}_{m}$ provided $m\ge n$ and 
$\widehat {\bf u}=(I_n~|~O_{n,m-n}){\bf u}$,
whereas $(A^T~|~O_{n,m-n}){\bf v}={\bf 0}_{n}$ if and only if 
$\widehat {\bf v}={\bf 0}_{n}^T$ provided $m< n$ and 
$\widehat {\bf v}=
(I_m~|~O_{n-m,m}){\bf v}$.
 Furthermore
given an $m\times n$ matrix $A$ for $m>n$, we can represent it 
as the block vector $A=(B_1^T~|~B_2^T~|~\dots ~|~B_h^T)^T$
where
 $B_i$ are $k_i\times n$ blocks for $i=1,\dots,h$, 
$\sum_{i=1}^hk_i= m$, and observe that $\mathcal N(A)=\cap_{i=1}^h \mathcal N(B_i)$, 
and 
we can compute the intersection of null spaces by applying \cite[Theorem 12.4.1]{GL96}. 
One can extend these comments 
to the tasks of the approximation of the singular spaces
of ill conditioned matrices.
\end{remark}


\section{Sparse and structured randomization.
Numerical rank without pivoting and orthogonalization}\label{sstr}


\subsection{Randomized structured preprocessing}\label{sourapstr}


Would the additive
preprocessing $A\Longrightarrow C=A+UV^T$ 
preserve the structure of an  $n\times n$ matrix $A$ 
where $A\in \mathbb R^{m\times n}$,
$U\in \mathbb R^{m\times r}$ and $V\in \mathbb R^{n\times r}$? 
Adding the  
matrix $UV^T$ 
makes small
impact
on the structure
if the  ratio $r/\min\{m,n\}$ is 
small, e.g.,
 the displacement rank increases
by  $O(r)$ 
(cf.  \cite{p01}),
but we can control this impact
even for large values $r$
 by endowing
the matrices $U$ and $V$ with  proper structure.
Given  
a pair of 
standard Gaussian random Toeplitz  $n\times r$ matrices $U$ and $V$
and
a displacement generator of a small length $d$
for  a nonsingular ill conditioned 
Toeplitz-like  $n\times n$ matrix $A$ that 
has a numerical nullity $r=\nnul(A)$ and a norm $||A||\approx 1$,
we can readily compute a displacement generator of
length $d+O(1)$
for the matrix $C=A+UV^T$. 
By exploiting the structure we can operate with this 
matrix in nearly linear arithmetic time,
e.g., solve a nonsingular linear system $A{\bf y}={\bf b}$ 
in $O(d^2n\log^2n)$
flops, even where $r$ is large
(see Theorem \ref{thops}).
Both randomized augmentation 
$A\Longrightarrow K=\begin{pmatrix}
W   &   V^T   \\
-U   &   A
\end{pmatrix}$
for proprer choice of the random blocks $U$, $V$, and $W$
and random sampling
$A\Longrightarrow A^TG$
and $A\Longrightarrow AH$
for proper choice of random matrices $G$ and $H$
 preserve matrix sparseness and structure even better.
Empirically these maps preserve their preconditioning properties
for such choices of the matrices $G$, $H$,
$U$, $V$, and $W$; likewise
endowing
random  multipliers with sparseness and structure
keeps the
support for safe numerical GENP and block Gaussian elimination
(see Remark \ref{rertm} and Tables \ref{tabprec},
\ref{tab44}, \ref{tabSVD_HEAD1T},
 and \ref{tabhank}). 

\begin{remark}\label{rehomo}
Alternative deterministic techniques of homotopy
continuation also support inversion 
in nearly linear time
of nonsingular Toeplitz matrices and 
other matrices 
with displacement  structure
 (see \cite[Section 6.9]{p01},
\cite{P07}, \cite{P10}).
\end{remark}


\subsection{Numerical rank 
without pivoting and orthogonalization}\label{stlss1}


If we know the numerical rank $\rho$ of
a matrix $ A$,
then we can simplify 
Proto-Algorithm \ref{algbasmp}
for rank-$\rho$ 
approximation
as well as the computation of
approximate bases for the leading and trailing singular spaces 
of the matrix $A$
(see Remarks \ref{reovers} and \ref{reqq+} 
and Corollaries \ref{cotrail}
and \ref{cotrail0}).

The customary algorithms for the 
numerical  rank of a matrix 
rely on computing its SVD or 
rank revealing factorization,
which involve 
pivoting and orthogonalization
and thus destroy 
matrix sparseness and structure.
Randomized
Proto-Algorithm \ref{algbasmp}
is 
 a noncostly alternative
where  
the given upper bound $\rho_+$
on the numerical rank is small.
Indeed Proto-Algorithm 
\ref{algbasmp} uses
rank revealing factorization
at Stage 2 and matrix inversion or 
orthogonalization at Stage 3,
but in these cases 
only  deals with matrices of small 
sizes
if  $\rho_+$ is small.

Next we describe other 
alternatives that  
avoid pivoting and
orthogonalization 
even where the
 numerical  rank $\rho$ is
large.
As by-product they compute
an
approximate matrix basis 
within an error norm in $O(\sigma_{\rho+1}(A))$
for 
the leading  singular space 
$ \mathbb T_{\rho, A}$ 
of an $m\times n$ matrix $ A$ 
and if we wish also rank-$\rho$ 
approximation of the matrix $A$
(see Remark \ref{recorr}).
 We let
$m\ge n$
(else shift to $ A^T$),
let
$[\rho_-,\rho_+]=[0,n]$
unless we know a more narrow range,
and successively test the selected 
candidate integers in the  range
$[\rho_-,\rho_+]$
until we find the numerical rank $\rho$. 
To improve reliability, we can 
repeat the tests for
distinct values of random parameters.

Exhaustive search
 defines and verifies the numerical rank $\rho$
with probability near $1$, but with proper
policies one can use fewer and simpler tests
because  
for $G\in \mathcal G_{0,1}^{m\times s}$
(and empirically for various random sparse and 
structured matrices $G$ as well)
 the matrix $B= A^TG$ is expected 
(a) to have full rank and  to be
well conditioned 
if and only if $s\ge \rho$,
(b) to  
 approximate a
matrix basis    
(within an error norm 
in $O(\sigma_{\rho+1}(A))$)
for a linear space   
$\mathbb T\supseteq \mathbb T_{\rho,B}=\mathbb T_{\rho, A}$ 
where
$s\ge \rho$, and (c)  to
approximate 
a
matrix basis
(within an error norm 
in $O(\sigma_{\rho+1}(A))$)
 for the space 
$\mathbb T_{\rho, A}$ where $s=\rho$.
Property (a) is implied by Theorem \ref{1},
properties (b) and (c) by Theorem \ref{thover}.


\begin{protoalgorithm}\label{algnrank} {\bf Numerical rank
with random sampling} 
(see Remarks \ref{recorr}--\ref{rebsch}).


\begin{description}


\item[{\sc Input:}] 
Two integers $\rho_-$ and $\rho_+$ and a
matrix 
$ A\in \mathbb R^{m\times n}$ 
having unknown numerical rank $\rho=\rank ( A)$
in the range $[\rho_-,\rho_+]$
such that $0\le \rho_-< \rho_+\le n\le m$,
a rule for the selection of a candidate 
integer
$\rho$ in a range $[\rho_-,\rho_+]$,
and  a Subroutine COND
that determines whether a given matrix 
has  full rank and is well 
conditioned or not.


\item[{\sc Output:}] 
an integer  $\rho$ expected to equal
numerical rank 
of the matrix 
$ A$ and
a
matrix $B$
 expected to
approximate 
(within an error norm 
in $O(\sigma_{\rho+1}(A))$)
a matrix basis
of the 
singular space  $\mathbb T_{\rho, A}$.
(Both expectations can actually fail, but
with a low probability, see Remark \ref{recorr}.)


\item[{\sc Initialization:}] 
Generate matrix $G\in \mathcal G_{0,1}^{m\times \rho_+}$ and
write  $B= A$, $G_{\rho}=G(I_{\rho}~|~O_{\rho,m-\rho})^T$ for $\rho=\rho_-,\rho_-+1,\dots,\rho_+$.


\item{\sc Computations}: $~$ 


\begin{enumerate}


\item 
Output $\rho=\rho_+$ and the matrix $B$
 and stop
if $\rho_-=\rho_+$.
Otherwise fix an integer $\rho$ in the range $[\rho_-,\rho_+]$.


\item 
Compute the matrix $B'=B^TG_{\rho}$
and apply to it the Subroutine COND. 


\item 
If this
matrix  has full rank
and is well 
conditioned,
write 
$\rho_+=\rho$
and $B= B'$
and go to Stage 1.
Otherwise 
write $\rho_-=\rho$
and go to Stage 1.


\end{enumerate}


\end{description}


\end{protoalgorithm} 

\begin{remark}\label{recorr}
The algorithm can output a wrong value 
of the numerical rank, although 
by virtue of Theorems \ref{thsng} and \ref{thover}
combined
this occurs with
a 
low probability. 
One can decrease this probability
by reapplying
the algorithm
to the same inputs and 
choosing distinct random 
parameters. Furthermore one can
fix
a 
tolerance $\tau$
of order $\sigma_{\rho+1}(A)$,
set $T=B$, and apply Stage 3 of 
Proto-Algorithm \ref{algbasmp}.
Then
$\nrank (A)$ 
is expected to exceed
the computed 
value $\rho$
if this stage outputs
 FAILURE and to equal $\rho$
otherwise, in which case 
the algorithm also
outputs a
rank-$\rho$ approximation of the matrix $A$
(within an error norm $\tau ||A||$
in $O(\sigma_{\rho+1}(A))$). 
For a sufficiently small tolerance  $\tau$
the latter outcome
implies that
certainly  $\rho\ge \nrank (A)$.
\end{remark}

\begin{remark}\label{rescond}
A Subroutine COND,
which
tests whether an $m\times \rho$ matrix $B'$
has  full rank and is well 
conditioned, can employ
 SVD of the matrix $ A$ or its rank revealing factorization,
thus involving pivoting or orthogonalization.
We can avoid  this charge
on matrix sparseness and structure
by using randomization
(although this is less important where $\rho_+$
is small).
Namely assume that, say $m\ge n$
and recall that the algorithm of \cite{D83}
computes a 
  close upper bound $\sigma_+^2$ 
on  the largest eigenvalue $\sigma^2$
of the matrix $S= A^T A$ by
recursively 
computing the vectors 
${\bf v}_i=S^i{\bf v}=S{\bf v}_{i-1}$, $i=1,2,\dots$ 
for a random vector ${\bf v}={\bf v}_0$.
By reapplying this algorithm
to the matrix $\sigma_+^2 I- A^T A$
we can approximate the absolutely smallest
eigenvalue of the matrix $S$,
which is actually equal to $\sigma^2_n( A)$.
Here we just need  a crude estimate 
to support our algorithm.
\end{remark}

\begin{remark}\label{rebsch}
The binary search
$\rho=\lceil(\rho_-+\rho_+)/2\rceil$
is an attractive policy for choosing
the candidate values $\rho$,
but one may prefer to move 
toward the left end $\rho_-$ of the range 
more rapidly,
to decrease the size of the matrix $B'$.
\end{remark}

In principle in our search for numerical rank
we can employ 
Corollary \ref{cotrail} or \ref{cotrail0}
instead of Theorem \ref{thover}.
Then we  would have to apply the Subroutine
COND to matrices of size $m\times n$
or larger, which means extra computational cost.
Because of that the two respective  
Prototype Algorithms below cannot  compete with
Proto-Algorithm \ref{algnrank} 
unless the input matrix has a small 
numerical nullity.


\begin{protoalgorithm}\label{algnrankap} {\bf Numerical rank
 via  randomized additive preprocessing.}


\begin{description}


\item[{\sc Input, Output}]
and Stage 1 of {\sc Computations} are the same as in Proto-Algorithm \ref{algnrank}.


\item{\sc Initialization}: $~$
Compute the integer
$r_+=n-\rho_-$ and
a scalar $\sigma$ of order $|| A||$,
generate two matrices
$U_+\in \mathcal G_{0,\sigma}^{m\times r_+}$ and
$V_+\in \mathcal G_{0,\sigma}^{n\times r_+}$,
and
write  
$U_{s}=U_+(I_{s}~|~O_{s,m-s})^T$ and
$V_{s}=V_+(I_{s}~|~O_{n-s,s})^T$ 
for $s=r_-,r_-+1,\dots,r_+$.


\item{\sc Computations}: $~$ 


\medskip

2. Compute the integer $s=n-\rho$.
 Compute the matrix
 $C=A+U_sV_s^T$
and apply to it the Subroutine COND. 

\medskip


3.  
If this matrix 
is rank deficient or ill conditioned 
write $\rho_+=\rho$ and go to Stage 1.
Otherwise write $\rho_-=\rho$ and 
go to Stage 1.



\end{description}


\end{protoalgorithm}


\begin{protoalgorithm}\label{algnrankaug} {\bf Numerical rank
via randomized augmentation.}


\begin{description}


\item[{\sc Input, Output}] and Stages 1 and 3
of {\sc Computations} are the same as in 
 Proto-Algorithm \ref{algnrankap}.


\item{\sc Initialization}: $~$
Compute 
the integer
$r_+=n-\rho_-$ and
a scalar 
$\sigma$ of order $|| A||$,
generate three matrices
$U_+\in \mathcal G_{0,\sigma}^{m\times r_+}$,
$V_+\in \mathcal G_{0,\sigma}^{n\times r_+}$,  and 
$W_+\in \mathcal G_{0,\sigma}^{r_+\times r_+}$,
and
write  $i=1$, $A_{0}= A$,
$U_{s}=U_+(I_{s}~|~O_{s,m-s})^T$, 
$V_{s}=V_+(I_{s}~|~O_{n-s,s})^T$, and 
$W_{s}=(I_{s}~|~O_{s,r_+-s})W_+(I_{s}~|~O_{s,r_+-s})^T$ 
for $s=r_-,r_-+1,\dots,r_+$.


\item{\sc Computations}: $~$ 



\medskip

2. Compute the integer $s=n-\rho$.
 Compute the matrix
 $ K=\begin{pmatrix}
W_{s}   &   V_{s}^T   \\
-U_{s}   &   A
\end{pmatrix}$ 
and apply to it the Subroutine COND. 



\end{description}


\end{protoalgorithm}


\subsection{Preprocessing for
Newton--Toeplitz iteration}\label{snewt}


 Newton's iteration
for matrix inversion  
\begin{equation}\label{eqnewt}
X_{i+1}=X_i(2I-CX_i),~i=0,1,\dots.
\end{equation}
squares the residuals 
$I-CX_i$, that is, 
\begin{equation}\label{eqnewtres}
I-CX_{i+1}=(I-CX_i)^2=(I-CX_0)^{2^{i+1}},~i=0,1,\dots.
\end{equation}
Therefore 
\begin{equation}\label{eqnewtresnorm}
||I-CX_{i+1}||\le||I-CX_i||^2=||I-CX_0||^{2^{i+1}},~i=0,1,\dots, 
\end{equation} 
and so the approximations $X_i$ quadratically converge to the inverse $C^{-1}$ right from the start 
provided that $||I-CX_0||<1$.
We can ensure that $||I-CX_0||\le 1-\frac{2n}{(\kappa (C))^2(1+n)}$ 
by choosing $X_0=\frac{2nC^T}{(1+n){||C||_1||C||_{\infty}}}$ \cite{PS91}.
 Newton's iteration can be incorporated into our randomized algorithms. 
E.g., we can use it instead of Gaussian elimination in Proto-Algorithm  \ref{flch}
of the next section.


The map $C\Longrightarrow X_0$ preserves the matrix structure of Toeplitz or Hankel type,
but is the structure maintained throughout the iteration? 
Not automatically. In fact Newton's loop can triple the displacement rank 
of a matrix $X_k$.
The papers \cite{P92}, \cite{P93}, and \cite{P93a}, however, have
proposed to maintain the structure via recursive compression of the displacements
(one can also say {\em recompression}),
thus defining
{\em Newton's structured} (e.g., New\-ton--Toep\-litz)  
iteration. Recall that we can readily recover a Toep\-litz-like
matrix from its displacement (cf. (\ref{eqtlk})).
According to the compression policy
proposed in the papers \cite{P92}, \cite{P93}, and \cite{P93a},
one should periodically set to $0$
the smallest singular values of the displacements of the matrices $X_i$ 
to keep the length of the displacements within a fixed tolerance,
equal to or a little exceeding the displacement rank of the input matrix $C$.
According to the estimates in \cite{p01}, the New\-ton--Toep\-litz  iteration 
converges quadratically right from the start provided
$||I-CX_0||< \frac{1}{(1+||Z_e||+||Z_f||)\kappa(C)}||L^{-1}||$, $||L^{-1}||\le c_{e,f}n$, 
$L$ denotes the associated displacement operator 
$L:~C\Longrightarrow Z_eC-CZ_f$ for $e\neq f$ 
or $L:~C\Longrightarrow C-Z_eCZ_f^T$ for $ef\neq 1$, and $c_{e,f}$ is 
a constant defined by $e$ and $f$. Similar bounds can be 
deduced for other classes of matrices with displacement 
structure \cite[Section 6.6]{p01}. See \cite{PBRZ99}, \cite{PRW02},
\cite[Chapter 6]{p01} and \cite{P10} on further
information.
The cost of computing  
the $n\times d$ generator
matrices $G$ and $H$ with SVD 
or rank revealing factorization
is  not high
for small ranks $d$, 
but 
 randomized methods of Section \ref{sapsr} 
enable further cost decrease.


\subsection{Application to tensor decomposition}\label{svianvtns}


Let 
\begin{equation}\label{eqtens}
{\bf A}=[A(i_1,\dots,i_d)]
\end{equation}
denote a $d$-dimensional {\em tensor} with entries 
$A(i_1,\dots,i_d)$ and {\em spacial indices}
$i_1,\dots,i_d$ ranging from $1$ to 
$n_1,\dots,n_d$, respectively.
Define the $d-1$ {\em unfolding matrices}
$A_k=[A(i_1\dots i_k;i_{k+1}\dots i_d)],~k=1,\dots,d$,
where the semicolon separates the
multi-indices
$i_1\dots i_k$ and $i_{k+1}\dots i_d$,
which define the rows and columns of the matrix $A_k$,
respectively, $k=1,\dots,d$.
The paper \cite{O09} proposed the following class of {\em Tensor Train Decompositions},
 hereafter referred to as {\em TT Decompositions}, where
the {\em summation indices} $\alpha_{1},\dots,\alpha_{d-1}$
ranged from $1$ to {\em compression ranks} $r_1,\dots,r_{d-1}$,
respectively,
\begin{equation}\label{eqtt}
T=\sum_{\alpha_1,\dots,\alpha_{d-1}}G_1(i_1,\alpha_1)G_2(\alpha_1,i_1,\alpha_2)\cdots
G_{d-1}(\alpha_{d-2},i_{d-1},\alpha_{d-1})G_d(\alpha_d,i_d).
\end{equation}


\begin{theorem}\label{thtt} \cite{O09}.
For any tensor ${\bf A}$ of (\ref{eqtens})
there exists a TT decomposition (\ref{eqtt})
such that ${\bf A}={\bf T}$ and $r_k=\rank (A_k)$
for $k=1,\dots,d-1$.
\end{theorem}
There is a large and growing number
of important applications of TT decompositions 
(\ref{eqtt})
to modern computations
(cf. e.g., \cite{OT09}, \cite{OT10}, \cite{OT11})
where the numerical ranks 
of the unfolding matrices $A_k$
are much smaller than their ranks, and
 it is desired to
 compress TT decompositions respectively.


\begin{theorem}\label{thot} \cite{OT10}.
For any tensor ${\bf A}$ of (\ref{eqtens})
and any set of positive integers 
 $r_k\le \rank (A_k)$,  $k=1,\dots,d-1$,
there exists a TT decomposition (\ref{eqtt})
such that
\begin{equation}\label{eqot}
||{\bf A}-{\bf T}||_F^2\le \sum_{k=1}^{d-1}\tau_k^2,~\tau_k=\min_{\rank (B)=r_k}||A_k-B||_F,~k=1,\dots,d-1. 
\end{equation}
\end{theorem}

The constructive proof of this theorem in \cite{OT10}
relies on inductive approximation of unfolding matrices 
by their SVDs truncated to the compression ranks $r_k$.
Let us sketch this  construction.
For $d=2$ we obtain  a desired TT
decomposition
$T(i_1,i_2)=\sum_{\alpha_1}^{r_1}G_1(i_1,\alpha_1)G_2(\alpha_1,i_2)$
 (that is a sum of $r_1$
outer products of $r_1$ pairs of vectors) simply
 by  truncating 
 the SVD of the matrix $A(i_1,i_2)$.
At the inductive step one truncates the SVD 
of the first unfolding matrix 
$A_1=S_{A_1}\Sigma_{A_1}T_{A_1}^T$
to obtain rank-$r_1$ approximation of this matrix
$B_1=S_{B_1}\Sigma_{B_1}T_{B_1}^T$
 where $\Sigma_{B_1}=\diag(\sigma_j(A_1))_{j=1}^{r_1}$
and the matrices $S_{B_1}$ and $T_{B_1}$
are formed by the first $r_1$ columns of the matrices 
$S_{A_1}$ and $T_{A_1}$, respectively.
Then it remains to approximate the tensor
${\bf B}=[B(i_1,\dots,i_d)]$
represented by the matrix $B_1$.
Rewrite it as
$\sum_{\alpha_1=1}^nS_{B_1}(i_1;\alpha_1)\widehat A(\alpha_1;i_2\dots i_d)$
for $\widehat A=\sum_{B_1}T_{B_1}^T$,
represent $\widehat A$ as the tensor
${\bf \widehat A}=[A(\alpha_1i_2,i_3,\dots,i_d)]$
of dimension $d-1$, apply the inductive hypothesis
to obtain a TT-approximation of this tensor,
and extend it to a TT-approximation of the original tensor 
${\bf A}$.

In \cite{OT10} the authors specify this construction as 
their Algorithm 1,
 prove error norm bound (\ref{eqot}), then
point out that the 
``computation of the truncated SVD 
for large scale and possibly dense 
unfolding matrices ... is unaffordable
in many dimensions", propose 
``to replace SVD by some other dyadic decompositions
$A_k\approx UV^T$,
which can be computed with low complexity",  and 
finally specify such recipe as  \cite[Algorithm 2]{OT10},
 which
is an iterative algorithm for 
skeleton or pseudoskeleton decomposition
of matrices and which they
 use at Stages 5 and 6
of their Algorithm 1.
The cost of each iteration 
of \cite[Algorithm 2]{OT10}
is quite low 
and empirically the iteration converges fast, but
the authors welcome alternative recipes
 having formal support.

Our randomized Proto-Algorithms 
\ref{algbasmp} and \ref{algnrank} can
serve as the alternatives to \cite[Algorithm 2]{OT10}.
For the input matrix $A_1$ above
they  
use $O(r_1)$ multiplications of this matrix
by $O(r_1)$ vectors, which means a low 
computational cost
 for sparse and structured inputs,
whereas the expected output is 
an approximate matrix basis for the space $\mathbb S_{r_1,A_1}$
or  $\mathbb T_{r_1,A_1}$ and 
a  rank-$r_1$ approximation to the matrix $A_1$,
within an expected error norm in
$O(\sigma_{r_1+1}(A_1))$. This is  the same order as
in  \cite[Algorithm 1]{OT10}, but now we do not use SVDs.
 One can further decrease the error bound by means
of small oversampling of Remark \ref{reovers} 
and power transform of Remark
\ref{regap}. 
     

\section{$2\times 2$ block triangulation of an ill conditioned matrix,
matrix inversion, and solving linear systems of
equations}\label{svianv1}


In this section we apply the results of Section \ref{sapsr} to 
compute $2\times 2$ block triangulation 
and the inverse of a nonsingular ill conditioned matrix.
In Section \ref{sourap},
which can be read independently
of the rest of the present section, we
 combine additive preconditioning and the SMW formula 
to solve a linear system of equations with such matrix.
One can alternatively 
combine dual additive preprocessing
with the dual SMW formula or
apply these techniques 
to precondition the input matrix.
(See Remarks \ref{reduls} and \ref{rezmr}).
We partition some algorithms 
of this section 
 into {\em symbolic} and {\em numerical} stages.
At the former stage we perform computations with infinite precision,
but they 
use a small fraction of the double precision 
 flops involved. 


\subsection{Block triangulation using
 approximate 
trailing singular spaces}\label{svianv}

In Section \ref{saptr}
we have approximated 
leading and
 trailing singular spaces
of ill conditioned matrices 
by applying randomized
additive preprocessing,
random sampling, 
or augmentation.
Next we extend the former algorithms 
(based first on additive preprocessing
of Theorem \ref{thtrail} 
and Corollary \ref{cotrail}
and then on random sampling
of Section \ref{sapsr0}) to $2\times 2$ block triangulation
of these matrices. 
One can similarly compute their
block triangulation by extending either
the algorithms of
 Section \ref{saptr} based on randomized
 augmentation of Corollary \ref{cotrail0}
(we leave this to the reader)
or
the algorithms of  Sections \ref{sapsr0} and \ref{saplddl}
for approximate matrix bases of the leading singular spaces 
of the input matrix (see the next subsection).


\begin{protoalgorithm}\label{algprecnmb} {\bf Block triangulation  
with randomization and orthogonalization.}


\begin{description}


\item[{\sc Input:}] 
A  matrix 
$A\in \mathbb R^{n\times n}$
 whose norm
$||A||$  is neither large nor small, its numerical rank $q$ satisfying
$0<q=n-r<n$, and a Subroutine {\em LIN}$\cdot${\em SOLVE} that either solves a linear system 
of equations if it is nonsingular and well condtioned or outputs FAILURE otherwise.


\item[{\sc Output:}] FAILURE (with a low probability)
or four orthogonal matrices $K_0$ and $L_0$ in $\mathbb R^{n\times q}$ and 
$K_1$ and $L_1$ in $\mathbb R^{n\times r}$ such that 
 with a probability near $1$ the $q\times q$ block submatrix 
$W_{00}=K_0^TAL_0$ 
of the matrix 
$W=(K_0~|~K_1)^TA(L_0~|~L_1)=\begin{pmatrix} W_{00}  & W_{01}  \\
W_{10}  &  W_{11} \end{pmatrix}$ 
is nonsingular, well conditioned, and strongly dominant such that 
$\sigma_{q} (W_{00})\gg\max\{||W_{01}||,~||W_{10}||,~||W_{11}||\}$. 


\item[{\sc Computations}] (see Remark \ref{recancel}): $~$


\begin{enumerate}


\item 
Generate two matrices
$U,V \in \mathbb G_{0,1}^{n\times r}$.


\item 
Compute the matrix $C=A+UV^T$, expected  
to be nonsingular and well conditioned. 


\item 
Apply the Subroutine 
{\em LIN}$\cdot${\em SOLVE} to compute 
the matrices $C^{-T}V$ and $C^{-1}U$.
Stop and output FAILURE if so does the subroutine.


\item 
Compute and output two orthogonal matrices $K_1=Q(C^{-1}U)$
and $L_1=Q(C^{-T}V)$.


\item 
Compute and output two orthogonal nmbs $K_0=\nmb(K_1)$
and $L_0=\nmb(L_1)$.
\end{enumerate}


\end{description}


\end{protoalgorithm}


 The algorithm can only fail with a low 
probability by virtue of Theorems \ref{thkappa2} and \ref{thtrail} 
and Corollary
 \ref{cotrail}. 
 We use the following theorem
to prove  
correctness of the algorithm.


\begin{theorem}\label{thbgj}
For a matrix $A\in  \mathbb R^{m\times n}$ and $0<q<l=\min\{n, m\}$,
write $r=n-q$ and $\bar r=m-q$.
 Let $K_0\in \mathbb R^{m\times q}$,
$L_0\in \mathbb R^{n\times q}$, 
 $K_1\in \mathbb R^{m\times \bar r}$, 
$L_1\in \mathbb R^{n\times r}$, and $Q_K,Q_L\in \mathbb R^{r\times r}$
be six orthogonal matrices such that
$K_1=S_{A,\bar r}Q_K$,
$L_1=T_{A,r}Q_L$,
$K_1^TK_0=O_{\bar r,q}$
 and
$L_1^TL_0=O_{r,q}$.
Then $||K_1^TA||\le \sigma_{q+1}(A)$,
$||AL_1||\le \sigma_{q+1}(A)$,
$||K_0AL_0||=\sigma_{1}(A)$,
and $\kappa(K_0AL_0)=\sigma_1(A)/\sigma_q(A)$.
\end{theorem}


\begin{proof}
Suppose $A=S_A\Sigma_A T_A^T$ is SVD of (\ref{eqsvd}).
Then $AL_1=S_A\Sigma_A T_A^TT_{A,r}Q_L=
S_A\Sigma_A\begin{pmatrix} O_{n,q} \\
Q_L  \end{pmatrix}=S_{A}\diag(O_{m-r,n-r},(\sigma_j(A))_{j=q+1}^nQ_L)$,
and so $||AL_1||\le\sigma_{q+1}(A)$ because $S_{A}$ and $Q_L$
are orthogonal matrices.
Similarly obtain that $||K_1^TA||\le\sigma_{q+1}(A)$.

Next deduce from the assumptions about $L_0$ and $L_1$
 that $L_0=T_{q,A}Q'_0$ for an  orthogonal
matrix $Q'_0\in \mathbb R^{q\times q}$
and similarly that $K_0=S_{q,A}Q_0$ for an  orthogonal
matrix $Q_0\in \mathbb R^{q\times q}$.
Therefore $$K_0AL_0=Q'_0S_{q,A}^TS_A\Sigma_A T_A^TT_{q,A}Q_0=(Q'_0~|~O_{m,\bar r})\Sigma_A(Q_0~|~O_{r,n})^T=Q'_0\diag(\sigma_j(A))_{j=1}^qQ_0,$$
and so
$||K_0AL_0||=\sigma_{1}(A)$,
 $\kappa(K_0AL_0)=\sigma_1(A)/\sigma_q(A)$.
\end{proof}

In Proto-Algorithm \ref{algprecnmb}
we expect to have 
$\mathcal R(L_1)\approx \mathcal T_{A,r}$
by virtue of Theorem \ref{thtrail}
and similarly to have $\mathcal R(K_1)\approx \mathcal S_{A,r}$. 
 Theorem \ref{thbgj} implies
that the norms $||K_1^TA||$ and $||AL_1||$
have an upper bound close to $\sigma_{q+1}(A)$,
whereas $\kappa(K_0^TAL_0)\approx \sigma_1(A)/\sigma_q(A)$.
Now correctness of the algorithm follows because 
the matrix $A$ has numerical rank $q$.

We can 
proceed with nonorthogonal matrices 
$K_0$, $K_1$, $L_0$, $L_1$, $Q_K$, and
$Q_l$ to simplify the computations,
by weakening numerical stability a little.
Then we can still expect that 
the norms $||W_{01}||$, $||W_{10}||$, and
$||W_{11}||$
have at most order
$\sigma_{q+1}(A)$,
the norm $||W_{00}||$ has 
order 
$\sigma_{1}(A)$,
 and 
the condition number 
$\kappa (W_{00})$ has 
order 
 $\sigma_{1}(A)/\sigma_{q}(A)$.
 Moreover choosing random matrices
$K_0\in \mathcal G_{0,1}^{m\times q}$ and
$L_0\in \mathcal G_{0,1}^{n\times q}$,
which are
expected to be well conditioned
by virtue of Theorems \ref{thsiguna}
and \ref{thsignorm}  
combined,
and we can still extend our probabilistic estimates 
for the values $||W_{i,j}||$
for $i,j=1,2$ and $\kappa (W_{00})$.
Here is our respective simplified algorithm.
Our tests  in Section \ref{stails}
show its efficiency.


\begin{protoalgorithm}\label{algprecnmb1}  {\bf Simplified randomized block triangulation.}


\begin{description}


\item[{\sc Input, Output}] and Stages 1 and 2 
of  {\sc Computations} are the same as in
Proto-Algorithm \ref{algprecnmb}
except that the output matrices $K_0$, $L_0$,
$K_1$ and $L_1$ are no longer assumed to be orthogonal. 


\item[{\sc Computations:}] $~$


\medskip

3. 
Generate and output two random matrices
$K_0,L_0 \in \mathcal G_{0,1}^{n\times q}$.


\medskip

4. 
Apply the Subroutine 
{\em LIN}$\cdot${\em SOLVE} to compute and to output 
the matrices $K_1=C^{-T}V$ and $L_1=C^{-1}U$.
Output FAILURE 
 and stop
if so does the subroutine.


\end{description}


\end{protoalgorithm}

Proto-Algorithms \ref{algprecnmb} and \ref{algprecnmb1}
do not produce block triangulation but prepare it.
Having strong domination of the block $W_{00}$, 
we can readily compute the block 
 factorizations
$$W=\begin{pmatrix}
   I  &  O \\
  W_{10}W_{00}^{-1}  & I
  \end{pmatrix}\begin{pmatrix}
   W_{00}  &  W_{01}\\
  O  &    G
  \end{pmatrix}$$
 for
$ G=W_{11}-W_{10}W_{00}^{-1}W_{01}$ 
and 
$$W^{-1}=\begin{pmatrix}
  W_{00}^{-1}  &  -W_{00}^{-1}W_{01}G^{-1} \\
  O  & G^{-1}
  \end{pmatrix}
\begin{pmatrix}
   I  &  O\\
 -W_{10}W_{00}^{-1}  & I
  \end{pmatrix}.$$
Both Proto-Algorithms \ref{algprecnmb} and \ref{algprecnmb1}
reduce the inversion of the matrix $A$
to the inversion of the matrices $W_{00}$
and $G$ of smaller sizes, where both matrices 
are expected to be nonsingular and better conditioned
than the matrix $A$ (cf. \cite[Section 9]{PGMQ}).




\begin{remark}\label{recancel}
We expect to arrive at the matrices $W_{01}$, 
 $W_{10}$ and  $W_{11}$ having small norms.
To counter the expected cancellation of the leading digits of 
the $2rn-r^2$ entries of these matrices,
 we should 
compute the matrices $C$,
$K_1$ and $L_1$, 
 their products
by the blocks of the matrix $A$,
and the Schur complement $G$
with a high precision $p_+$
(or partly symbolically, with
infinite precision).
These computations
involve $O(n^2r)$ flops, that is 
 a $r/n$ fraction of order $n^3$
flops in high precision $p_+$
 required by Gaussian elimination.
See further study in
\cite[Section 9]{PGMQ}.
Having implemented this 
part of the computations with
higher precision,  
 we have 
outperformed the standard algorithms
(see Section \ref{sexgeneraltrail} and
Tables \ref{tablsnmb} and \ref{tablsge}).

\end{remark}


\subsection{Block triangulation 
using
 approximate
leading singular spaces}\label{svianvd}

   
Suppose a square matrix $A$
has a small positive numerical rank $q$ and define 
a dual variation of Proto-Algorithm \ref{algprecnmb}
by applying  Theorem \ref{thhead}.
In this case
matrix inversions are limited to the case of
$q\times q$ matrices $H$, $K_0^TK_0$ and $L_0^TL_0$. 
In our dual algorithm we assume that the nonsingular
input matrix $A$ has been scaled so that
the norm $||A^{-1}||$ is neither large nor small.
See some recipes for the approximation of 
this norm
at the end of Section \ref{sapaug}
and 
see Remark \ref{redualh1} on how to proceed without estimating this norm.


\begin{protoalgorithm}\label{algprecnmbd} {\bf Block triangulation   
using
 approximate
leading singular spaces.}


\begin{description}


\item[{\sc Input:}] 
A nonsingular  ill conditioned matrix 
$A\in \mathbb R^{n\times n}$  scaled so 
that the norm
$||A^{-1}||$ is neither large nor small; 
the numerical rank $q$ of the matrix $A$ such that
$0<q=n-r<n$,
 and a Subroutine {\em INVERT} 
that either inverts a matrix if it is nonsingular and well conditioned 
or outputs FAILURE otherwise.


\item[{\sc Output:}] FAILURE 
(with a low probability)
or four matrices $K_0,~L_0\in \mathbb R^{n\times q}$ and 
$K_1,~L_1\in \mathbb R^{n\times r}$ such that 
$$W=\begin{pmatrix} W_{00}  & W_{01}  \\
W_{10}  &  W_{11} \end{pmatrix}=(K_0~|~K_1)^TA(L_0~|~L_1)$$ 
and  the block submatrix 
$W_{00}=K_0^TAL_0$ is expected to be nonsingular, 
well conditioned, and strongly dominant 
such that 
$\sigma_{q} (W_{00})\gg\max\{||W_{01}||,~||W_{10}||,~||W_{11}||\}.$
 

\item[{\sc Computations}:] $~$

 
\begin{enumerate}


\item 
Generate two matrices
$U_-$ and $V_- $ in $\mathcal G_{0,1}^{n\times q}$.

\item 
Compute 
the matrix 
$H=I_q+V_-^TAU_-$ of (\ref{eqc--1}).

\item 
Apply the Subroutine 
{\em INVERT} to compute 
the matrix 
$H^{-1}$.
Output FAILURE 
 and stop
if so does the subroutine.
 
\item 
Compute 
the matrix 
$C_-=A-AU_-H^{-1}V_-^TA$ of (\ref{eqsmwd}).

\item 
Compute and output 
the matrices 
$K_0=C_-U_-/||C_-U_-||$ and $L_0=C_-^{T}V_-/||C_-^{T}V_-||$.


\item 
Compute
the matrices $M\approx \nmb(K_0^T)$
and $N\approx \nmb(L_0^T)$
(see our Section \ref{saptr}, \cite{PQ10} and \cite{PQa}
on the approximation of nmbs).

\item 
Compute and output 
the matrices $K_1=M/||M||$
and $L_1=N/||N||$.


\end{enumerate}


\end{description}


\end{protoalgorithm}


The algorithm fails with a low 
probability by virtue of Theorems \ref{thkappa2} and \ref{thhead}
and Corollary
 \ref{cohead}. Complete the
correctness proof by extending
Theorem \ref{thbgj}.

\begin{remark}\label{redualh}
As in the previous subsection,
we must perform a small fraction of our computations
with high accuracy.
Namely we must compute the matrix $H$ 
 with high or infinite precision, but for that we need
 $O(qn^2)$ flops,
 versus order $n^3$ high precision flops in 
Gaussian elimination. 
Unlike the previous subsection, this stage
involves only matrix multiplications, a matrix addition 
and no inversions, although we need matrix inversions
for computing nmbs at Stage 6.
\end{remark}

\begin{remark}\label{redualh1}
Instead of applying  Theorem \ref{thhead}
we can employ any other
algorithm that computes a pair of  
 approximate matrix bases $K_0$ and $L_0$ 
for the left and right  leading singular spaces.
E.g., we can apply 
randomized 
dual augmentation
of Corollary \ref{cotrail01} 
or just compute
$K_0=A^TV$ and $L_0=AU$
for $V\in \mathcal G_{0,1}^{q\times n}$
and  $U\in \mathcal G_{0,1}^{n\times q}$
(cf. Proto-Algorithm \ref{algbasmp}).
In a heuristic variation
we can choose
the matrices 
 $U,V^T\in \mathcal T_{0,1}^{n\times q}$
 where $A$ is a Toeplitz-like matrix
and where the numerical rank $q$ is not small.
The latter computation
requires no estimates for the norm  $||A^{-1}||$
and in our tests  has 
led to higher output accuracy
than  Proto-Algorithm \ref{algprecnmbd}.
For a further heuristic simplification
in the case where $n=2q$, choose the 
Toeplitz matrices $U$ and $V$ 
in the form $F=(Z~|~T)$, $Z\in \mathcal Z_{0,1}^{q\times n}$
and
${\bf e}_1^T\in \mathcal Z_{0,1}^{1\times n}$, 
and then $\begin{pmatrix} Z^{-1}T    \\
-I_q \end{pmatrix}$ is a nmb$(F)$.
\end{remark}


\subsection{Randomized additive preconditioning with the SMW recovery}\label{sourap}


Suppose that we seek the solution ${\bf y}=A^{-1}{\bf b}$ of a 
real nonsingular  ill conditioned
linear system $A{\bf y}={\bf b}$ of $n$ equations where
we are given a small upper bound $r$ on the numerical nullity 
of $A$. 
Assume that the norm $||A||$
is neither large nor small.
Then randomized additive 
preprocessing $A\Longrightarrow C=A+UV^T$ for $U,V\in \mathcal G_{0,1}^{n\times r}$ 
is expected to produce a well conditioned matrix $C$ (cf. Theorem \ref{thkappa2}). 
The  SMW formula implies that
${\bf y}=C^{-1}{\bf b}+C^{-1}UG^{-1}V^TC^{-1}{\bf b}~~{\rm for}~~G=I_r-V^TC^{-1}U$.
Substitute $X_U=C^{-1}U$ and ${\bf x}_{\bf b}=C^{-1}{\bf b}$ and obtain 

\begin{equation}\label{eqy}
{\bf y}={\bf x}_{\bf b}+X_UG^{-1}V^T{\bf x}_{\bf b}~{\rm for}~G=I_r-V^TX_U.
\end{equation}
This reduces the computation of  ${\bf y}$ essentially
to the solution of the matrix equation 
$CX=(U~|~{\bf b})$ for $X=(X_U~|~x_{\bf b})$,
computing the matrix $G$, 
and its inversion. 
The solution algorithm below incorporates iterative refinement at this stage.

\begin{protoalgorithm}\label{flch} {\bf Randomized Solution of a Linear System
with Iterative Refinement.}

\medskip

{\sc Input:}  a vector ${\bf b}\in \mathbb R^{n\times 1}$, 
a nonsingular ill conditioned matrix $A\in \mathbb R^{n\times n}$,
and its 
  numerical nullity $r=\nnul(A)$ (cf. Remark \ref{rennul}).

\medskip

{\sc Output:} A vector $\tilde{\bf y}\approx A^{-1}{\bf b}$. 

\medskip

{\sc Computations}: 
\begin{enumerate}
\item 
Generate two matrices $U,V\in \mathcal G_{0,\sigma}^{n\times r}$.
\item 
Compute  the matrix $C=A+UV^T$,  
expected to be nonsingular and well conditioned.
\item 
Apply Gaussian elimination (or another direct algorithm) 
to compute an approximate inverse $Y\approx C^{-1}$. (Perform the
computations in double precision.
Application of the same 
algorithm to the original ill conditioned linear system $A{\bf y}={\bf b}$ 
would require about as many flops but in extended precision.)
\item 
Apply 
 iterative refinement employing the approximate inverse $Y$ to compute sufficiently accurate 
solution $X_U$ of the matrix equation $CX_U=U$. (High accuracy is required to counter the 
cancelation of leading bits in the subsequent 
computation of the Schur complement $G=I_r-V^TC^{-1}U$.) Then 
recover a close approximation to the vector ${\bf y}=A^{-1}{\bf b}$ by applying 
equation (\ref{eqy}).
\end{enumerate}
\end{protoalgorithm}

The algorithm  reduces the original task of computations with ill conditioned matrix $A$ to the 
computations with the well conditioned matrix $C$ and $O(n^2r)$ additional flops.
To handle an ill conditioned input $A$,
we must 
perform computations with extended precision
to counter magnification of rounding errors, 
but we can confine this stage essentially 
to computing the Schur complement $G=I_r-V^TC^{-1}U$.
This is a small fraction of the computational time
of the customary algorithms for a linear system
$A{\bf y}={\bf b}$ provided the ratio $r/n$ is small
and the precision required to handle the ill conditioned matrix
$A$ is high.

Let us supply some estimates. 
To support iterative refienment
 we must use a precision $p$
exceeding $\log_2\kappa(C)$;
for well conditioned matrices $C$
we can assume that $p>2\log_2\kappa(C)$, say. 
Then
order $p-\log_2\kappa (C)$ new
correct bits are produced 
per  an output value
by a
loop of iterative refinement
(see \cite[Section 9]{PGMQ}),
reduced essentially  to multiplication of the matrices $C$ and $Y$ by $2r$ vectors, that is 
to $(4n-2)nr$ flops in a low  (e.g., double) precision $p$.
The refinement algorithm outputs order $rn$ values; 
one can   
accumulate them with high accuracy as the  sums of sufficiently
many low precision summands (as in symbolic lifting \cite{P11}).
Overall with this advanced implementation  we only perform
$O(rn^2p_+/p)$ flops in low precision $p$
at Stage 4 of Proto-Algorithm \ref{flch}.

For comparison Gaussian elimination uses $\frac{2}{3}n^3+O(n^2)$ flops
in extended  precision  $p_+\approx p_{\rm out}+\log_2\kappa (A)$ 
to output the solution to the ill conditioned linear system $A{\bf y}={\bf b}$  
with a prescribed precision $p_{\rm out}$.
We compute an approximate
inverse $Y$ of the well conditioned matrix $C$
at Stage 3
by using  $\frac{2}{3}n^3+O(n^2)$ flops as well,
but in the low
precision $p$.
The cost of performing Stages 1 and 2 is dominated,
and so our progress is significant
where $np\gg rp_+$ and $p_+$ greatly exceeds $p$.

The estimated computational cost further decreases 
where the matrices $A$, $U$ and $V$
have consistent patterns of sparseness and structure.
E.g.,
the decrease is by a factor $n$ where they are Toeplitz or 
 Toeplitz-like matrices.

\begin{remark}\label{reduls}
Given a nonsingular $n\times n$ matrix $A$ (with $||A^{-1}||\approx 1$)
and a small upper bound $q$ on its numerical rank,
we can 
define a dual variation of Proto-Algorithm \ref{flch} 
based on Corollary \ref{cohead}
as follows:
generate a pair of matrices $U_-,V_-\in \mathbb G_{0,1}^{n\times q}$
and then
 compute the matrices $H$ of (\ref{eqc--1}) and $C_-$ of (\ref{eqsmwd})
to reduce the solution of a linear system of equations
$A{\bf y}={\bf b}$ to computing the vector 
${\bf y}=(C_-^{-1}-U_-V_-^T){\bf b}$ (cf. (\ref{eqc--1})). 
The matrix $H=I_q+V_-^TAU_-$ must be computed with high
accuracy, but this only requires $O(qn^2)$ flops in high  
or infinite precision.
Furthermore, unlike the computations by means of the SMW formula,
we only need matrix multiplications and an addition 
(and no inversions) at this stage.
Similarly one can employ Corollary \ref{cotrail01}
and expression (\ref{eqhead+}) instead of 
Corollary \ref{cohead}.
\end{remark}

\begin{remark}\label{rennul} 
There is no point for applying Proto-Algorithm \ref{flch} 
where the matrix $A$ is  well conditioned 
or has numerical nullity
exceeding $r$. In the former case
 the preconditioning is not needed,
whereas in the latter case additive preprocessing  
would produce an ill conditioned matrix $C$.
In both cases preprocessing would give no benefits
but would involve  
extraneous computations and additional rounding errors.
In the case where $r$ is equal to the numerical nullity of $A$,
however, 
these deficiencies can be overwhelmed by 
the benefits of avoiding
order $n^3$ high precision flops.
\end{remark}

\begin{remark}\label{rezmr}
Instead of using additive preconditioners
directly for solving linear systems,
one can combine them with the SMW or dual SMW formulae
to obtain multiplicative preconditioners.
Assume a nonsingular ill conditioned $n\times n$ 
matrix $A$ and let $A_-$ denote
the inverse $A^{-1}$ computed with double precision.
Then the matrices  $A_-A$ and $AA_-$
are much better conditioned than the matrix $A$
according to the experiments of \cite{R90}.
Note that both linear systems
of equations $A_-A{\bf y}=A_-{\bf b}$
and $AA_-{\bf x}={\bf b}$, for ${\bf y}=A_-{\bf x}$,
are equivalent to the system
$A{\bf y}={\bf b}$.
This empirical technique is interesting itself
and can probably be advanced by means of its recursive application.
It may also accentuate the preconditioning power of 
our randomized preprocessing. Instead of defining 
the multiplier $A_-$ as
the inverse $A^{-1}$  computed with double precision,
we can compute this multiplier as $C^{-1}(I_n-UG^{-1}V^TC^{-1})$ by applying
the SMW formula. Moreover we can drop the factor $C^{-1}$
and write either $A_-=I_n-UG^{-1}V^TC^{-1}$ or  $A_-=I_n-C^{-1}UG^{-1}V^T$ 
to have $AA_-=A-AUG^{-1}V^TC^{-1}$ and $A_-A=A-UG^{-1}V^TC^{-1}A$ or
$AA_-=A-AC^{-1}UG^{-1}V^T$ and $A_-A=A-C^{-1}UG^{-1}V^TA$.
We can similarly utilize the dual SMW formula of (\ref{eqc--1}) and (\ref{eqsmwd}),
in which case we can compute the multiplier
$A_-=C_-^{-1}-U_-V_-^T$ for $C_-$ of (\ref{eqsmwd}),
and then we would have
$AA_-=(I-AUH^{-1}V^T)^{-1}-AU_-V_-^T$
and $A_-A=(I-UH^{-1}V^TA)^{-1}-U_-V_-^TA$
for $H=I_q+V_-^AU_-T$ of (\ref{eqc--1}).
\end{remark}



\section{Numerical Experiments}\label{sexp}

 
Our numerical experiments with random general, Hankel, Toeplitz and circulant matrices 
have been performed in the Graduate Center of the City University of New York 
on a Dell server with a dual core 1.86 GHz
Xeon processor and 2G memory running Windows Server 2003 R2. The test
Fortran code was compiled with the GNU gfortran compiler within the Cygwin
environment.  Random numbers were generated with the random\_number
intrinsic Fortran function, assuming the uniform probability distribution 
over the range $\{x:-1 \leq x < 1\}$.  The tests have been designed by the first author 
and performed by his coauthors.


\subsection{Conditioning tests}\label{scondtests}


We have computed the condition numbers of random general $n\times n$ matrices for 
$n=2^k$, $k=5,6,\dots,$ with entries sampled in the range $[-1,1)$ as well as 
complex general, Toeplitz, and circulant matrices 
whose entries had real and imaginary parts sampled at random in the same range $[-1,1)$. 
We performed 100 tests for each class of inputs, each dimension $n$,  and each nullity $r$.
 Tables \ref{tab01}--\ref{tabcondcirc} display
 the test results. The last four columns of each table 
display the average (mean), minimum, maximum, and standard deviation
of the computed condition numbers of the input matrices, respectively. Namely we 
computed
the values $\kappa (A)=||A||~||A^{-1}||$ for general, Toeplitz, and circulant matrices $A$ and
the values $\kappa_1 (A)=||A||_1~||A^{-1}||_1$ for Toeplitz matrices $A$.
We computed and displayed in Table \ref{tabcondtoep} the 1-norms of 
Toeplitz matrices and their inverses rather than their 2-norms 
to facilitate the computations in the case of inputs of large sizes.
Relationships (\ref{eqnorm12}) link 
 the 1-norms and 2-norms to one another, but 
the empirical data in 
Table \ref{nonsymtoeplitz} consistently show 
even  closer links,
in all cases of
general, Toeplitz, and circulant  $n\times n$ matrices $A$ where
$n=32,64,\dots, 1024$. 


\subsection{Preconditioning tests}\label{sprecondtests}\label{sprec}


Table \ref{tabprec}  covers our tests for the preconditioning power 
of additive preprocessing in \cite{PIMR10}.
We have tested the input matrices of the following classes.


1n. {\em Nonsymmetric matrices of type I with numerical nullity $r$.} 
 $A=S\Sigma_{r}T^T$ are $n\times n$
matrices where $S$ and $T$ are $n\times n$ random orthogonal matrices, that is,   
the factors $Q$ in the QR factorizations of random real matrices;
$\Sigma_{r}=\diag (\sigma_j)_{j=1}^n$ is the diagonal matrix such that $\sigma_{j+1}\leq \sigma_j$ for 
$j=1,\dots,n-1, ~\sigma_1=1$, 
the values $\sigma _2,\dots, \sigma_{n-r-1}$ are randomly sampled in the semi-open
interval $[0.1,1)$, $~\sigma_{n-r}=0.1,~\sigma_j =10^{-16}$ for $j=n-r+1,\dots, n,$
and therefore $\kappa (A)=10^{16}$  \cite[Section 28.3]{H02}. 


1s. {\em Symmetric matrices of type I with numerical nullity $r$.}
The same as in part 1n, but for $S=T$.


The matrices of the six other classes have been constructed in the form of $\frac{A}{||A||}+\beta I$
where the recipes for defining the matrices $ A$ and scalars $\beta$ are specified below.

2n. {\em Nonsymmetric matrices of type II with numerical nullity $r$.}
 $ A=(W~|~WZ)$ where $W$ and $Z$ are random orthogonal matrices of sizes $n\times (n-r)$ and 
$(n-r)\times r$, respectively.


2s. {\em Symmetric matrices of type II with numerical nullity $r$.}
 $ A=WW^T$ where $W$ are random orthogonal matrices of size $n\times (n-r)$.


3n. {\em Nonsymmetric Toeplitz-like matrices with numerical nullity $r$.} 
$ A=c(T~|~TS)$ for random Toeplitz matrices $T$ of size $n\times (n-r)$ and $S$ 
of size $(n-r)\times r$ and for a positive scalar $c$ such that $||A||\approx 1$.

 
3s. {\em Symmetric Toeplitz-like matrices with numerical nullity $r$.} 
$ A=cTT^T$ for random Toeplitz matrices $T$ of size $n\times (n-r)$ and 
a positive scalar $c$ such that $||A||\approx 1$.

 
4n. {\em Nonsymmetric Toeplitz matrices with numerical nullity $1$.} 
$ A=(a_{i,j})_{i,j=1}^{n}$ is a Toeplitz  $n\times n$ matrix. Its entries 
$a_{i,j}=a_{i-j}$ are random for $i-j<n-1$. The entry $a_{n,1}$ is selected
 to ensure that the last row is linearly expressed through the other rows.


4s. {\em Symmetric Toeplitz matrices with numerical nullity $1$.} 
$ A=(a_{i,j})_{i,j=1}^{n}$ is a Toeplitz  $n\times n$ matrix. Its entries 
$a_{i,j}=a_{i-j}$ are random for $|i-j|<n-1$, whereas the entry $a_{1,n}=a_{n,1}$ 
is a root of the quadratic equation $\det A=0$. We repeatedly generated the
 matrices $A$ until we arrived at the quadratic equation having real roots.

We set $\beta=10^{-16}$ for symmetric matrices $ A$
in the classes 2s, 3s, and 4s, so that $\kappa (A)=10^{16}+1$ in these cases.  
For nonsymmetric matrices $ A$ we defined the scalar $\beta$ by an iterative
process such that $||A||\approx 1$ and $10^{-18}||A||\le \kappa (A)\le 10^{-16}||A||$
\cite[Section 8.2]{PIMR10}.
Table \ref{tabprec} displays the average values of 
the condition numbers $\kappa (C)$ of the matrices 
$C=A+UU^T$ over 100,000 tests for the inputs in the above classes, 
$r=1,2,4,8$ and $n=100$. We  defined the additive preprocessor $UU^T$
 by a normalized $n\times r$ matrix $U= U/|| U||$ where  
$ U^T=(\pm I~|~O_{r,r}~|~ \pm I~|~ O_{r,r}~|~\dots~|~O_{r,r}~|~\pm I~|~O_{r,s})$, 
we chosen the integer $s$
to obtain $n\times r$ matrices $ U$ and  chosen
the signs for the matrices $\pm I$
 at random.
In our further tests the condition numbers of the matrices $C=A+10^pUV^T$ 
for $p=-10,-5,5,10$ were steadily growing within a factor $10^{|p|}$
as the value $|p|$ was growing. This  showed the importance of proper scaling 
of the additive preprocessor $UV^T$.


\subsection{GENP
with random circulant multipliers}\label{sexgeneral}


Table \ref{tab44} shows the results of our tests of the solution 
of a nonsingular well conditioned linear system $A{\bf y}={\bf b}$ of $n$ equations 
whose coefficient matrix had ill conditioned
$n/2\times n/2$   leading principal block for $n=64, 256,1024$.
We  performed 100 numerical tests for each dimension $n$ and computed 
the maximum, minimum and average relative residual norms 
$||A{\bf y}-{\bf b}||/||{\bf b}||$
as well as standard deviation. 
GENP applied to these systems has output corrupted solutions with residual norms
ranging from 10 to $10^8$. When we preprocessed the systems with circulant multipliers
filled with $\pm 1$ (choosing the $n$ signs $\pm$ at random), the norms decreased to at
worst $10^{-7}$ for all inputs. Table \ref{tab44} also shows further decrease of the norm
in a single step of iterative refinement. Table 2 in \cite{PQZa}
shows similar results of the tests where  
the input matrices were chosen similarly
but so that all their leading blocks 
 had numerical nullities $0$ and $1$
and where  Householder multipliers $I_n-{\bf u}{\bf v}^T/{\bf u}^T{\bf v}$
replaced the circulant multipliers. Here ${\bf u}$ and ${\bf v}$
denote two vectors filled with integers $1$ and $-1$ under 
random choice of the signs $+$ and $-$.


\subsection{Approximation of the tails and heads of SVDs and 
low-rank appro\-xi\-ma\-tion of 
a matrix }\label{stails}  


At some specified stages of our tests of this subsection and Section 
\ref{sexgeneraltrail}
we performed  
additions, subtractions and multiplications
with infinite precision
(hereafter referred to as {\em error-free ring operations}). 
At the other stages we performed computations 
 with double 
precision, and we rounded to 
double 
precision
all random values. We performed two refinement iterations for 
the computed solution of every linear system of equations
and matrix inverse.

Table \ref{tabSVD_TAIL}   
shows the data from our tests on the approximation
of trailing singular spaces of the SVD of an $n\times n$
matrix $A$ having 
numerical nullity  $r=n-q$ and on the
 approximation of this matrix with a matrix of rank $q=n-r$. 
For $n=64, 128, 256$ and  $q=1,8,32$ we  generated $n\times n$ random 
orthogonal matrices $S$ and $T$ and diagonal matrices 
$\Sigma=\diag(\sigma_j)_{j=1}^n$ such that $\sigma_j=1/j,~j=1,\dots,q$,
$\sigma_j=10^{-10},~j=q+1,\dots,n$ (cf. \cite[Section 28.3]{H02}). 
Then we applied error-free  ring operations to compute the input matrices
$A=S_A\Sigma_A T_A^T$, for which $||A||=1$ and $\kappa (A)=10^{10}$.
Furthermore we  generated pairs of random 
  $n\times r$ matrices $U$ and $V$ for $r=1,8,32$,
scaled them to  $||UV^T||\approx 1$,
and computed the matrices $C=A+UV^T$
(by applying error-free  ring operations), 
$B_{A,r}=C^{-1}U$ (by using two refinement iterations), 
$T_{A,r}$, 
 $B_{A,r}Y_{A,r}$ as a least-squares approximation to $T_{A,r}$,
$Q=Q(B_{A,r})$, and 
$A-AQQ^T$ (by applying error-free  ring operations).
Table \ref{tabSVD_TAIL} summarizes the data on the values
$\kappa (C)$ and the residual norms  
${\rm rn}_1=||B_{A,r}Y_{A,r}-T_{A,r}||$ 
and ${\rm rn}_2=||A-AQQ^T||$ observed in 
100 runs of our tests for every pair of $n$ and $q$.

We  performed 
 similar tests on the approximation
of leading singular spaces of the SVDs of the same $n\times n$
matrices $A$ having numerical rank $q$ 
and 
numerical nullity  $r=n-q$ and on the
 approximation of this matrix with a matrix of rank $q$. 
In some tests  
we employed dual additive preprocessing
to approximate matrix bases 
 for the leading singular spaces $\mathbb T_{q,A}$.
We have generated the pairs of $n\times q$ random matrices $U_-$ and $V_-$ for $q=1,8,32$,
scaled them to have $||U_-V_-^T||\approx ||A^{-1}||=10^{10}$,
and successively computed the matrices 
$H=I_q+V_-^TAU_-$ of (\ref{eqc--1}) 
(by applying error-free  ring operations), $H^{-1}$
(by using two refinement iterations),
$C_-=A-AU_-H^{-1}V_-^TA$ of (\ref{eqsmwd}),
 $B_{q,A}=C_-^TV_-$, 
 $T_{q,A}$,  
$B_{q,A}Y_{q,A}$  as a least-squares approximation to $T_{q,A}$, 
$Q_{q,A}=Q(B_{q,A})$, and 
$A-AQ_{q,A}(Q_{q,A})^T$ (by applying error-free  ring operations). 
Table \ref{tabSVD_HEAD}  summarizes the data on the condition numbers
$\kappa (C_-)$ and the residual norms 
${\rm rn}^{(1)}=||B_{q,A}Y_{q,A}-T_{q,A}||$ 
and
${\rm rn}^{(2)}=||A-AQ_{q,A}(Q_{q,A})^T||$ 
obtained in 100 runs of our tests
for every pair of $n$ and $q$.

We have also performed similar tests where  
we generated random $n\times q$ matrices $U$ (for $q=1,8,32$)
and random Toeplitz $n\times q$ matrices $\bar U$ 
(for $q=8,32$)
and then
replaced the above matrix $B_{q,A}$ with the approximate matrix bases $A^TU$
and $A^T\bar U$
for the leading singular space $\mathbb T_{q,A}$.
Tables \ref{tabSVD_HEAD1} and \ref{tabSVD_HEAD1T}  display the results of these tests.
In both cases the residual norms are equally small and are 
about as small as in Tables  \ref{tabSVD_TAIL}
and \ref{tabSVD_HEAD}.


\subsection{$2\times 2$ block factorization and solving linear systems of
equations}\label{sexgeneraltrail}


For our next tests we have chosen $n=32,64$ and $r=1,2,4$ and for every pair $\{n,r\}$ generated 100
instances of vectors ${\bf b}$ and matrices $A$, $U$, and $V$ as
follows.
We  generated (a) random vectors  ${\bf b}$ of dimension $n$, 
(b) 
random real orthogonal $n\times n$ matrices $S$ and $T$, 
 and (c) random $n\times r$ matrices $U$ and $V$, which we
 scaled to have $||U||=||V||=1$.
Then we defined the matrices $\Sigma=\diag(\sigma_j)_{j=1}^n$, with
 $\sigma_{n-j}=10^{-17}$ for $j=0,1,\dots,r-1$, and $\sigma_{n-j}=1/(n-j)$ for $j=r,\dots,n-1$,
and then applied error-free  ring operations to
compute the matrices $A=S\Sigma T^T$.
Note that $||A||=1$ and $||A^{-1}||=10^{17}$.

For every such input
we solved the linear systems $A{\bf y}={\bf b}$
by applying Proto-Algorithm \ref{algprecnmb}. We first generated  random
 $n\times (n-r)$ matrices $K_0$ and $L_0$ and
then 
computed the matrices $C=A+UV^T$ (by applying error-free  ring operations), 
$K_1=C^{-T}V$ and $L_1=C^{-1}U$
(by using two refinement iterations), and
$W=(K_{0}~|~K_{1})^TA(L_0~|~L_1)=\begin{pmatrix} W_{00}  & W_{01}  \\
W_{10} & W_{11} \end{pmatrix}$
(by applying error-free  ring operations). In all our tests 
the matrices $C$ were nonsingular and well 
conditioned, and
the  leading 
principal $(n-r)\times(n-r)$ blocks $W_{00}=K_0^TAL_0$ were
well conditioned and strongly dominated the three other blocks $W_{01}$,
$W_{10}$, and $W_{11}$ in the $2\times 2$ block matrices $W$,
as we expected to see by virtue of our analysis in Section \ref{svianv}. 
Then we computed the vector $\widehat{\bf b}=(K_0~|~K_1)^T{\bf b}$
(by applying error-free  ring operations) and
solved the linear system $W{\bf x}=\widehat{\bf b}$
(by using two refinement iterations). Finally
we computed the vector ${\bf y}=(L_0~|~L_1){\bf x}$ (by applying
error-free  ring operations).
Table \ref{tablsnmb} shows the average (mean) values 
 of the 
relative residual norms $||A{\bf y}-{\bf b}||/||{\bf b}||$
of the output vectors ${\bf y}$ (these values range about $10^{-10}$)
as well as the minimums, maximums, and 
standard deviations in these tests.
For the same  ill conditioned inputs  
the Subroutine MLDIVIDE(A,B) for Gaussian elimination from MATLAB
has produced corrupted outputs, as can be seen from 
 Table \ref{tablsge}.

We have also performed similar tests for $n=32,64$ and $n\times n$ matrices $A$
and vectors ${\bf b}$. We generated them
as before, but for $q=n-r=1,2,4$,
and then we computed  orthogonal matrices $K_0$, $K_1$, $L_0$ and $L_1$
by employing dual additive preprocessing and 
Proto-Algorithm \ref{algprecnmbd}. 
We first generated and scaled the pairs of
random 
 $n\times q$ 
matrices 
$U_-$ and $V_-$ 
such that $||U_-||\approx ||V_-||\approx 3*10^8$, 
and so 
 $||U_-||~||V_-||\approx ||A^{-1}||=10^{17}$. 
Then we successively computed the matrices $H$ 
(by applying error-free  ring operations),
$H^{-1}$ (by using two refinement iterations),
$C_-=A-AU_-H^{-1}V_-^TA$ of (\ref{eqsmwd}),
 $C_-U_-$ and $C_-^TV_-$ (all of them by applying error-free 
 ring operations), 
$K_0=Q(C_-U_-)$, 
$L_0=Q(C_-^TV_-)$, $K_1=Q(\nmb(K_0^T))$ and $L_1=Q(\nmb(L_0^T))$,
and continued  as in the tests for Table \ref{tablsnmb}.
We  displayed  the results 
in Table \ref{tablsnmbd},
showing the
residual norms of the order $10^{-9}$
on the average.

Furthermore we have performed similar tests where we 
first generated random $n\times q$ matrices $U$ and $V$
and then computed the matrix products
$A^TV$ and $AU$  (by applying error-free  ring operations),
and
replaced the above matrices 
$K_0$ and $L_0$ by $K_0=Q(A^TV)$ and $L_0=Q(AU)$. 
 Table \ref{tablsmrm} displays the results of these tests,
showing the
residual norms of order $10^{-25}$
on the average.
Then again for the same  ill conditioned inputs  
the Subroutine MLDIVIDE(A,B) for Gaussian elimination from MATLAB
 produced corrupted outputs, similarly to the results  in 
 Table \ref{tablsge}.

We applied 
Proto-Algorithm \ref{flch} to
solve linear systems of equations with the same inputs as above
for small integers $r$. 
In these computations we
used two refinement iterations for 
computing an approximate 
 inverse $Y\approx G^{-1}$ and 
 the solutions $X_U$ and ${\bf x}_{\bf b}$ to 
the $r+1$ linear systems 
of equations $CX_U=U$
and $C{\bf x}_{\bf b}={\bf b}$, all
with the matrix $C$. 
We computed 
 the following matrices and vectors
by applying error-free  ring operations,
$C=A+UV^T$, $G=I_r-V^TX_U$, and 
${\bf y}={\bf x}_{\bf b}+X_UYV^T{\bf x}_{\bf b}$ of (\ref{eqy}).
Table \ref{tablsmwsmnl} 
shows the test results  
for the same inputs as we used for tests of Table \ref{tablsnmb},
except that now we 
have doubled the matrix size
to $n=64$ and $n=128$.



\subsection{Solution of a real symmetric Toeplitz linear system of equations 
with randomized augmentation}\label{sexphank}


We have solved 100  real symmetric linear systems of equations $T{\bf y}={\bf b}$ for each $n$ 
where we used vectors ${\bf b}$ with random coordinates from the range $[-1,1)$ 
and Toeplitz matrices $T=S+10^{-9}I_n$ for a singular symmetric 
Toeplitz  $n\times n$ matrices $S$ having rank $n-1$ and nullity $1$
and generated according to the recipe in \cite[Section 10.1b]{PQ10}.
Table \ref{tabhank} shows the average
CPU time of the solutions by our Algorithm 6.1 
and, for comparison, based on the QR factorization and SVD,
which we computed by applying 
the LAPACK procedures DGEQRF and DGESVD, respectively.
To solve the auxiliary Toeplitz linear system 
$K{\bf x}={\bf e}_1$ in Algorithm 6.1,
we  
first employed the Toeplitz linear solver
 of \cite{KV99}, \cite{V99}, 
\cite {VBHK01},  and \cite{VK98} and then applied 
 iterative refinement with double precision.

The abbreviations ``Alg. 6.1", ``QR", and ``SVD" indicate  
the respective algorithms.
The last two columns of the table display the ratios of these data 
in the first and the two other columns.
We measured the CPU time with the mclock function by counting cycles. One can 
convert them into seconds by dividing their number by a constant CLOCKS\_PER\_SEC, which is
1000 on our platform. 
We marked the table entries by a "-" where the tests were running too long 
and have not been completed.  
We obtained the solutions ${\bf y}$ with the relative residual norms of
about $10^{-15}$ in all three algorithms, which showed that
Algorithm 6.1 employing iterative refinement was as reliable 
as the QR and SVD based solutions but ran much faster. 
We refer the reader to \cite[Table 3]{PQZC} on similar test results for 
the solution of ill conditioned homogeneous Toeplitz linear systems.


\begin{table}[h]
\begin{center}
\caption{The norms of random general, Toeplitz and circulant matrices and of their inverses}
\label{nonsymtoeplitz}
\begin{tabular}{|c|c|c|c|c|c|c|c|}
\hline
\textbf{matrix $A$}&\textbf{$n$}&\textbf{$||A||_1$}&\textbf{$||A||_2$}&\textbf{$\frac{||A||_1}{||A||_2}$}&\textbf{$||A^{-1}||_1$}&\textbf{$||A^{-1}||_2$}&\textbf{$\frac{||A^{-1}||_1}{||A^{-1}||_2}$}\\\hline
General & $32$ & $1.9\times 10^{1}$ & $1.8\times 10^{1}$ & $1.0\times 10^{0}$ & $4.0\times 10^{2}$ & $2.1\times 10^{2}$ & $1.9\times 10^{0}$ \\ \hline
General & $64$ & $3.7\times 10^{1}$ & $3.7\times 10^{1}$ & $1.0\times 10^{0}$ & $1.2\times 10^{2}$ & $6.2\times 10^{1}$ & $2.0\times 10^{0}$ \\ \hline
General & $128$ & $7.2\times 10^{1}$ & $7.4\times 10^{1}$ & $9.8\times 10^{-1}$ & $3.7\times 10^{2}$ & $1.8\times 10^{2}$ & $2.1\times 10^{0}$ \\ \hline
General & $256$ & $1.4\times 10^{2}$ & $1.5\times 10^{2}$ & $9.5\times 10^{-1}$ & $5.4\times 10^{2}$ & $2.5\times 10^{2}$ & $2.2\times 10^{0}$ \\ \hline
General & $512$ & $2.8\times 10^{2}$ & $3.0\times 10^{2}$ & $9.3\times 10^{-1}$ & $1.0\times 10^{3}$ & $4.1\times 10^{2}$ & $2.5\times 10^{0}$ \\ \hline
General & $1024$ & $5.4\times 10^{2}$ & $5.9\times 10^{2}$ & $9.2\times 10^{-1}$ & $1.1\times 10^{3}$ & $4.0\times 10^{2}$ & $2.7\times 10^{0}$ \\ \hline
Toeplitz & $32$ & $1.8\times 10^{1}$ & $1.9\times 10^{1}$ & $9.5\times 10^{-1}$ & $2.2\times 10^{1}$ & $1.3\times 10^{1}$ & $1.7\times 10^{0}$ \\ \hline
Toeplitz & $64$ & $3.4\times 10^{1}$ & $3.7\times 10^{1}$ & $9.3\times 10^{-1}$ & $4.6\times 10^{1}$ & $2.4\times 10^{1}$ & $2.0\times 10^{0}$ \\ \hline
Toeplitz & $128$ & $6.8\times 10^{1}$ & $7.4\times 10^{1}$ & $9.1\times 10^{-1}$ & $1.0\times 10^{2}$ & $4.6\times 10^{1}$ & $2.2\times 10^{0}$ \\ \hline
Toeplitz & $256$ & $1.3\times 10^{2}$ & $1.5\times 10^{2}$ & $9.0\times 10^{-1}$ & $5.7\times 10^{2}$ & $2.5\times 10^{2}$ & $2.3\times 10^{0}$ \\ \hline
Toeplitz & $512$ & $2.6\times 10^{2}$ & $3.0\times 10^{2}$ & $8.9\times 10^{-1}$ & $6.9\times 10^{2}$ & $2.6\times 10^{2}$ & $2.6\times 10^{0}$ \\ \hline
Toeplitz & $1024$ & $5.2\times 10^{2}$ & $5.9\times 10^{2}$ & $8.8\times 10^{-1}$ & $3.4\times 10^{2}$ & $1.4\times 10^{2}$ & $2.4\times 10^{0}$ \\ \hline   
Circulant & $32$ & $1.6\times 10^{1}$ & $1.8\times 10^{1}$ & $8.7\times 10^{-1}$ & $9.3\times 10^{0}$ & $1.0\times 10^{1}$ & $9.2\times 10^{-1}$ \\ \hline
Circulant & $64$ & $3.2\times 10^{1}$ & $3.7\times 10^{1}$ & $8.7\times 10^{-1}$ & $5.8\times 10^{0}$ & $6.8\times 10^{0}$ & $8.6\times 10^{-1}$ \\ \hline
Circulant & $128$ & $6.4\times 10^{1}$ & $7.4\times 10^{1}$ & $8.6\times 10^{-1}$ & $4.9\times 10^{0}$ & $5.7\times 10^{0}$ & $8.5\times 10^{-1}$ \\ \hline
Circulant & $256$ & $1.3\times 10^{2}$ & $1.5\times 10^{2}$ & $8.7\times 10^{-1}$ & $4.7\times 10^{0}$ & $5.6\times 10^{0}$ & $8.4\times 10^{-1}$ \\ \hline
Circulant & $512$ & $2.6\times 10^{2}$ & $3.0\times 10^{2}$ & $8.7\times 10^{-1}$ & $4.5\times 10^{0}$ & $5.4\times 10^{0}$ & $8.3\times 10^{-1}$ \\ \hline
Circulant & $1024$ & $5.1\times 10^{2}$ & $5.9\times 10^{2}$ & $8.7\times 10^{-1}$ & $5.5\times 10^{0}$ & $6.6\times 10^{0}$ & $8.3\times 10^{-1}$ \\ \hline
\end{tabular}
\end{center}
\end{table}


\begin{table}[h]
  \caption{The condition numbers $\kappa (A)$ of random matrices $A$}
  \label{tab01}
  \begin{center}
    \begin{tabular}{| c | c | c | c | c | c |}
    \hline
\bf{$n$}&\bf{input} & \bf{min}  &\bf{max} &\bf{mean} &\bf{std} \\ \hline

$ 32   $ & $ {\rm real} $ & $2.4\times 10^{1}$ & $1.8\times 10^{3}$ & $2.4\times 10^{2}$ & $3.3\times 10^{2}$ \\ \hline
$ 32   $ & $ {\rm complex} $ & $2.7\times 10^{1}$ & $8.7\times 10^{2}$ & $1.1\times 10^{2}$ & $1.1\times 10^{2}$ \\ \hline
$ 64   $ & $ {\rm real} $ & $4.6\times 10^{1}$ & $1.1\times 10^{4}$ & $5.0\times 10^{2}$ & $1.1\times 10^{3}$ \\ \hline
$ 64   $ & $ {\rm complex} $ & $5.2\times 10^{1}$ & $4.2\times 10^{3}$ & $2.7\times 10^{2}$ & $4.6\times 10^{2}$ \\ \hline
$ 128  $ & $ {\rm real} $ & $1.0\times 10^{2}$ & $2.7\times 10^{4}$ & $1.1\times 10^{3}$ & $3.0\times 10^{3}$ \\ \hline
$ 128  $ & $ {\rm complex} $ & $1.3\times 10^{2}$ & $2.5\times 10^{3}$ & $3.9\times 10^{2}$ & $3.3\times 10^{2}$ \\ \hline
$ 256  $ & $ {\rm real} $ & $2.4\times 10^{2}$ & $8.4\times 10^{4}$ & $3.7\times 10^{3}$ & $9.7\times 10^{3}$ \\ \hline
$ 256  $ & $ {\rm complex} $ & $2.5\times 10^{2}$ & $1.4\times 10^{4}$ & $1.0\times 10^{3}$ & $1.5\times 10^{3}$ \\ \hline
$ 512  $ & $ {\rm real} $ & $3.9\times 10^{2}$ & $7.4\times 10^{5}$ & $1.8\times 10^{4}$ & $8.5\times 10^{4}$ \\ \hline
$ 512  $ & $ {\rm complex} $ & $5.7\times 10^{2}$ & $3.2\times 10^{4}$ & $2.3\times 10^{3}$ & $3.5\times 10^{3}$ \\ \hline
$ 1024 $ & $ {\rm real} $ & $8.8\times 10^{2}$ & $2.3\times 10^{5}$ & $8.8\times 10^{3}$ & $2.4\times 10^{4}$ \\ \hline
$ 1024 $ & $ {\rm complex} $ & $7.2\times 10^{2}$ & $1.3\times 10^{5}$ & $5.4\times 10^{3}$ & $1.4\times 10^{4}$ \\ \hline
$ 2048 $ & $ {\rm real} $ & $2.1\times 10^{3}$ & $2.0\times 10^{5}$ & $1.8\times 10^{4}$ & $3.2\times 10^{4}$ \\ \hline
$ 2048 $ & $ {\rm complex} $ & $2.3\times 10^{3}$ & $5.7\times 10^{4}$ & $6.7\times 10^{3}$ & $7.2\times 10^{3}$ \\ \hline  
    \end{tabular}
  \end{center}
\end{table}

\begin{table}[h]
\caption{The condition numbers $\kappa_1 (A)=\frac{||A||_1}{||A^{-1}||_1}$ of random Toeplitz matrices $A$}
\label{tabcondtoep}
\begin{center}
\begin{tabular}{|c|c|c|c|c|}
\hline
\textbf{$n$}&\textbf{min}&\textbf{mean}&\textbf{max}&\textbf{std}\\\hline
$256$ & $9.1\times 10^{2}$ & $9.2\times 10^{3}$ & $1.3\times 10^{5}$ & $1.8\times 10^{4}$  \\ \hline
$512$ & $2.3\times 10^{3}$ & $3.0\times 10^{4}$ & $2.4\times 10^{5}$ & $4.9\times 10^{4}$  \\ \hline
$1024$ & $5.6\times 10^{3}$ & $7.0\times 10^{4}$ & $1.8\times 10^{6}$ & $2.0\times 10^{5}$ \\ \hline
$2048$ & $1.7\times 10^{4}$ & $1.8\times 10^{5}$ & $4.2\times 10^{6}$ & $5.4\times 10^{5}$ \\ \hline
$4096$ & $4.3\times 10^{4}$ & $2.7\times 10^{5}$ & $1.9\times 10^{6}$ & $3.4\times 10^{5}$ \\ \hline
$8192$ & $8.8\times 10^{4}$ & $1.2\times 10^{6}$ & $1.3\times 10^{7}$ & $2.2\times 10^{6}$ \\ \hline
\end{tabular}
\end{center}
\end{table}

\begin{table}[h]
\caption{The condition numbers $\kappa (A)$ of random circulant matrices $A$}
\label{tabcondcirc}
\begin{center}
\begin{tabular}{|c|c|c|c|c|}
\hline
\textbf{$n$}&\textbf{min}&\textbf{mean}&\textbf{max}&\textbf{std}\\\hline
$256$ & $9.6\times 10^{0}$ & $1.1\times 10^{2}$ & $3.5\times 10^{3}$ & $4.0\times 10^{2}$ \\ \hline
$512$ & $1.4\times 10^{1}$ & $8.5\times 10^{1}$ & $1.1\times 10^{3}$ & $1.3\times 10^{2}$ \\ \hline
$1024$ & $1.9\times 10^{1}$ & $1.0\times 10^{2}$ & $5.9\times 10^{2}$ & $8.6\times 10^{1}$ \\ \hline
$2048$ & $4.2\times 10^{1}$ & $1.4\times 10^{2}$ & $5.7\times 10^{2}$ & $1.0\times 10^{2}$ \\ \hline
$4096$ & $6.0\times 10^{1}$ & $2.6\times 10^{2}$ & $3.5\times 10^{3}$ & $4.2\times 10^{2}$ \\ \hline
$8192$ & $9.5\times 10^{1}$ & $3.0\times 10^{2}$ & $1.5\times 10^{3}$ & $2.5\times 10^{2}$ \\ \hline
$16384$ & $1.2\times 10^{2}$ & $4.2\times 10^{2}$ & $3.6\times 10^{3}$ & $4.5\times 10^{2}$ \\ \hline
$32768$ & $2.3\times 10^{2}$ & $7.5\times 10^{2}$ & $5.6\times 10^{3}$ & $7.1\times 10^{2}$ \\ \hline
$65536$ & $2.4\times 10^{2}$ & $1.0\times 10^{3}$ & $1.2\times 10^{4}$ & $1.3\times 10^{3}$ \\ \hline
$131072$ & $3.9\times 10^{2}$ & $1.4\times 10^{3}$ & $5.5\times 10^{3}$ & $9.0\times 10^{2}$ \\ \hline
$262144$ & $6.3\times 10^{2}$ & $3.7\times 10^{3}$ & $1.1\times 10^{5}$ & $1.1\times 10^{4}$ \\ \hline
$524288$ & $8.0\times 10^{2}$ & $3.2\times 10^{3}$ & $3.1\times 10^{4}$ & $3.7\times 10^{3}$ \\ \hline
$1048576$ & $1.2\times 10^{3}$ & $4.8\times 10^{3}$ & $3.1\times 10^{4}$ & $5.1\times 10^{3}$ \\ \hline   
\end{tabular}
\end{center}
\end{table}


\begin{table}[!ht]
\caption{Preconditioning tests}
\label{tabprec}
\begin{center}
\begin{tabular}{|c|c|c|c|c|c|}
\hline Type&$r$&Cond $(C)$\\
\hline 1n &1&3.21E+2\\
\hline 1n &2&4.52E+3\\
\hline 1n &4&2.09E+5\\
\hline 1n &8&6.40E+2\\
\hline 1s &1&5.86E+2x\\
\hline 1s &2&1.06E+4\\
\hline 1s &4&1.72E+3\\
\hline 1s &8&5.60E+3\\
\hline 2n &1&8.05E+1\\
\hline 2n &2&6.82E+3\\
\hline 2n &4&2.78E+4\\
\hline 2n &8&3.59E+3\\
\hline 2s &1&1.19E+3\\
\hline 2s &2&1.96E+3\\
\hline 2s &4&1.09E+4\\
\hline 2s &8&9.71E+3\\
\hline 3n &1&2.02E+4\\
\hline 3n &2&1.53E+3\\
\hline 3n &4&6.06E+2\\
\hline 3n &8&5.67E+2\\
\hline 3s &1&2.39E+4\\
\hline 3s &4&1.69E+3\\
\hline 3s &8&6.74E+3\\
\hline 4n &1&4.93E+2\\
\hline 4n &2&4.48E+2\\
\hline 4n &4&2.65E+2\\
\hline 4n &8&1.64E+2\\ 
\hline 4s &1&1.45E+3\\ 
\hline 4s &2&5.11E+2\\ 
\hline 4s &4&7.21E+2\\ 
\hline 4s &8&2.99E+2\\ 
\hline
\end{tabular}
\end{center}
\end{table}



\begin{table}[ht]
  \caption{Relative residual norms:
 randomized circulant GENP for  
well conditioned linear systems with
ill conditioned leading blocks
  (cf. \cite[Table 2]{PQZa})}
  \label{tab44}
  \begin{center}
    \begin{tabular}{| c | c | c | c | c | c | c |}
      \hline
      \bf{dimension} & \bf{iterations} & \bf{min} & \bf{max} & \bf{mean} & \bf{std} \\ \hline
 $64$ & $0$ & $4.7\times 10^{-14}$ & $8.0\times 10^{-11}$ & $4.0\times 10^{-12}$ & $1.1\times 10^{-11}$ \\ \hline
 $64$ & $1$ & $1.9\times 10^{-15}$ & $5.3\times 10^{-13}$ & $2.3\times 10^{-14}$ & $5.4\times 10^{-14}$ \\ \hline
 $256$ & $0$ & $1.7\times 10^{-12}$ & $1.4\times 10^{-7}$ & $2.0\times 10^{-9}$ & $1.5\times 10^{-8}$ \\ \hline
 $256$ & $1$ & $8.3\times 10^{-15}$ & $4.3\times 10^{-10}$ & $4.5\times 10^{-12}$ & $4.3\times 10^{-11}$ \\ \hline
 $1024$ & $0$ & $1.7\times 10^{-10}$ & $4.4\times 10^{-9}$ & $1.4\times 10^{-9}$ & $2.1\times 10^{-9}$ \\ \hline
 $1024$ & $1$ & $3.4\times 10^{-14}$ & $9.9\times 10^{-14}$ & $6.8\times 10^{-14}$ & $2.7\times 10^{-14}$ \\ \hline
    \end{tabular}
  \end{center}
\end{table}


\begin{table}[h] 
  \caption{Tails of the SVDs and 
lower-rank approximations  (cf. \cite{PQ10})}
\label{tabSVD_TAIL}
  \begin{center}
    \begin{tabular}{| c | c | c | c | c | c |c|}
      \hline
$r$  & $\kappa(C)$  or ${\rm rrn}_i$ & n & \bf{min} & \bf{max} & \bf{mean} & \bf{std}\\ \hline

1 & $\kappa(C)$ & 64 & $2.38\times 10^{+02}$ & $1.10\times 10^{+05}$ & $6.25\times 10^{+03}$ & $1.68\times 10^{+04}$ \\ \hline
1 &  $\kappa(C)$ & 128 & $8.61\times 10^{+02}$ & $7.48\times 10^{+06}$ & $1.32\times 10^{+05}$ & $7.98\times 10^{+05}$ \\ \hline
1 &  $\kappa(C)$ & 256 & $9.70\times 10^{+02}$ & $3.21\times 10^{+07}$ & $3.58\times 10^{+05}$ & $3.21\times 10^{+06}$ \\ \hline
1 & ${\rm rn}_1$ & 64 & $4.01\times 10^{-10}$ & $1.50\times 10^{-07}$ & $5.30\times 10^{-09}$ & $1.59\times 10^{-08}$ \\ \hline
1 & ${\rm rn}_1$ & 128 & $7.71\times 10^{-10}$ & $5.73\times 10^{-07}$ & $1.58\times 10^{-08}$ & $6.18\times 10^{-08}$ \\ \hline
1 & ${\rm rn}_1$ & 256 & $7.57\times 10^{-10}$ & $3.2\times 10^{-07}$ & $1.69\times 10^{-08}$ & $5.02\times 10^{-08}$ \\ \hline
1 & ${\rm rn}_2$ & 64 & $4.01\times 10^{-10}$ & $1.50\times 10^{-07}$ & $5.30\times 10^{-09}$ & $1.59\times 10^{-08}$ \\ \hline
1 & ${\rm rn}_2$ & 128 & $7.71\times 10^{-10}$ & $5.73\times 10^{-07}$ & $1.58\times 10^{-08}$ & $6.18\times 10^{-08}$ \\ \hline
1 & ${\rm rn}_2$ & 256 & $7.57\times 10^{-10}$ & $3.22\times 10^{-07}$ & $1.69\times 10^{-08}$ & $5.02\times 10^{-08}$ \\ \hline
8 &  $\kappa(C)$ & 64 & $1.26\times 10^{+03}$ & $1.61\times 10^{+07}$ & $2.68\times 10^{+05}$ & $1.71\times 10^{+06}$ \\ \hline
8 &  $\kappa(C)$ & 128 & $2.92\times 10^{+03}$ & $3.42\times 10^{+06}$ & $1.58\times 10^{+05}$ & $4.12\times 10^{+05}$ \\ \hline
8 & $\kappa(C)$ & 256 & $1.39\times 10^{+04}$ & $8.75\times 10^{+07}$ & $1.12\times 10^{+06}$ & $8.74\times 10^{+06}$ \\ \hline
8 & ${\rm rn}_1$ & 64 & $3.39\times 10^{-10}$ & $2.27\times 10^{-06}$ & $2.74\times 10^{-08}$ & $2.27\times 10^{-07}$ \\ \hline
8 & ${\rm rn}_1$ & 128 & $4.53\times 10^{-10}$ & $1.91\times 10^{-07}$ & $1.03\times 10^{-08}$ & $2.79\times 10^{-08}$ \\ \hline
8 & ${\rm rn}_1$ & 256 & $8.74\times 10^{-10}$ & $1.73\times 10^{-07}$ & $7.86\times 10^{-09}$ & $1.90\times 10^{-08}$ \\ \hline
8 & ${\rm rn}_2$ & 64 & $1.54\times 10^{-09}$ & $7.59\times 10^{-06}$ & $8.87\times 10^{-08}$ & $7.58\times 10^{-07}$ \\ \hline
8 & ${\rm rn}_2$ & 128 & $1.82\times 10^{-09}$ & $7.27\times 10^{-07}$ & $2.95\times 10^{-08}$ & $8.57\times 10^{-08}$ \\ \hline
8 & ${\rm rn}_2$ & 256 & $2.62\times 10^{-09}$ & $3.89\times 10^{-07}$ & $2.27\times 10^{-08}$ & $5.01\times 10^{-08}$ \\ \hline
32 & $\kappa(C)$ & 64 & $1.77\times 10^{+03}$ & $9.68\times 10^{+06}$ & $1.58\times 10^{+05}$ & $9.70\times 10^{+05}$ \\ \hline
32 & $\kappa(C)$ & 128 & $1.65\times 10^{+04}$ & $6.12\times 10^{+07}$ & $1.02\times 10^{+06}$ & $6.19\times 10^{+06}$ \\ \hline
32 & $\kappa(C)$ & 256 & $3.57\times 10^{+04}$ & $2.98\times 10^{+08}$ & $4.12\times 10^{+06}$ & $2.98\times 10^{+07}$ \\ \hline
32 & ${\rm rn}_1$ & 64 & $2.73\times 10^{-10}$ & $3.29\times 10^{-08}$ & $2.95\times 10^{-09}$ & $4.93\times 10^{-09}$ \\ \hline
32 & ${\rm rn}_1$ & 128 & $3.94\times 10^{-10}$ & $1.29\times 10^{-07}$ & $7.18\times 10^{-09}$ & $1.64\times 10^{-08}$ \\ \hline
32 & ${\rm rn}_1$ & 256 & $6.80\times 10^{-10}$ & $4.00\times 10^{-07}$ & $1.16\times 10^{-08}$ & $4.27\times 10^{-08}$ \\ \hline
32 & ${\rm rn}_2$ & 64 & $2.10\times 10^{-09}$ & $1.49\times 10^{-07}$ & $1.55\times 10^{-08}$ & $2.18\times 10^{-08}$ \\ \hline
32 & ${\rm rn}_2$ & 128 & $2.79\times 10^{-09}$ & $3.80\times 10^{-07}$ & $3.81\times 10^{-08}$ & $6.57\times 10^{-08}$ \\ \hline
32 & ${\rm rn}_2$ & 256 & $5.35\times 10^{-09}$ & $1.05\times 10^{-06}$ & $5.70\times 10^{-08}$ & $1.35\times 10^{-07}$ \\ \hline
    \end{tabular}
  \end{center}
\end{table}


\begin{table}[h] 
  \caption{Heads of SVDs and 
low-rank approximations 
 via dual additive preprocessing}
\label{tabSVD_HEAD}
  \begin{center}
    \begin{tabular}{| c | c | c | c | c | c |c|}
      \hline
$q$  & $\kappa(C_-)$  or ${\rm rrn}_i$ & n & \bf{min} & \bf{max} & \bf{mean} & \bf{std}\\ \hline

1 & $\kappa(C_-)$ & 64 & $1.83\times 10^{+02}$ & $1.26\times 10^{+06}$ & $1.74\times 10^{+04}$ & $1.27\times 10^{+05}$ \\ \hline
1 & $\kappa(C_-)$ & 128 & $6.75\times 10^{+02}$ & $8.76\times 10^{+05}$ & $2.35\times 10^{+04}$ & $9.10\times 10^{+04}$ \\ \hline
1 & $\kappa(C_-)$ & 256 & $4.19\times 10^{+03}$ & $5.82\times 10^{+05}$ & $4.43\times 10^{+04}$ & $8.98\times 10^{+04}$ \\ \hline
1 & ${\rm rn}^{(1)}$ & 64 & $2.43\times 10^{-10}$ & $3.86\times 10^{-08}$ & $2.55\times 10^{-09}$ & $5.43\times 10^{-09}$ \\ \hline
1 & ${\rm rn}^{(1)}$ & 128 & $4.36\times 10^{-10}$ & $1.15\times 10^{-07}$ & $4.45\times 10^{-09}$ & $1.24\times 10^{-08}$ \\ \hline
1 & ${\rm rn}^{(1)}$ & 256 & $6.40\times 10^{-10}$ & $3.17\times 10^{-08}$ & $4.00\times 10^{-09}$ & $5.16\times 10^{-09}$ \\ \hline
1 & ${\rm rn}^{(2)}$ & 64 & $8.30\times 10^{-10}$ & $3.86\times 10^{-08}$ & $2.81\times 10^{-09}$ & $5.35\times 10^{-09}$ \\ \hline
1 & ${\rm rn}^{(2)}$ & 128 & $1.21\times 10^{-9}$ & $1.15\times 10^{-07}$ & $4.80\times 10^{-09}$ & $1.23\times 10^{-08}$ \\ \hline
1 & ${\rm rn}^{(2)}$ & 256 & $1.72\times 10^{-9}$ & $3.18\times 10^{-08}$ & $4.53\times 10^{-09}$ & $4.97\times 10^{-09}$ \\ \hline
8 &  $\kappa(C_-)$ & 64 & $1.37\times 10^{+03}$ & $1.87\times 10^{+06}$ & $7.57\times 10^{+04}$ & $2.16\times 10^{+05}$ \\ \hline
8 &  $\kappa(C_-)$ & 128 & $3.80\times 10^{+03}$ & $8.64\times 10^{+06}$ & $2.00\times 10^{+05}$ & $8.73\times 10^{+05}$ \\ \hline
8 & $\kappa(C_-)$ & 256 & $2.57\times 10^{+04}$ & $1.54\times 10^{+07}$ & $7.25\times 10^{+05}$ & $2.03\times 10^{+06}$ \\ \hline
8 & ${\rm rn}^{(1)}$ & 64 & $1.87\times 10^{-9}$ & $4.48\times 10^{-07}$ & $2.29\times 10^{-08}$ & $5.20\times 10^{-08}$ \\ \hline
8 & ${\rm rn}^{(1)}$ & 128 & $3.04\times 10^{-09}$ & $3.73\times 10^{-07}$ & $2.72\times 10^{-08}$ & $5.83\times 10^{-08}$ \\ \hline
8 & ${\rm rn}^{(1)}$ & 256 & $3.78\times 10^{-09}$ & $2.01\times 10^{-06}$ & $4.81\times 10^{-08}$ & $2.02\times 10^{-07}$ \\ \hline
8 & ${\rm rn}^{(2)}$ & 64 & $1.30\times 10^{-09}$ & $2.47\times 10^{-07}$ & $1.09\times 10^{-08}$ & $2.70\times 10^{-08}$ \\ \hline
8 & ${\rm rn}^{(2)}$ & 128 & $1.85\times 10^{-09}$ & $1.50\times 10^{-07}$ & $1.36\times 10^{-08}$ & $2.75\times 10^{-08}$ \\ \hline
8 & ${\rm rn}^{(2)}$ & 256 & $2.19\times 10^{-09}$ & $1.10\times 10^{-06}$ & $2.36\times 10^{-08}$ & $1.10\times 10^{-07}$ \\ \hline
32 & $\kappa(C_-)$ & 64 & $3.75\times 10^{+03}$ & $3.25\times 10^{+07}$ & $6.01\times 10^{+05}$ & $3.28\times 10^{+06}$ \\ \hline
32 & $\kappa(C_-)$ & 128 & $2.41\times 10^{+04}$ & $1.09\times 10^{+08}$ & $1.95\times 10^{+06}$ & $1.10\times 10^{+07}$ \\ \hline
32 & $\kappa(C_-)$ & 256 & $1.33\times 10^{+05}$ & $2.11\times 10^{+10}$ & $2.18\times 10^{+08}$ & $2.11\times 10^{+09}$ \\ \hline
32 & ${\rm rn}^{(1)}$ & 64 & $7.78\times 10^{-09}$ & $1.39\times 10^{-06}$ & $8.17\times 10^{-08}$ & $1.94\times 10^{-07}$ \\ \hline
32 & ${\rm rn}^{(1)}$ & 128 & $9.81\times 10^{-09}$ & $2.35\times 10^{-06}$ & $1.17\times 10^{-07}$ & $3.05\times 10^{-07}$ \\ \hline
32 & ${\rm rn}^{(1)}$ & 256 & $2.05\times 10^{-08}$ & $3.99\times 10^{-06}$ & $1.91\times 10^{-07}$ & $5.06\times 10^{-07}$ \\ \hline
32 & ${\rm rn}^{(2)}$ & 64 & $1.84\times 10^{-09}$ & $2.62\times 10^{-07}$ & $1.85\times 10^{-08}$ & $4.09\times 10^{-08}$ \\ \hline
32 & ${\rm rn}^{(2)}$ & 128 & $2.47\times 10^{-09}$ & $6.77\times 10^{-07}$ & $2.93\times 10^{-08}$ & $8.38\times 10^{-08}$ \\ \hline
32 & ${\rm rn}^{(2)}$ & 256 & $5.05\times 10^{-09}$ & $8.85\times 10^{-07}$ & $4.38\times 10^{-08}$ & $1.14\times 10^{-07}$ \\ \hline
    \end{tabular}
  \end{center}
\end{table}


\begin{table}[h] 
  \caption{Heads of SVDs and 
low-rank approximation with random multipliers}
\label{tabSVD_HEAD1}
  \begin{center}
    \begin{tabular}{| c | c | c | c | c | c |c|}
      \hline
$q$  & ${\rm rrn}_i$ & n & \bf{min} & \bf{max} & \bf{mean} & \bf{std}\\ \hline

1 & ${\rm rn}^{(1)}$ & 64 &  $ 2.35\times 10^{-10} $  &  $ 1.32\times 10^{-07} $  &  $ 3.58\times 10^{-09} $  &  $ 1.37\times 10^{-08} $  \\ \hline		
1 & ${\rm rn}^{(1)}$ & 128 &  $ 4.41\times 10^{-10} $  &  $ 3.28\times 10^{-08} $  &  $ 3.55\times 10^{-09} $  &  $ 5.71\times 10^{-09} $  \\ \hline		
1 & ${\rm rn}^{(1)}$ & 256 &  $ 6.98\times 10^{-10} $  &  $ 5.57\times 10^{-08} $  &  $ 5.47\times 10^{-09} $  &  $ 8.63\times 10^{-09} $  \\ \hline		
1 & ${\rm rn}^{(2)}$ & 64 &  $ 8.28\times 10^{-10} $  &  $ 1.32\times 10^{-07} $  &  $ 3.86\times 10^{-09} $  &  $ 1.36\times 10^{-08} $  \\ \hline		
1 & ${\rm rn}^{(2)}$ & 128 &  $ 1.21\times 10^{-09} $  &  $ 3.28\times 10^{-08} $  &  $ 3.91\times 10^{-09} $  &  $ 5.57\times 10^{-09} $  \\ \hline		
1 & ${\rm rn}^{(2)}$ & 256 &  $ 1.74\times 10^{-09} $  &  $ 5.58\times 10^{-08} $  &  $ 5.96\times 10^{-09} $  &  $ 8.47\times 10^{-09} $  \\ \hline		
8 & ${\rm rn}^{(1)}$ & 128 &  $ 2.56\times 10^{-09} $  &  $ 1.16\times 10^{-06} $  &  $ 4.30\times 10^{-08} $  &  $ 1.45\times 10^{-07} $  \\ \hline		
8 & ${\rm rn}^{(1)}$ & 256 &  $ 4.45\times 10^{-09} $  &  $ 3.32\times 10^{-07} $  &  $ 3.40\times 10^{-08} $  &  $ 5.11\times 10^{-08} $  \\ \hline		
8 & ${\rm rn}^{(2)}$ & 64 &  $ 1.46\times 10^{-09} $  &  $ 9.56\times 10^{-08} $  &  $ 5.77\times 10^{-09} $  &  $ 1.06\times 10^{-08} $  \\ \hline		
8 & ${\rm rn}^{(2)}$ & 128 &  $ 1.64\times 10^{-09} $  &  $ 4.32\times 10^{-07} $  &  $ 1.86\times 10^{-08} $  &  $ 5.97\times 10^{-08} $  \\ \hline		
8 & ${\rm rn}^{(2)}$ & 256 &  $ 2.50\times 10^{-09} $  &  $ 1.56\times 10^{-07} $  &  $ 1.59\times 10^{-08} $  &  $ 2.47\times 10^{-08} $  \\ \hline		
32 & ${\rm rn}^{(1)}$ & 64 &  $ 6.80\times 10^{-09} $  &  $ 2.83\times 10^{-06} $  &  $ 1.01\times 10^{-07} $  &  $ 3.73\times 10^{-07} $  \\ \hline		
32 & ${\rm rn}^{(1)}$ & 128 &  $ 1.25\times 10^{-08} $  &  $ 6.77\times 10^{-06} $  &  $ 1.28\times 10^{-07} $  &  $ 6.76\times 10^{-07} $  \\ \hline		
32 & ${\rm rn}^{(1)}$ & 256 &  $ 1.85\times 10^{-08} $  &  $ 1.12\times 10^{-06} $  &  $ 1.02\times 10^{-07} $  &  $ 1.54\times 10^{-07} $  \\ \hline		
32 & ${\rm rn}^{(2)}$ & 64 &  $ 1.84\times 10^{-09} $  &  $ 6.50\times 10^{-07} $  &  $ 2.30\times 10^{-08} $  &  $ 8.28\times 10^{-08} $  \\ \hline		
32 & ${\rm rn}^{(2)}$ & 128 &  $ 3.11\times 10^{-09} $  &  $ 1.45\times 10^{-06} $  &  $ 2.87\times 10^{-08} $  &  $ 1.45\times 10^{-07} $  \\ \hline		
32 & ${\rm rn}^{(2)}$ & 256 &  $ 4.39\times 10^{-09} $  &  $ 2.16\times 10^{-07} $  &  $ 2.37\times 10^{-08} $  &  $ 3.34\times 10^{-08} $  \\ \hline		

    \end{tabular}
  \end{center}
\end{table}


\begin{table}[h] 
  \caption{Heads of SVDs and 
low-rank approximations with random Toeplitz multipliers}
\label{tabSVD_HEAD1T}
  \begin{center}
    \begin{tabular}{| c | c | c | c | c | c |c|}
      \hline
$q$  & ${\rm rrn}^{(i)}$ & n & \bf{min} & \bf{max} & \bf{mean} & \bf{std}\\ \hline

8 & ${\rm rrn}^{(1)}$ & 64 &  $ 2.22\times 10^{-09} $  &  $ 7.89\times 10^{-06} $  &  $ 1.43\times 10^{-07} $  &  $ 9.17\times 10^{-07} $  \\ \hline		
8 & ${\rm rrn}^{(1)}$ & 128 &  $ 3.79\times 10^{-09} $  &  $ 4.39\times 10^{-05} $  &  $ 4.87\times 10^{-07} $  &  $ 4.39\times 10^{-06} $  \\ \hline				
8 & ${\rm rrn}^{(1)}$ & 256 &  $ 5.33\times 10^{-09} $  &  $ 3.06\times 10^{-06} $  &  $ 6.65\times 10^{-08} $  &  $ 3.12\times 10^{-07} $  \\ \hline			
8 & ${\rm rrn}^{(2)}$ & 64 &  $ 1.13\times 10^{-09} $  &  $ 3.66\times 10^{-06} $  &  $ 6.37\times 10^{-08} $  &  $ 4.11\times 10^{-07} $  \\ \hline				
8 & ${\rm rrn}^{(2)}$ & 128 &  $ 1.81\times 10^{-09} $  &  $ 1.67\times 10^{-05} $  &  $ 1.90\times 10^{-07} $  &  $ 1.67\times 10^{-06} $  \\ \hline			
8 & ${\rm rrn}^{(2)}$ & 256 &  $ 2.96\times 10^{-09} $  &  $ 1.25\times 10^{-06} $  &  $ 2.92\times 10^{-08} $  &  $ 1.28\times 10^{-07} $  \\ \hline	
32 & ${\rm rrn}^{(1)}$ & 64 &  $ 6.22\times 10^{-09} $  &  $ 5.00\times 10^{-07} $  &  $ 4.06\times 10^{-08} $  &  $ 6.04\times 10^{-08} $  \\ \hline
32 & ${\rm rrn}^{(1)}$ & 128 &  $ 2.73\times 10^{-08} $  &  $ 4.88\times 10^{-06} $  &  $ 2.57\times 10^{-07} $  &  $ 8.16\times 10^{-07} $  \\ \hline
32 & ${\rm rrn}^{(1)}$ & 256 &  $ 1.78\times 10^{-08} $  &  $ 1.25\times 10^{-06} $  &  $ 1.18\times 10^{-07} $  &  $ 2.03\times 10^{-07} $  \\ \hline	
32 & ${\rm rrn}^{(2)}$ & 64 &  $ 1.64\times 10^{-09} $  &  $ 1.26\times 10^{-07} $  &  $ 9.66\times 10^{-09} $  &  $ 1.48\times 10^{-08} $  \\ \hline
32 & ${\rm rrn}^{(2)}$ & 128 &  $ 5.71\times 10^{-09} $  &  $ 9.90\times 10^{-07} $  &  $ 5.50\times 10^{-08} $  &  $ 1.68\times 10^{-07} $  \\ \hline	
32 & ${\rm rrn}^{(2)}$ & 256 &  $ 4.02\times 10^{-09} $  &  $ 2.85\times 10^{-07} $  &  $ 2.74\times 10^{-08} $  &  $ 4.48\times 10^{-08} $  \\ \hline	
    \end{tabular}
  \end{center}
\end{table}


\begin{table}[h]
\caption{Relative residual norms: ill conditioned linear systems via nmb
approximation and block triangulation}
\label{tablsnmb}
\begin{center}
\begin{tabular}{|c|c|c|c|c|c|}
\hline
\textbf{$n$}&\textbf{r}&\textbf{min}&\textbf{max}&\textbf{mean}&\textbf{std}\\\hline
$32$ & $1$ & $1.49\times 10^{-13}$ & $1.36\times 10^{-9}$ & $4.25\times 10^{-11}$ & $1.56\times 10^{-10}$ \\ \hline
$32$ & $2$ & $3.70\times 10^{-13}$ & $2.13\times 10^{-8}$ & $3.83\times 10^{-10}$ & $2.35\times 10^{-9}$ \\ \hline
$32$ & $4$ & $9.33\times 10^{-13}$ & $1.08\times 10^{-8}$ & $3.37\times 10^{-10}$ & $1.26\times 10^{-9}$ \\ \hline

$64$ & $1$ & $1.11\times 10^{-12}$ & $6.87\times 10^{-9}$ & $2.03\times 10^{-10}$ & $7.49\times 10^{-10}$ \\ \hline
$64$ & $2$ & $1.53\times 10^{-12}$ & $1.21\times 10^{-8}$ & $5.86\times 10^{-10}$ & $1.77\times 10^{-9}$ \\ \hline
$64$ & $4$ & $2.21\times 10^{-12}$ & $1.27\times 10^{-7}$ & $1.69\times 10^{-9}$ & $1.28\times 10^{-8}$ \\ \hline

\end{tabular}
\end{center}
\end{table}

\begin{table}[h]
\caption{Relative residual norms: ill conditioned linear systems with MLDIVIDE(A,B)
}
\label{tablsge}
\begin{center}
\begin{tabular}{|c|c|c|c|c|c|}
\hline
\textbf{$n$}&\textbf{r}&\textbf{min}&\textbf{max}&\textbf{mean}&\textbf{std}\\\hline
$32$ & $1$ & $6.34\times 10^{-3}$ & $7.44\times 10^{1}$ & $1.74\times 10^{0}$ & $7.53\times 10^{0}$ \\ \hline
$32$ & $2$ & $2.03\times 10^{-2}$ & $1.32\times 10^{1}$ & $9.19\times 10^{-1}$ & $1.62\times 10^{0}$ \\ \hline
$32$ & $4$ & $4.57\times 10^{-2}$ & $1.36\times 10^{1}$ & $1.14\times 10^{0}$ & $1.93\times 10^{0}$ \\ \hline

$64$ & $1$ & $3.82\times 10^{-3}$ & $9.93\times 10^{0}$ & $1.03\times 10^{0}$ & $1.66\times 10^{0}$ \\ \hline
$64$ & $2$ & $1.96\times 10^{-2}$ & $1.27\times 10^{2}$ & $3.09\times 10^{0}$ & $1.40\times 10^{1}$ \\ \hline
$64$ & $4$ & $7.13\times 10^{-3}$ & $6.63\times 10^{0}$ & $8.23\times 10^{-1}$ & $1.20\times 10^{0}$ \\ \hline

\end{tabular}
\end{center}
\end{table}


\begin{table}[h] 
  \caption{Relative residual norms: ill conditioned linear systems  
with dual additive preprocessing  and block triangulation}
\label{tablsnmbd}
  \begin{center}
    \begin{tabular}{| c | c | c | c | c | c |c|}
      \hline
$n$  & q & \bf{min} & \bf{max} & \bf{mean} & \bf{std}\\ \hline

32 & 1 &  $ 2.33\times 10^{-14} $  &  $ 2.28\times 10^{-06} $  &  $ 2.31\times 10^{-08} $  &  $ 2.28\times 10^{-07} $  \\ \hline
32  & 2 &  $ 3.40\times 10^{-13} $  &  $ 4.93\times 10^{-08} $  &  $ 9.11\times 10^{-10} $  &  $ 5.71\times 10^{-09} $  \\ \hline
32 & 4 &  $ 5.97\times 10^{-13} $  &  $ 1.63\times 10^{-07} $  &  $ 2.22\times 10^{-09} $  &  $ 1.64\times 10^{-08} $  \\ \hline
64  & 1 &  $ 3.90\times 10^{-14} $  &  $ 2.78\times 10^{-05} $  &  $ 2.81\times 10^{-07} $  &  $ 2.78\times 10^{-06} $  \\ \hline
64  & 2 &  $ 3.53\times 10^{-13} $  &  $ 3.76\times 10^{-08} $  &  $ 1.13\times 10^{-09} $  &  $ 4.72\times 10^{-09} $  \\ \hline
64  & 4 &  $ 3.54\times 10^{-12} $  &  $ 2.53\times 10^{-07} $  &  $ 5.19\times 10^{-09} $  &  $ 2.83\times 10^{-08} $  \\ \hline

    \end{tabular}
  \end{center}
\end{table}


\begin{table}[h] 
  \caption{Relative residual norms: ill conditioned linear system  
with random  multipliers and block triangulation}
\label{tablsmrm}
  \begin{center}
    \begin{tabular}{| c | c | c | c | c | c |c|}
      \hline
$n$ & q & \bf{min} & \bf{max} & \bf{mean} & \bf{std}\\ \hline

32  & 1 &  $ 7.08\times 10^{-30} $  &  $ 4.00\times 10^{-23} $  &  $ 4.52\times 10^{-25} $  &  $ 4.01\times 10^{-24} $  \\ \hline
32  & 2 &  $ 7.49\times 10^{-30} $  &  $ 2.29\times 10^{-21} $  &  $ 2.77\times 10^{-23} $  &  $ 2.33\times 10^{-22} $  \\ \hline
32  & 4 &  $ 1.46\times 10^{-28} $  &  $ 1.63\times 10^{-07} $  &  $ 4.83\times 10^{-25} $  &  $ 2.73\times 10^{-24} $  \\ \hline
64  & 1 &  $ 1.13\times 10^{-29} $  &  $ 1.01\times 10^{-24} $  &  $ 2.31\times 10^{-26} $  &  $ 1.11\times 10^{-25} $  \\ \hline
64  & 2 &  $ 6.60\times 10^{-29} $  &  $ 6.90\times 10^{-24} $  &  $ 1.45\times 10^{-25} $  &  $ 7.73\times 10^{-25} $  \\ \hline
64  & 4 &  $ 2.60\times 10^{-28} $  &  $ 1.41\times 10^{-21} $  &  $ 1.61\times 10^{-23} $  &  $ 1.42\times 10^{-22} $  \\ \hline

    \end{tabular}
  \end{center}
\end{table}


\begin{table}[h] 
  \caption{Relative residual norms: ill conditioned linear systems 
with the SMW formula}
\label{tablsmwsmnl}
\begin{center}
\begin{tabular}{|c|c|c|c|c|c|}
\hline
\textbf{$n$}&\textbf{r}&\textbf{min}&\textbf{max}&\textbf{mean}&\textbf{std}\\\hline
$64$	& $1$	 & $1.18\times 10^{-15}$ & $6.30\times 10^{-13}$ & $2.37\times 10^{-14}$ & $7.45\times 10^{-14}$ \\ \hline
$64$	& $2$	 & $3.42\times 10^{-15}$ & $1.94\times 10^{-10}$ & $2.15\times 10^{-12}$ & $1.94\times 10^{-11}$ \\ \hline
$64$	& $4$	 & $6.66\times 10^{-15}$ & $1.25\times 10^{-10}$ & $1.82\times 10^{-12}$ & $1.25\times 10^{-11}$ \\ \hline
$128$   & $1$	 & $5.79\times 10^{-15}$ & $4.85\times 10^{-12}$ & $1.21\times 10^{-13}$ & $4.96\times 10^{-13}$ \\ \hline
$128$   & $2$	 & $1.45\times 10^{-14}$ & $1.85\times 10^{-11}$ & $5.23\times 10^{-13}$ & $1.88\times 10^{-12}$ \\ \hline
$128$   & $4$	 & $8.41\times 10^{-14}$ & $4.75\times 10^{-11}$ & $2.89\times 10^{-12}$ & $5.95\times 10^{-12}$ \\ \hline
\end{tabular}
\end{center}
\end{table}





\begin{table}[h]
\caption{The CPU time  (in cycles) for solving an ill conditioned 
real symmetric Toeplitz linear system}
\label{tabhank}
  \begin{center}

    \begin{tabular}{| c | c | c | c | c | c |}

\hline

\bf{$n$} & \bf{Alg. 6.1} & \bf{QR}  & \bf{SVD} & \bf{QR/Alg. 6.1} & \bf{SVD/Alg. 6.1} \\ \hline

 $512$  &  $56.3$  &  $148.4$   & $4134.8$ & $2.6$   & $73.5$  \\ \hline
 $1024$ &  $120.6$  &  $1533.5$  & $70293.1$ & $12.7$  & $582.7$ \\ \hline
 $2048$ &  $265.0$  &  $11728.1$ & $-$       & $44.3$ & $-$      \\ \hline
 $4096$ &  $589.4$  &  $-$       & $-$       & $-$     & $-$      \\ \hline
 $8192$ &  $1304.8$ &  $-$       & $-$       & $-$     & $-$      \\ \hline


\end{tabular}

  \end{center}

\end{table}



\section{Related work, our progress, and further study}\label{srel}


Preconditioned iterative  algorithms  for linear systems of equations 
is a classical subject \cite{A94}, \cite{B02}, \cite{G97}.
The problem of creating inexpensive preconditioners for general use
has been around for a long while as well. 
 On estimating the condition numbers of random matrices 
see  \cite{D88},  \cite{E88},  \cite{ES05},
\cite{CD05}, \cite{SST06}, \cite{B11},
and the bibliography therein. The study in the case of
random structured matrices
was stated as a challenge in  \cite{SST06}.
We provide such estimates for Gaussian random Toeplitz and circulant
matrices in Sections \ref{scgrtm} and \ref{scgrcm}.
They can be surprising because
the paper \cite{BG05} has proved that the condition numbers
 grow exponentially in $n$ as $n\rightarrow \infty$
for some large and important classes
of  $n\times n$ 
Toeplitz matrices, 
whereas we prove the opposit for 
Gaussian random Toeplitz $n\times n$ matrices.
 
Our present study of randomized preconditioning 
extends  
 \cite{PGMQ}, \cite{PIMR10}, \cite{PQZa},  and \cite{PQZC}. 
In Sections \ref{sapsr}--\ref{sexp} we
outline and test some new applications of 
randomized preconditioning, whereas
our Theorems  \ref{1} and \ref{thkappa2} and 
and Corollaries 
\ref{cosumm3}
and \ref{coaugkappamn1} 
 support such 
 applications formally. The formal support
relies on using Gaussian random matrices,
but empirically our algorithms remained as 
efficient where instead we employed Gaussian
 random 
Toeplitz matrices  
and quite typically  
where we further decreased the number 
of random parameters involved,
e.g., where we used block vectors 
 $U$ with the blocks $\pm I$ and O in the map $A\Longrightarrow C=A+UU^T$
(see the end of Section \ref{sprecondtests} and Table \ref{tabprec})
and
used circulant or Householder multipliers 
filled with $\pm 1$
(see our Section \ref{sexgeneral}
and Table \ref{tab44} and  \cite[Table 2]{PQZa}), thus 
limiting  randomization to the choice of 
the block sizes in the block vector $U$ and 
of the signs $\pm$
for the blocks $\pm I$ and entries $\pm 1$.

Besides the cited 
estimates for the condition numbers of  
Gaussian random Toeplitz and circulant matrices,
our technical novelties include 
 randomized multiplicative and
additive  preconditioning,
dual additive preprocessing,
augmentation and
 dual augmentation techniques,
 extensions of
 the SMW formula,
 proof of the power of all these techniques,
randomized 
algorithms for numerical rank
and 
 approximation of
leading and trailing singular spaces
avoiding orthogonalization and
pivoting
and their
 applications to low-rank approximation 
of a matrix, tensor decomposition, and
$2\times 2$
block factorizations 
of ill conditioned matrices.

Approximation by low-rank matrices
and  the
extensions to tensor decompositions
are thriving research areas,
 with numerous applications 
to matrix and tensor computations. 
One can partly trace its 
previous study 
through the papers
\cite{GTZ97}, \cite{GT01}, \cite{GOS08}, \cite{HMT11},
\cite{KB09},
 \cite{MMD08}, \cite{OT09}, \cite{OT10}, 
\cite{OT11}, 
 \cite{T00},
 and the bibliography therein, but 
much earlier advances in this area appeared in 
the papers
\cite{BCLR79}, \cite{B80},  
 \cite{B85}, \cite{B86}, \cite{BC87}, directed to 
estimating the border rank of matrices and tensors and 
initially motivated by the design of fast 
matrix multiplication algorithms.
Presently, application of tensor 
decompositions to 
the acceleration of matrix computations
 is a fashionable subject with
applications to many important areas of modern computing (see, e.g.,   
\cite{T00}, \cite{MMD08}, \cite{OT09}, \cite{KB09}), but then again its
earliest examples appeared in the cited papers on border rank 
and in \cite{P72}. The latter paper
has introduced  the technique of trilinear aggregation,
 a basic ingredient of all subsequent fast algorithms  
for matrix multiplication with the inputs of both immense
sizes (far beyond any practical interest)
\cite{P79}, \cite{P81},
\cite{P84}, \cite{CW90}, \cite{S10}, \cite{VW12}
and realistic moderate sizes \cite{P81},
\cite{P84}, \cite{LPS92}, \cite{K04}.  
Historically this work was the first example of the 
acceleration of fundamental matrix computations
by means of tensor decomposition.
More special direction of Tensor Train (TT) decomposition was proposed 
in \cite{O09} and further developed in \cite{OT09}, \cite{OT10}, \cite{O11}.
It is closely related to the DMRG quantum computations
\cite{V03}, \cite{W93} and  Hierarchical Tucker (HT)
tensor decomposition \cite{HK09}, \cite{H12}.

We hope that our present paper will motivate
 further study
of randomized 
 aug\-men\-ta\-tion and dual
 aug\-men\-ta\-tion,
their link to aggregation processes
(cf. \cite{MP80}, \cite{PQa}),
specification of our methods  to
ill conditioned  matrices
that 
have displacement or rank structures
(cf. \cite{KKM79}, \cite{P90}, \cite{GKO95}, \cite{p01}, \cite{VBHK01},
\cite{EG99}, \cite{VVM07}, \cite{VVM08}
on the displacement and rank structured matrices),
and
comparison and combination of our techniques with
each other  and various  known methods, such as
the techniques of \cite{R90} (cf. our Remark \ref{rezmr}),
the homotopy continuation methods and Newton's structured iteration
(cf. Section \ref{snewt}, \cite{PKRK}, \cite[Chapter 6]{p01}, and \cite{P10}). 
Formal 
and experimental 
support of weakly 
randomized preconditioning (using fewer random parameters) 
remains an important research challenge.
Natural extensions of our study should also include
 lower estimates for the smallest singular value
of a random $m\times n$ matrix where $m\gg n$ or $n\gg m$,
further links between 
augmentation and additive preprocessing,
and the refinement, development and testing
of our algorithms.



\clearpage


{\bf {\LARGE {Appendix}}}
\appendix 


\section{Operations with structured matrices in terms of their displacements}\label{apstr}


The following simple theorem can be found in
 \cite{P00} or \cite[Section 1.5]{p01}.

\begin{theorem}\label{thdgs}
Assume five matrices $A$, $B$, $C$, $M$ and $N$
and a pair of scalars $\alpha $ and $\beta $.
Then as long as the matrix sizes are compatible
we have
\begin{equation}\label{eqlcm}
A(\alpha M+dN)-(\alpha M+\beta N)B=\alpha (AM-MB)+\beta (AN-NB),
\end{equation}
\begin{equation}\label{eqtrm}
A^TM^T-B^TM^T=-(BM-MA)^T,
\end{equation}
\begin{equation}\label{eqprm}
A(MN)-(MN)C=(AM-MB)N+M(BN-NC).
\end{equation}
Furthermore for a nonsingular  matrix $M$  we have
\begin{equation}\label{eqinvm}
AM^{-1}-M^{-1}B=-M^{-1}(BM-MA)M^{-1}.
\end{equation}
\end{theorem}

\begin{corollary}\label{codgdr}
Under the assumptions of Theorem \ref{thdgs} we have
\begin{equation}\label{eqlcgg}
G_{A,B}(\alpha M+\beta N)=(\alpha G_{A,B}(M)~|~\beta G_{A,B}(N)),
\end{equation}
\begin{equation}\label{eqlcgh}
H_{A,B}(\alpha M+\beta N)=(\alpha H_{A,B}(M)~|~\beta H_{A,B}(N)),
\end{equation}
\begin{equation}\label{eqtrg}
G_{A,B}(M^T)=-H_{B^T,A^T}(M^T),~H_{A,B}(M^T)=G_{B^T,A^T}(M^T),
\end{equation}
\begin{equation}\label{eqprgg}
G_{A,C}(MN)=(G_{A,B}(M)~|~M G_{B,C}(N)),
\end{equation}
\begin{equation}\label{eqprgh}
H_{A,C}(MN)=(N^TH_{A,B}(M)~|~H_{B,C}(N)),
\end{equation}
\begin{equation}\label{eqinvg}
G_{A,B}(M^{-1})=-M^{-1}G_{B,A}(M),~H_{A,B}(M^{-1})=M^{-T}H_{B,A}(M)
\end{equation}

\noindent and consequently
\begin{equation}\label{eqlcr}
d_{A,B}(\alpha M+\beta N)\le d_{A,B}(M)+d_{A,B}(N), 
\end{equation}
\begin{equation}\label{eqtrr}
d_{A,B}(M^T)=d_{B^T,A^T}(M),
\end{equation}
\begin{equation}\label{eqpr}
d_{A,C}(MN)\le d_{A,B}(M)+d_{B,C}(N),
\end{equation}
\begin{equation}\label{eqinv}
d_{A,B}(M^{-1})=d_{B,A}(M).
\end{equation}
\end{corollary}


\section{Uniform random sampling and nonsingularity of random matrices}\label{srsnrm}


{\em Uniform random sampling} of elements from a finite set $\Delta$ is their selection   
from  this set at random, independently of each other and
under the uniform probability distribution on the set $\Delta$. 


\begin{theorem}\label{thdl} 
Under the assumptions of Lemma \ref{ledl} let the values of the variables 
of the polynomial be randomly and uniformly sampled from a finite set $\Delta$. 
Then the polynomial vanishes with a probability at most $\frac{d}{|\Delta|}$. 
\end{theorem}


\begin{corollary}\label{codlstr} 
Let the entries of a general or Toeplitz  $m\times n$ 
matrix have been randomly and uniformly 
sampled from a finite set $\Delta$ of cardinality $|\Delta|$ (in any fixed ring). 
Let $l=\min\{m,n\}$.
Then (a) every $k\times k$ submatrix $M$ for $k\le l$ is nonsingular with a probability at 
least $1-\frac{k}{|\Delta|}$ and (b) is strongly nonsingular with a probability at least 
$1-\sum_{i=1}^k\frac{i}{|\Delta|}= 1-\frac{(k+1)k}{2|\Delta|}$.
Furthermore (c) if the submatrix $M$ is indeed nonsingular, then any entry of its inverse is 
nonzero with a probability at least $1-\frac{k-1}{|\Delta|}$.
\end{corollary}


\begin{proof}
The claimed properties of nonsingularity and nonvanishing hold for generic matrices. 
The singularity of a $k\times k$ matrix means that its determinant vanishes,
but the determinant is a polynomial of total degree $k$ in the entries. Therefore
Theorem \ref{thdl} implies
parts (a) and consequently (b). Part (c) follows because a fixed entry of the inverse vanishes
if and only if the respective entry of the adjoint vanishes, but up to the sign the latter 
entry is the determinant of a $(k-1)\times (k-1)$ submatrix of the input matrix $M$, and so it is
a polynomial of degree $k-1$ in its entries. 
\end{proof}


\section{Extremal singular values of random complex matrices}\label{scomplin}


We have assumed dealing with real random matrices and vectors throughout the paper, 
but most of our study can be readily extended to the computations in the field 
$\mathbb C$ of complex numbers if we replace the transposes 
$A^T$ by the Hermitian transposes $A^H$. 
All the results of
Section \ref{sngrm} apply to this case
equally well. 
Below we elaborate upon the respective 
 extension of our 
probabilistic bounds on the norms and singular values.

\begin{definition}\label{defcompl} 
The set $\mathcal G_{\mathbb C,\mu,\sigma}^{m\times n}$ 
of complex Gaussian random $m\times n$ matrices with a mean $\mu$ and a variance $\sigma$
is the set $\{A+B\sqrt {-1}\}$  for $(A~|~B) \in \mathcal {G}_{\mu,\sigma}^{m\times 2n}$
(cf. Definition \ref{defrndm}). 
\end{definition}

We can immediately extend Theorem \ref{thsignorm}
to the latter matrices.
Let us extend Theorem \ref{thsiguna}.
Its original proof in \cite{SST06} 
relies on the following result.

\begin{lemma}\label{lebas} 
Suppose $y$ is a positive number; 
${\bf w}\in \mathbb R^{n\times 1}$ is any fixed real unit vector, $||{\bf w}||=1$,
 $A\in \mathcal G_{\mu,\sigma}^{n\times n}$
and therefore is nonsingular with probability $1$.
 Then 
$${\rm Probability}\{||A^{-1}{\bf w}||>1/y\}\le 
\sqrt {\frac{2}{\pi}}\frac{y}{\sigma}~{\rm for}~j=1,\dots,n.$$
\end{lemma}

 The following lemma and corollary extend Lemmas \ref{leinp} and \ref{lebas} to the complex case.


\begin{lemma}\label{lebasic}  
The bound of Lemma \ref{leinp} also holds provided
   ${\bf t}={\bf q}+{\bf r}\sqrt {-1}$  is a fixed complex unit vector
and ${\bf b}={\bf f}+{\bf g}\sqrt {-1}\in \mathcal G_{\mathbb C,\mu,\sigma}^{n\times 1}$ is 
a complex vector such that
${\bf f}$, ${\bf g}$, ${\bf q}$ and ${\bf r}$ are real vectors,
$||{\bf t}||=1$,  and the vectors ${\bf f}$ and ${\bf g}$ are in $\mathcal G_{\mu,\sigma}^{n\times 1}$.
\end{lemma}
\begin{proof}
We have 
${\bf t}^H{\bf b}={\bf q}^T{\bf f} +{\bf r}^T{\bf g}+({\bf q}^T{\bf g} -{\bf r}^T{\bf f})\sqrt {-1}$,
and so 
$|{\bf t}^H{\bf b}|^2=|{\bf q}^T{\bf f} +{\bf r}^T{\bf g}|^2+|{\bf q}^T{\bf g} -{\bf r}^T{\bf f}|^2$.
Hence
$|{\bf t}^H{\bf b}|\ge|{\bf q}^T{\bf f} +{\bf r}^T{\bf g}|=|{\bf u}^T{\bf v}|$
where ${\bf u}^T=({\bf q}^T~|~{\bf r}^T)$ and ${\bf v}=({\bf f}^T~|~{\bf g}^T)^T$.
Note that  ${\bf v}\in \mathcal G_{\mu,\sigma}^{1\times 2n}$ and $||{\bf u}||=||{\bf t}||=1$
and apply  Lemma \ref{leinp} to real vectors ${\bf u}$ 
and ${\bf v}$ replacing ${\bf b}$ and  ${\bf t}$.
\end{proof}
 

\begin{corollary}\label{cobas}  
Suppose $y$ is a positive number and suppose a matrix
 $A\in \mathcal G_{\mathbb C,\mu,\sigma}^{n\times n}$
and therefore is nonsingular with probability $1$.
 Then 
$${\rm Probability}\{||A^{-1}{\bf e}_j||>1/y\}\le 
\sqrt {\frac{2}{\pi}}\frac{y}{\sigma}~{\rm for}~j=1,\dots,n.$$
\end{corollary}
\begin{proof}
In the case of  real matrices $A$ 
the corollary  is supported by the argument in the proof of  \cite[Lemma 3.2]{SST06},
which employs the well known estimate that 
we state as our Lemma \ref{leinp}. Now we employ 
Lemma \ref{lebasic} instead of this estimate,
otherwise
keep the same argument as in \cite{SST06},
and arrive at Corollary \ref{cobas}. 
 \end{proof}


\begin{corollary}\label{cobasic}  
Under the assumptions of Corollary \ref{cobas}
we have $||A^{-1}||\le \sum_{j=1}^n X_j$
where $X_j$ are nonnegative random variables 
 such that
$${\rm Probability}\{X_j>1/y\}\le \sqrt {\frac{2}{\pi}}\frac{y}{\sigma}~{\rm for}~j=1,\dots,n.$$
\end{corollary}
\begin{proof}
Recall that for any $n\times n$ matrix $B$ we have $||B||=||B{\bf w}||$ 
for some unit vector ${\bf w}=\sum_{j=1}^nw_j{\bf e}_j$.
We have $|w_j|\le ||{\bf w}||=1$
for all $j$.
Substitute $B=A^{-1}$ and obtain 
$||A^{-1}||=||A^{-1}{\bf w}||= ||\sum_{j=1}^nA^{-1}w_j{\bf e}_j||\le
\sum_{j=1}^n |w_j|~||A^{-1}{\bf e}_j||$, and so $||A^{-1}||\le
\sum_{j=1}^n X_j$ where $X_j=||A^{-1}{\bf e}_j||$
for all $j$. It remains to combine
this bound with Corollary \ref{cobas}. 
 \end{proof}

The corollary implies that ${\rm Probability}\{||A^{-1}||>1/y\}$ converges to $0$
 proportionally to $y$ 
as $y\rightarrow 0$, which
can be viewed as an extension of 
 Theorem \ref{thsiguna} to the case of complex inputs.
One can deduce similar extensions of 
Theorems
 \ref{thsigunat1}--\ref{thcircsing}.
The resulting estimates are a little weaker
than in Section \ref{scgrm},
being overly pessimistic;
actually
random 
complex matrices are a little better conditioned 
than random real matrices (see  
 \cite{E88}, \cite{ES05}, \cite{CD05} and our Table \ref{tab01}).

{\bf Acknowledgement:}
Our research has been supported by NSF Grant CCF-1116736 and
PSC CUNY Awards  63153--0041 and 64512--0042.
 We are also very grateful to Professor Dr. Siegfried M. Rump 
and Mr. Jesse Wolf for helpful comments.





\end{document}